\documentclass[11pt]{amsart}

\usepackage[margin=1in]{geometry}

\usepackage{setspace}
\usepackage{subfiles}  

\usepackage{amstext}
\usepackage{amsthm}
\usepackage{amssymb}
\usepackage[colorlinks=true, linkcolor=blue, citecolor=green, urlcolor=cyan]{hyperref}
\newcommand{\nocolorref}[2]{\hypersetup{linkcolor=black}\hyperref[#1]{#2}\hypersetup{linkcolor=blue}} 

\usepackage{tikz}
\usetikzlibrary{mindmap}
\usepackage{subfigure}

\usepackage{ulem}

\usepackage[linesnumbered,ruled]{algorithm2e}

\newtheorem{theor}{Theorem}[section]
\newtheorem*{theorem*}{Theorem}
\newtheorem*{theoremA}{Theorem~A}

\newtheorem{assump}[theor]{Assumption}
\newtheorem{lemma}[theor]{Lemma}
\newtheorem*{lemma*}{Lemma}
\newtheorem{cor}[theor]{Corollary}
\newtheorem{defi}[theor]{Definition}
\newtheorem{prop}[theor]{Proposition}
\newtheorem{fact}[theor]{Fact}

\theoremstyle{definition}
\newtheorem{rem}[theor]{Remark}

\theoremstyle{plain}

\newcommand{\R}{\mathbb{R}}

\def\Prob{{\mathbb P}}

\usepackage{xcolor}
\definecolor{b}{HTML}{4472c4}
\definecolor{o}{HTML}{ED7D31}
\definecolor{g}{HTML}{70ad47}
\definecolor{t}{RGB}{40,154,150}

\def\R{{\mathbb R}}

\def\Prob{{\mathbb P}}

\def\Vol{{\rm Vol}}

\def\deg{{\rm deg}}

\def\diam{\operatorname{diam}}

\date{\today}

\def\cal{\mathcal}

\newcommand{\pd}{{\rm d}}

\newcommand{\rp}{{\rm p}}

\newcommand{\gc}[1]{{#1}_{\rm gc}} 

\newcommand{\rb}{r_{\mathrm{bump}}}

\definecolor{pumpkin}{RGB}{255,117,24} 

\newcommand{\bbE}{\mathbb{E}}
\newcommand{\bbP}{\mathbb{P}}
\newcommand{\bbR}{\mathbb{R}}

\newcommand{\bfV}{\mathbf{V}}

\newcommand{\rmp}{\mathrm{p}}
\newcommand{\rmrp}{\mathrm{r}_{\mathrm{p}}}

\newcommand{\os}{\curlyvee} 

\newcommand{\step}[1]{\medskip  \noindent \underline{#1.\,}}

\newcommand{\dm}[1]{\varpi_{#1}} 

\newcommand{\rinj}{r_{\rm inj}} 
\newcommand{\ang}[4]{ \measuredangle^{#1}({#2}^{#3}_{#4}) } 
\newcommand{\D}[2]{|{#1}\,{#2}|} 

\newcommand{\tv}[1]{\nocolorref{eq: def_tv}{\hat #1}} 

\newcommand{\MAssump}{Let $(M,g)$ be a Riemannian manifold whose sectional curvatures are uniformly bounded in absolute value by a constant $\kappa > 0$.\,}

\newcommand{\rM}{\min\{ \rinj(M), \dm{\kappa}\}}
\newcommand{\rMrp}{\min\{ \rinj(M), \dm{\kappa}, \rmrp\}}

\newcommand{\radrmp}{{\rm r}_{\rmp}} 

\newcommand{\lrmp}{{\ell}_{\rmp}} 

\newcommand{\cn}[2]{\nocolorref{eq:def_normalized_number_of_common_neighbors}{{\rm N}_{#1}(#2)}} 

\newcommand{\acn}[2]{\nocolorref{eq:def_absolute_number_of_common_neighbors}{{\rmp}_{#1}(#2)}} 

\newcommand{\Ept}[1]{\nocolorref{eq:def_Ept}{{\cal E}_{\tr{pt}}(#1)}} 
\newcommand{\Ecn}[2]{\nocolorref{eq:def_Ecn}{{\cal E}_{\tr{cn}}(#1,#2)}} 
\newcommand{\Enavi}[2]{\nocolorref{eq:def_Enavi}{{\cal E}_{\tr{navi}}(#1,#2)}} 
\newcommand{\Eortho}{\nocolorref{eq: event-orthogonal-ortho}{{\cal E}_{\perp}}}
\newcommand{\Eclu}{\nocolorref{eq:def_Eclu}{{\cal E}_{\tr{clu}}}}

\newcommand{\Oclu}[1]{\nocolorref{eq:def_Oclu}{\Omega_{\tr{clu}}(#1)}}
\newcommand{\Oortho}[1]{\nocolorref{eq:def_Oortho}{\Omega_{\perp}(#1)}}
\newcommand{\Onet}[1]{\nocolorref{eq:def-Onet-t}{\Omega_{\tr{net}}(#1)}}

\newcommand{\Err}{\mathrm{Err}}

\newcommand{\rG}{{\rm r}_G}

\newcommand{\B}[2]{\nocolorref{eq: def_Br}{B_{#1}(#2)}} 

\newcommand{\Edge}[2]{\nocolorref{eq:def_edge_set}{E(#1,#2)}} 

\newcommand{\st}{{\rm t}} 

\newcommand{\fe}[1]{\nocolorref{eq: def_fe}{\varepsilon_{#1}}} 

\newcommand{\axis}{\mathsf{ax}} 

\renewcommand{\sp}{\mathsf{s}_n} 

\newcommand{\ipcn}[2]{\nocolorref{eq:def_ipcn}{\langle #1,\,#2\rangle_{\tr{cn}}}} 

\newcommand{\tr}[1]{\text{\tiny{\rm #1}}} 
\newcommand{\tref}[1]{\text{\tiny{\ref{#1}}}} 

\author{Han Huang, Pakawut Jiradilok, and Elchanan Mossel}
\date{\today}

\address{Department of Mathematics, University of Missouri, Columbia, MO 65203}
\email[H.~Huang]{hhuang@missouri.edu}

\address{Science Division, Mahidol University International College, Nakhon Pathom 73170, Thailand}
\email[P.~Jiradilok]{pakawut.jir@mahidol.edu}

\address{Department of Mathematics, Massachusetts Institute of Technology, Cambridge, MA 02139}
\email[E.~Mossel]{elmos@mit.edu}

\title{Reconstructing Riemannian Metrics From Random Geometric Graphs}

\begin{document}
\maketitle

\begin{abstract}
Random geometric graphs are random graph models defined on metric measure spaces. A random geometric graph is generated by first sampling points from a metric space and then connecting each pair of sampled points independently with a probability that depends on their distance.

In recent work of Huang, Jiradilok, and Mossel~\cite{HJM24}, the authors study the problem of reconstructing an embedded manifold form a random geometric graph sampled from the manifold, where edge probabilities depend monotonically on the Euclidean distance between the embedded points. They show that, under mild regularity assumptions on the manifold, the sampling measure, and the connection probability function, it is possible to recover the pairwise Euclidean distances of the embedded sampled points up to a vanishing error as the number of vertices grows.

In this work we consider a similar and arguably more natural problem where the metric is the Riemannian metric on the manifold. Again points are sampled from the manifold and a random graph is generated where the connection probability is monotone in the Riemannian distance. 
Perhaps surprisingly we obtain stronger results in this setup.
 Unlike the previous work that only considered dense graph we provide reconstruction algorithms from sparse graphs with average degree $n^{1/2}{\rm polylog}(n)$, where $n$ denotes the number of vertices. Our algorithm is also a more efficient algorithm for distance reconstruction with improved error bounds. The running times of the algorithm is 
 $O(n^2\,{\rm polylog}(n))$ which up to polylog factor matches the size of the input graph. 
 Our distance error also nearly matches the volumetric lower bounds for distance estimation. 
\end{abstract}


\section{Introduction}

Random graphs provide a powerful framework for modeling complex networks.
The classical Erdős--Rényi model \(G(n,p)\), in which edges appear independently with probability \(p\), remains a cornerstone of modern probabilistic combinatorics, offering analytic clarity for studying connectivity, typical distances, and phase transitions.
Comprehensive references include~\cite{bollobas2011random,janson2011random} and the survey~\cite{frieze2015introduction}.

To model more complex and realistic graphs it is natural to consider edge probabilities that depend on spatial or geometric proximity.
Random geometric graphs~(RGGs) model this phenomenon by associating latent positions to vertices and defining adjacency as a function of distance.
In the classical \emph{hard-disc} model, two vertices \(v,w\) are connected whenever \(\mathrm{dist}(X_v,X_w) \le r\).
The \emph{soft} or random-connection model generalizes this by assigning connection probabilities \(\rmp(\mathrm{dist}(X_v,X_w))\) that decay with distance, reflecting fading or noisy interactions as observed for example in wireless and spatial networks~\cite{Pen03,dettmann2016random}.

An exciting new research direction in the study of RGGs concerns the extent to which one can infer latent geometric or probabilistic structure from the observed connectivity pattern.
Earlier work established detectability thresholds distinguishing geometric graphs from Erdős--Rényi graphs in various regimes~\cite{BDEM16,LMSY22,BBN20}.
Other studies have explored nonparametric inference for translation-invariant RGGs on spheres, aiming to recover the connection kernel \(p(\langle x,y\rangle)\) from a single graph~\cite{ADC19,EMP22,DDC23}.
A more recent work~\cite{HJM24} considered the ambitious task of reconstructing the Euclidean distance between sampled points in a manifold given a random geometric graph where the edge probabilities monotonically depend on the Euclidean distance between points. 
There, it was shown that under mild regularity assumptions on a submanifold \(M \subset \mathbb{R}^N\), its sampling measure \(\mu\), and the distance--probability function \(\rmp\), both intrinsic and extrinsic distances between latent points can be recovered up to vanishing error.
This result established that geometric information is, in principle, encoded in the combinatorial structure of random geometric graphs.

In the work of~\cite{HJM24}, the authors consider $M$ to be a smooth, compact, connected $d$ dimensional submanifold of $\R^N$ without boundary, equipped with a probability measure $\mu$ satisfying the following lower \emph{Ahlfors regularity} condition (see Assumption \ref{def: mu}).
The distance-probability function \(\rmp:[0,\mathrm{diam}(M)]\to[0,1]\) is assumed to be non-increasing and \emph{bi-Lipschitz continuous}.
The random graph \(G\) is generated as follows.
Sample \(X_1,\dots,X_n \in M\) independently from \(\mu\).
Each vertex \(i\) corresponds to the latent point \(X_i\), and for each distinct pair \(i,j\), include an edge with probability \(\rmp(\mathrm{dist}(X_i,X_j))\), where \(\mathrm{dist}(\cdot,\cdot)\) denotes the Euclidean distance in the ambient space \(\mathbb{R}^N\). 

The main algorithm of~\cite{HJM24} assumes that the distance--probability function 
\(\rmp\), the intrinsic dimension \(d\), curvature bounds of \(M\), and a quantitative 
non-self-intersection bound are known. 
Given only the observed graph \(G\) and this auxiliary information, the algorithm 
produces estimates \(\mathrm{d}(i,j)\) for the Euclidean distances between latent points. 
Its theoretical guarantees are summarized below.

\begin{theor}[\cite{HJM24}, Theorem~1.1]
The \emph{Algorithm HJM24} runs in time \(O(n^3)\) and, with probability at least 
\(1 - n^{-\omega(1)}\), produces estimates satisfying
\[
    \big|\,\mathrm{d}(i,j) - \mathrm{dist}(X_i,X_j)\,\big|
    \;\le\;
    C n^{-\varepsilon/d},
\]
for all vertex pairs \(i,j\), where \(\varepsilon \in (0,1)\) is an absolute constant 
and \(C\) depends only on geometric properties of \(M\) 
(curvature and diameter bounds, separation parameter, and regularity parameters of \(\mu\)) 
as well as on the bi-Lipschitz constants of \(\rmp\).
\end{theor}

\subsection{Our Contributions and comparison to prior work.}
In this paper, we consider a different and arguably more natural variant of the question studied in~\cite{HJM24}.
Here, we consider the case where the latent space is a compact Riemannian manifold \((M,g)\) without boundary, equipped with a probability measure \(\mu\) satisfying lower Ahlfors regularity as above. Crucially the probabilities of edges are now determined by a connection probability function $\rmp$ applied to the intrinsic distance between the points on the manifold. This is arguably a better model, as for example, on Earth it is more natural to measure distances by the distance needed to travel rather by the Euclidean distance in $3$ dimensions.   
Our results also relax the condition on the connection function \(\rmp\) to be Lipschitz continuous, but only bi-Lipschitz on a small interval \([0,\mathrm{r}_{\rmp}]\) near zero (see Definition~\ref{def: distance-probability}).  We introduce a sparsity parameter \(\sp \in (0,1]\) into the model, allowing each edge to appear independently with probability \(\sp\,\rmp(\D{X_i}{X_j})\), where 
\begin{align*}
    \D{X_i}{X_j} := \text{the Riemannian distance between $X_i$ and $X_j$ on $M$.}
\end{align*}
Informally, our main theorem is stated below; see Theorem~\ref{thm:main} for a formal statement. 

\begin{theoremA} [Informal]
\label{the:mainInformal}
    \smallskip
\noindent 
There exists an \(O(n^2\,\mathrm{polylog}(n))\)-time algorithm that, given a random geometric graph 
\(
    G \sim G(n, M, g, \mu, \rmp, \sp),
\)
with sparsity parameter
\[
    1 \ge \sp = \omega\!\left(\frac{\log n}{\sqrt{n}}\right),
\]
$\rmp$, and some bounds on the geometric properties of \((M,g)\), the regularity of \(\mu\), 
produces an estimator of all pairwise distances. With high probability, all vertex pairs \(v,w\),
\[
    \big|\,\mathrm{d}(v,w) - \mathrm{dist}(X_v,X_w)\,\big|
    \;\le\;
    C \!\left(\frac{\log^2 n}{\sp\,n}\right)^{\!\frac{1}{d+1}},
\]
where \(C\) is a constant depending only on geometric properties of \(M\) 
(sectional curvature bounds, injectivity radius, diameter), 
the regularity parameters of \(\mu\), and the bi-Lipschitz properties of \(\rmp\).
\end{theoremA}

To further complement our result, we also establish a lower bound showing that no algorithm can achieve a uniform distance error of order smaller than \((1/n)^{6/(d+2)}/\log(n)\) with high probability; see Theorem~\ref{theor: lower-bound-informal} below. Compared with Theorem~A, the lower bound is different from the upper bound only in the exponent of \(n\) by a factor of approximately \(6/(d+2)\) versus \(1/(d+2)\), up to logarithmic factors. 

In comparison to the previous work~\cite{HJM24} we would like to comment that: 

\begin{itemize}
    \item \underline{Riemannian distance vs.\ Euclidean embedded distance.}
    The model we study here is more geometric as the distances depend only on the Riemannian distance rather than on the Euclidean distances which depend on the specifics of the embedding. 

    \item \underline{Sparse connectivity regime.}
    We introduce a sparsity parameter \(\sp \in (0,1]\) so that each edge appears independently with probability 
    \(\sp\,\rmp(\D{X_i}{X_j})\).
    Our analysis covers the regime
    \[
        \sp = \omega\!\left(\frac{\log n}{\sqrt{n}}\right),
    \]
    which allows for substantially sparser graphs than previously studied.

    \item \underline{Sharper quantitative guarantees to near volumetric limits in dense settings.}
    We derive explicit error bounds that nearly match the volumetric lower bounds for distance estimation with \(n\) samples: With $n$ points from a $d$-dimensional manifold, it cannot form an $\varepsilon$-net for $\varepsilon$ smaller than order $n^{-\frac{1}{d}}$  via simple volume arguments. Though $n^{-1/d}$ is not necessarily a lower bound for distance reconstruction error, it does provide a natural benchmark for achievable accuracy.

    \item \underline{Improved algorithmic efficiency to near running time of reading the input graph.}
    We design an efficient reconstruction algorithm that runs in 
    \(O(n^2\,\mathrm{polylog}(n))\) time, improving upon the earlier cubic-time procedure.
    The method is only polylogarithmically slower than reading the input graph itself.
       
\end{itemize}

\subsection{Formal Definitions and Statements}

We now present the formal definitions, standing assumptions, and the main results.

\begin{assump}[Manifold]
\label{assump: manifold}
We assume that $(M,g)$ is a compact, connected, smooth Riemannian manifold of dimension $d \ge 2$. 
Let $K(\sigma)$ denote the sectional curvature at a $2$-plane $\sigma \subset TM$, and define
\[
    \kappa := \sup_{\sigma} |K(\sigma)|\,,
\]
which is finite by compactness of $M$. 
Let $\rinj(M)$ denote the injectivity radius of $M$, which is also positive due to compactness.
\end{assump}

\begin{assump}[Distance--probability function]
\label{def: distance-probability} 
Let $\rmp:[0,\infty) \to [0,1]$ be a non-increasing function. 
We assume there exist positive constants $L_{\rmp}, \ell_\rmp, \rmrp > 0$ such that
\begin{itemize}
    \item for all $a,b \in [0,\rmrp]$, \quad $|\rmp(a) - \rmp(b)| \ge \ell_\rmp |a-b|$, and
    \item for all $a,b \in [0,\mathrm{diam}(M)]$, \quad $|\rmp(a) - \rmp(b)| \le L_\rmp |a-b|$.
\end{itemize}
That is, $\rmp$ is bi-Lipschitz on $[0,\rmrp]$ and globally Lipschitz on $[0,\mathrm{diam}(M)]$.
We refer to $\rmp$ as the \emph{distance--probability} function.
\end{assump}

\begin{assump}[Sampling measure]
\label{def: mu} 
Let $\mu$ be a probability measure on $M$, and define
\[
\mu_{\min}(r) := \min_{p \in M} \mu(\B{r}{p}).
\]
We assume that $\mu$ satisfies a \emph{lower Ahlfors regularity} condition: there exist constants ${\rm r}_\mu > 0$ and $c_\mu > 0$ such that
\begin{equation}\label{eq: mu-ball}
\mu_{\min}(r) \ge c_\mu \, r^d, \qquad \forall r \in [0, {\rm r}_\mu].
\end{equation}
\end{assump}

\begin{defi}[Random geometric graph]
\label{def: random-geometric-graph}
Given a triple $(M,\mu,\rmp)$, a vertex set $\bfV$ with $|\bfV|=n$, and a sparsity parameter $\sp \in (0,1]$, 
define the random geometric graph $G(\bfV,M,g,\mu,\rmp,\sp)$ as follows. 
Sample i.i.d.\ points $\{X_v\}_{v\in\bfV}$ from $\mu$, 
and let $\{{\cal U}_{i,j}\}_{i\ne j}$ be i.i.d.\ ${\rm Uniform}(0,1)$ random variables. 
The edge set of $G$ is
\[
    E(G) = \big\{ \{i,j\} \in \tbinom{\bfV}{2} : {\cal U}_{i,j} < \sp\,\rmp(\D{X_i}{X_j}) \big\}.
\]
For brevity we write $G(n,M,g,\mu,\rmp,\sp)$ when $|\bfV|=n$.
\end{defi}

\begin{defi}[Accessible radius]
\label{def: radius_G}
Define
\[
    \rG := \frac{1}{16} \min\bigg\{ \rmrp,\, \rinj(M),\, \frac{1}{\sqrt{\kappa}},\, {\rm r}_\mu \bigg\}.
\]
This constant depends only on $(M,g)$, $\rmp$, and $\mu$.
Within distance $\rG$, the manifold is approximately Euclidean, $\mu$ satisfies lower Ahlfors regularity, 
and $\rmp$ behaves bi-Lipschitzly.
\end{defi}

\begin{theor}
\label{thm:main}
Assume $\sp^2 n = \omega(\log^2 n)$. 
There exists an algorithm (Algorithm~\ref{alg:distance}) running in time $O(n^2 \mathrm{polylog}(n))$ 
such that the following holds.

For any $(M,g,\mu)$ satisfying Assumptions~\ref{assump: manifold} and~\ref{def: mu}, 
and any distance--probability function $\rmp$ satisfying Assumption~\ref{def: distance-probability}, 
the algorithm takes as input
\[
    G \sim G(\bfV,M,g,\mu,\rmp,\sp), \quad 
    \rmp, \quad \sp, \quad \text{a lower bound on both } \rG \quad \text{and } c_\mu,
\]
and with probability $1 - n^{-\omega(1)}$ it outputs estimates $d(w,v)$ for all $w,v \in \bfV$ satisfying
\[
    \big|\,\D{X_w}{X_v} - {\rm d}(w,v)\,\big| 
    \;\le\; 
    C \left( \frac{\log^2 n}{\sp n}\right)^{\!1/(d+2)} ,
\]
for all $w,v \in \bfV$, 
where $C$ depends only on $d$, $L_\rmp$, $\ell_\rmp$, $c_\mu$, $\rG$, and $\mathrm{diam}(M)$.
\end{theor}

\begin{theor}[Lower bound]
\label{theor: lower-bound-informal}
For every $n$, there exist two $d$-dimensional Riemannian manifolds 
$(M_1,g_1,\mu_1)$ and $(M_2,g_2,\mu_2)$ 
together with distance--probability functions $\rmp$ so that the following holds. First,  
the corresponding $\rG$, $c_\mu$, and diameters bounded by universal constants (independent of $n$). Second, there exists a coupling $\pi$ of the random geometric graphs
\[
    G(n,M_1,g_1,\mu_1,\rmp,1)
    \quad\text{and}\quad
    G(n,M_2,g_2,\mu_2,\rmp,1),
\]
such that with probability $1-o_n(1)$ for $(G_1,G_2)\sim\pi$, we have $G_1=G_2$ and
\[
    \big|
    \D{X_u^{(1)}}{X_v^{(1)}}_{M_1} - 
    \D{X_u^{(2)}}{X_v^{(2)}}_{M_2}
    \big|
    \;\ge\; \frac{1}{\log n} \, n^{-6/(d+2)}
\]
for some $\{u,v\}\subseteq V(G_1)=V(G_2)$. 
Here $X_u^{(i)}$ denotes the latent point in $M_i$ corresponding to vertex $u$ in $G_i$.
\end{theor}

\subsection{Ideas from the Proof} 
We sketch some of the main ideas of the proof. 
For the discussion below, we assume that $\mathbf{V}$ denotes the vertex set of the graph. We begin by outlining the main ideas and the overall construction of the algorithm, illustrating the comparison from the approach in~\cite{HJM24}. The emphasis in this first part is on the algorithmic aspects and we treat the embedded manifold and the Riemannian manifold settings uniformly by subsuming both under the umbrella term “locally Euclidean.” In the second part, we turn to the geometric aspects and discuss how to extend the approach to the Riemannian manifold setting.

\subsubsection{Clusters and cluster edge density} \phantom{.}

\step{Cluster edge density approach}  
A natural idea that was already used in~\cite{HJM24} is that {\em given} a cluster of nearby points, we can use concentration to estimate the distance from (points in) this cluster to a point that is far away. We first elaborate on this idea and will then discuss how to find clusters. 

For any vertex $v$ and a small parameter $\xi$, suppose we are given the vertex subset $U_v$
such that, for every $v' \in U_v$, 
$$\D{X_v}{X_{v'}} \le \xi.$$
Then, for any other vertex $w$, we estimate $\D{X_v}{X_w}$ by the edge density from $U_v$ to $w$, as follows.

\begin{figure}[h!]
  \centering
  \includegraphics[width=0.4\textwidth]{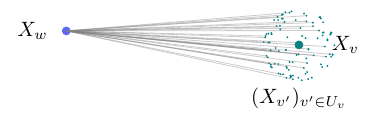}
 
  \caption{ For $\D{X_v}{X_{w}}$ in the bi-Lipschitz regime of $\rmp$, edge density from cluster $U_v$ to vertex $w$ reveals $\rmp(\D{X_v}{X_w})$.}
  \label{fig:xv_neighbors}
\end{figure}

First, for any $v' \in U_v$, since $\D{X_v}{X_{v'}} \le \xi$, the Lipschitz continuity of $\rmp$ implies
\[
\big|\rmp(\D{X_v}{X_w}) - \rmp(\D{X_{v'}}{X_w})\big| \;\le\; L_\rmp \,\xi .
\]

Second, conditional on the latent locations $X_{U_v}=\{X_{v'}\}_{v'\in U_v}$ and $X_w$, 
the edges between $U_v$ and $w$ appear independently with probability 
$\sp \,\rmp(\D{X_{v'}}{X_w})$. 
Hence the number of such edges is a sum of independent Bernoulli random variables with mean
\[
\sp \sum_{v' \in U_v} \rmp(\D{X_{v'}}{X_w}) .
\]
By a Chernoff bound, with probability at least $1 - n^{-\omega(1)}$,
\[
\Big|\, (\#\text{ of edges from }U_v\text{ to }w) 
- \sp \sum_{v' \in U_v} \rmp(\D{X_{v'}}{X_w}) \,\Big|
\;\le\; O\!\big(\sqrt{\sp |U_v| \log n}\big).
\]

Normalizing by $\sp |U_v|$ yields the estimator
\[
\rmp(\D{X_v}{X_w})
\;=\;
\cn{U_v}{w} \;+\; O\!\big(\xi + \fe{U_v}\big),
\]
where
\[
\cn{U_v}{w}
:= \frac{\# \text{ of edges from } U_v \text{ to } w}{\sp |U_v|}
\quad\text{and}\quad
\fe{U_v} := \sqrt{\frac{\log n}{\sp |U_v|}} .
\]
If $\D{X_v}{X_w}$ is in the bi-Lipschitz regime of $\rmp$, then we can invert $\rmp$ to estimate $\D{X_v}{X_w}$ up to error
$O(\xi + \fe{U_v})$. Thus, the error in estimating $\D{X_v}{X_w}$ has two components: 
the \emph{fluctuation error} $\fe{U_v}$ from edge randomness, 
and the \emph{radius error} $\xi$ from replacing $\rmp(\D{X_{v'}}{X_w})$ 
by $\rmp(\D{X_v}{X_w})$ for $v' \in U_v$.


\step{Balancing the two errors}  
Suppose $\xi$ is small relative to $\rG$ defined in~\eqref{def: radius_G}, 
and take the “optimal” set
\[
U_v := \{\, v' \neq v : \D{X_v}{X_{v'}} \le \xi \,\}.
\]
If $\mu$ is the uniform measure on $M$, then
\[
|U_v| \;=\; \frac{{\rm Vol}(\B{\xi}{X_v})}{{\rm Vol}(M)}\, n \;\gtrsim\; \xi^d n .
\]
For $X_v$ and $X_w$ not too far apart (in the Bi-Lipschitz regime of $\rmp$), we have
\[
\sp \sum_{v' \in U_v} \rmp(\D{X_{v'}}{X_w}) \;\simeq\; \sp |U_v| ,
\]
so the fluctuation scale is $\sqrt{\sp |U_v|} \simeq \sqrt{\sp \xi^d n}$. 
Balancing the two errors gives
\[
\xi \;\simeq\; (\sp n)^{-\frac{1}{d+2}} .
\]
Consequently, this cluster edge density method estimates $\rmp(\D{X_v}{X_w})$ 
up to error $(\sp n)^{-\frac{1}{d+2}}$, which is the rate achieved in Theorem~\ref{thm:main}.
In other words, if for each vertex $v$ we can find a set $U_v$ of size at least $\xi^d n$ within radius $\xi$, with the choice of $\xi$ above, then we can estimate all pairwise distances up to error $O((\sp n)^{-\frac{1}{d+2}})$. For vertices $v,w$ that are far apart, we can rely on the fact that the geodesic distance is a path metric, meaning that we can estimate $\D{X_v}{X_w}$ by summing up the estimated distances along a path of short hops between $v$ and $w$.

\subsubsection{Finding clusters via common-neighbor counts in \cite{HJM24}} \phantom{.}
For a subset $W$ of vertices, call it $(\varepsilon,m)$-dense if, for every $i \in W$, the set
\[
W_i := \{\, i' \in W : \D{X_i}{X_{i'}} \le \varepsilon \,\}
\]
has size at least $m$.

Consider the feature map
\[
i \in W \;\mapsto\; \rmp_i : M \to [0,1], 
\qquad \rmp_i(x) := \rmp(\D{X_i}{x}),
\]
and define the inner product
\[
\langle p_i, p_j \rangle_W := \mathbb{E}_{Z \sim \mu}\!\left[ \rmp_i(Z)\,\rmp_j(Z) \right].
\]
Let $V_{\tr{cn}} \subset V$ be another subset of vertices disjoint from $W$.  
For distinct $i,j \in W$, the number of common neighbors of $i$ and $j$ in $V_{\tr{cn}}$ is a sum of $|V_{\tr{cn}}|$ independent Bernoulli random variables with mean $\sp^2\langle p_i, p_j \rangle_W$, 
conditional on $X_i,X_j$. 
Thus we estimate $\langle p_i, p_j \rangle_W$ via the normalized common-neighbor count
\[
\cn{V_{\tr{cn}}}{i,j} := \frac{\# \text{ of common neighbors between } i \text{ and } j}{\sp^2 |V_{\tr{cn}}|},
\]
except that the self inner product $\langle p_i, p_i \rangle_W$ is not directly observable.

\begin{figure}[h!]
\centering
\begin{minipage}{0.5\textwidth}
    \includegraphics[width=\linewidth]{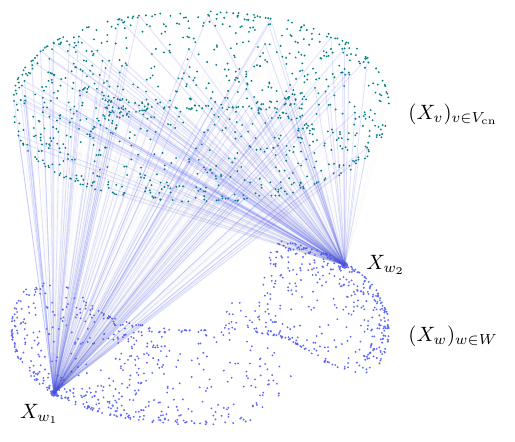}
\end{minipage}\hfill
\begin{minipage}{0.4\textwidth}
    \textbf{Figure 2}:
        In the diagram, two set of latent points $X_W =(x_w)_{w \in W}$ and $X_{V_{\tr{cn}}} = (x_v)_{v \in V_{\tr{cn}}}$ are shown (with $M$ being a torus).  The (normalized) common neighbor count between any two vertices $w_1,w_2 \in W$ in $V_{\tr{cn}}$ estimates the inner product $\langle p_i, p_j \rangle_W$. When $W$ is $(\varepsilon,m)$-dense with small $\varepsilon$, the closer the common neighbor count $\cn{V_{\tr{cn}}}{w_1,w_2}$ is to $\max_{w_1',w_2' \in {W \choose 2}}\cn{V_{\tr{cn}}}{w_1',w_2'}$, the smaller the latent distance $\D{X_{w_1}}{X_{w_2}}$ is.
\end{minipage}
\end{figure}

The maximizer of $\langle p_i, p_j \rangle_W$ over $i,j \in W$ occurs at $i = j$; 
this holds even without a manifold structure for $M$ or regularity of $\mu$. 
To conclude the converse implication
\[
\langle p_i, p_j \rangle_W \text{ near its maximum} 
\;\Rightarrow\; 
\D{X_i}{X_j} \text{ is small},
\]
we do rely heavily on the manifold structure of $M$, essentially that $M$ is locally Euclidean, together with the lower Ahlfors $d$-regularity of $\mu$. This idea goes back to~\cite{HJM24}; here we adapt it to the Riemannian manifold setting. 
Building on it, we construct an algorithm (Algorithm~\ref{alg:cluster}) that, 
given any $(\varepsilon,m)$-dense set $W$ with $\varepsilon$ sufficiently small, 
extracts a cluster pair $(U,u)$ with $u \in U \subseteq W$ such that
\begin{align}
    \label{eq: cluster_goal}
|U| \gtrsim m
\quad\text{and}\quad
\D{X_u}{X_{u'}} \le O(\sqrt{\varepsilon}) \quad \text{for all } u' \in U,
\end{align}
by seeking $(U,u)$ where each pair $(u,u')$ has a common-neighbor count close to the maximum 
over all pairs in $W$. We remark that there is a loss in \eqref{eq: cluster_goal}. When $\mu$ is uniform on $M$, the standard volumetric bound suggests any subset $ W \subseteq {\bf V}$ can at most be $(\varepsilon, \Theta(\varepsilon^dn))$-dense for $\varepsilon \gtrsim n^{-1/d}$. Thus the smallest non-empty cluster we can hope to extract from this approach has radius at least $\sqrt{n^{-1/d}} = n^{-1/(2d)}$, which is already suboptimal for the balanced choice of $\xi$ in the cluster edge density method.

\step{Iterative cluster extraction}  
The procedure above does not control which vertex is chosen as the cluster center $u$. 
To obtain clusters around different centers, we rerun the algorithm on suitably chosen subsets $W$ 
that remain $(\varepsilon,m)$-dense.

Initialize with $W_1=\bfV$ to obtain $(U_1,u_1)$. 
Suppose we have $(U_1,u_1),\ldots,(U_k,u_k)$. 
Using the cluster edge density method, we can estimate the distances from every vertex to each $u_i$, $i\in[k]$. 
We then filter $\bfV$ to a smaller set $W_{k+1}$ consisting of vertices that satisfy certain approximate distance constraints to the existing centers. 
If $W_{k+1}$ is still $(\varepsilon,m)$-dense, 
we run the cluster extractor again to obtain a new cluster $(U_{k+1},u_{k+1})$ whose center automatically satisfies the new constraints (since all vertices in $W_{k+1}$ do). 
Iterating this yields enough clusters so that every $v \in \bfV$ is close to some cluster center $u$. 


\step{Orthogonal cluster frame}
The main technical effort in~\cite{HJM24} is to prove that each filtered set $W_{k+1}$ remains $(\varepsilon,m)$-dense,
which requires careful control of error accumulation in the distance estimates. 
An intuitive visualization is that $W_{k+1}$ typically corresponds to the intersection of several fuzzy annuli centered at the existing cluster centers $u_1,\ldots,u_k$, 
where ``fuzzy'' reflects the uncertainty caused by estimation errors. 
The ``corners'' and ``boundaries'' of this intersection are the regions most prone to losing density, 
making it challenging to ensure the $(\varepsilon,m)$-density property. A key intermediate step is that, for a given cluster $(U_k,u_k)$, 
one also extracts several \emph{orthogonal clusters} around $u_k$, 
denoted $(U_{k,1},u_{k,1}),\ldots,(U_{k,d},u_{k,d})$, 
such that the directions from $X_{u_k}$ to $X_{u_{k,s}}$ for $s\in[d]$ 
are approximately orthogonal in the tangent space $T_{X_{u_k}}M$. 
With these $d+1$ clusters, for any other vertex $v$ whose latent position $X_v$ is not too far from $X_{u_k}$, 
the distances $\D{X_v}{X_{u_{k,s}}}$ for $s\in[d]$ together with $\D{X_v}{X_{u_k}}$ 
--- or, from observation of the graph , the edge densities from $U_{k,s}$ to $v$ --- 
can be interpreted as defining a local coordinate system around $X_{u_k}$. 
This local coordinate system helps control the geometry of the filtered set $W_{k+1}$ (all vertices satisfying certain distance constraints to $u_k$ and $u_{k,s}$ for $s\in[d]$),
preserve its density, and enable the extraction of a new cluster near $v$ in the next iteration, which is not too far from $u_k$ in the latent distance.


\step{Limitations in~\cite{HJM24}}
One of the main limitations of this approach is that the radius of each extracted cluster scales as $O(\sqrt{\varepsilon})$, rather than $O(\varepsilon)$, as mentioned in~\eqref{eq: cluster_goal}. 
From the earlier error-balancing analysis, this enlarged cluster radius leads to a suboptimal choice of $\xi$ in the cluster edge density method. 

Moreover, to preserve probabilistic independence during the iterative filtering and cluster extraction process, 
the method in~\cite{HJM24} partitions the vertex set into $n^{c}$ disjoint subsets (for some constant $c>0$), 
each of size $n^{1-c}$, and runs the cluster extraction algorithm on each subset whenever a new cluster is sought. 
This partitioning results in a degradation of the final error bound, since each subset contains only $n^{1-c}$ vertices instead of the full $n$ is weakening its $(\varepsilon,m)$-density property, which leads to a looser final error bound.

Finally, the dominant computational cost arises from computing the number of common neighbors for every pair of vertices in each working subset, 
which leads to an overall running time of $O(n^3)$. 

\subsubsection{Our approach}
Our approach builds upon the framework of~\cite{HJM24}, but stems from a simple yet powerful observation that was not use in~\cite{HJM24}.   
Suppose we have a cluster $(U,u)$ such that
\begin{align*}
    \D{X_u}{X_{u'}} \le \eta, \quad \forall u' \in U,
    \qquad \text{and} \qquad |U| \gtrsim \eta^d n,
\end{align*}
for some radius parameter $\eta>0$. 
Now consider two vertices $v,w$ whose latent positions $X_v,X_w$ are both not too far from $X_u$. 
Then the \emph{difference} between the expected numbers of edges from $U$ to $v$ and from $U$ to $w$ is
\begin{align*}
    &\mathbb{E}\big[\# \text{ of edges from } U \text{ to } v\big] 
    - \mathbb{E}\big[\# \text{ of edges from } U \text{ to } w\big] \\
    &= \sp \sum_{u' \in U} \big( \rmp(\D{X_{u'}}{X_v}) - \rmp(\D{X_{u'}}{X_w}) \big).
\end{align*}
By the Lipschitz continuity of $\rmp$ and the triangle inequality, we have
\[
\big| \rmp(\D{X_{u'}}{X_v}) - \rmp(\D{X_{u'}}{X_w}) \big|
= O\big(\D{X_v}{X_w}\big).
\]
In fact, when $X_v$ and $X_w$ are relatively close but much farther from $X_u$ than $\eta$, we can obtain a finer estimate:
\[
 c_\theta \D{X_v}{X_w} \le 
\big( \rmp(\D{X_{u'}}{X_v}) - \rmp(\D{X_{u'}}{X_w}) \big)
\le  C_\theta \D{X_v}{X_w}\,,
\]
where $\theta$ denotes the angle $\angle X_u X_v X_w$ in the geodesic triangle formed by $X_u, X_v, X_w$ centered at $X_v$, and $c_\theta,C_\theta>0$ are constants depending on $\theta$ and $\rmp$. 
Consequently,
\[
 c_\theta \D{X_v}{X_w} - O(\fe{U}) \le 
|N_{U_u}{v} - N_{U_u}{w}| \le C_\theta \D{X_v}{X_w} + O(\fe{U}).
\]
Importantly, unlike the cluster edge density distance estimator, this method introduces \emph{no radius error}.  
In summary, while clusters with larger radii may be suboptimal for estimating the absolute distance $\D{X_u}{X_w}$ compared with the balanced-error cluster, 
they are significantly more effective at detecting \emph{changes} in distance between two nearby points $v$ and $w$.  
As an informal analogy, a coarse measurement scale may be unsuitable for measuring absolute lengths precisely, 
but can be more sensitive in detecting small relative changes.

\begin{figure}[h!]
\centering
\begin{minipage}{0.45\textwidth}
    \includegraphics[width=\linewidth]{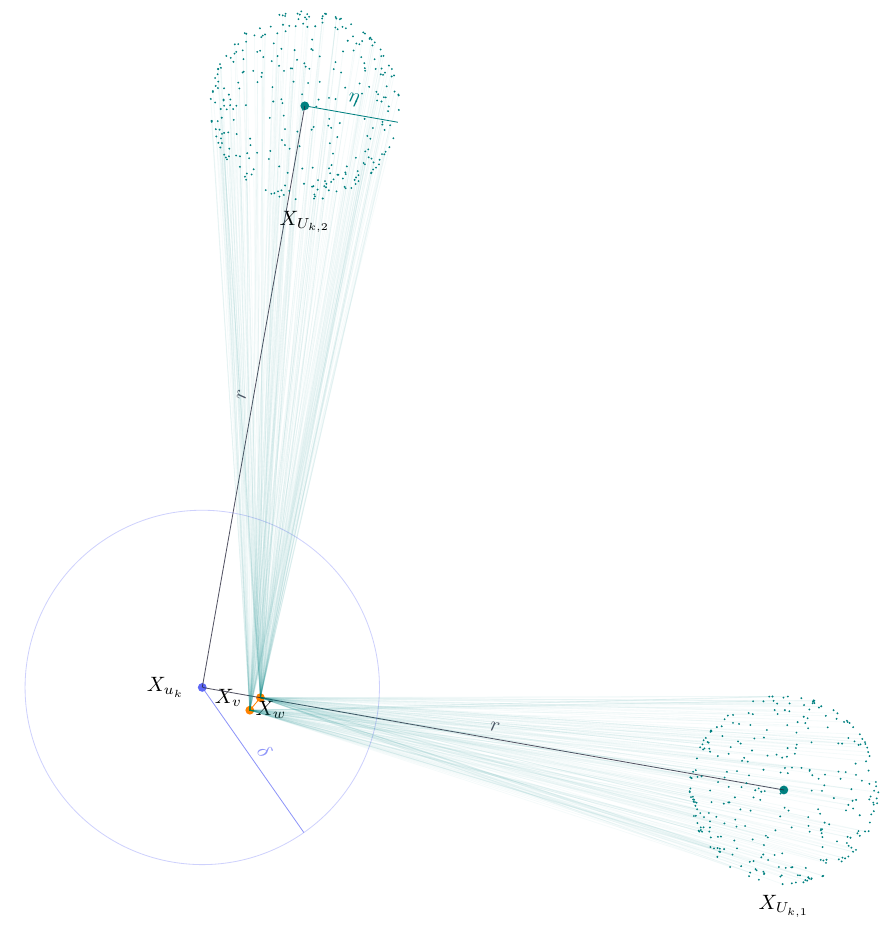}
\end{minipage}\hfill
\begin{minipage}{0.4\textwidth}
    \textbf{Figure 3}:
    In the diagram, when two vertices \(v, w\) both lie within distance at most \( \delta \) from the cluster center \( X_{u_k} \)—which can be estimated using the cluster edge density method—the difference in their common-neighbor counts \( \cn{U_{k,s}}{v}\)  and \(\cn{U_{k,s}}{w} \), with respect to the orthogonal cluster \( U_{k,s} \), provides an estimate of \( \D{X_v}{X_w} \) up to a multiplicative error (through the mapping \( \phi \) defined in~\eqref{eq: intro_phi}). This holds whenever \( \D{X_v}{X_w} \gg \sqrt{\frac{\log n}{\sp n}} \), the dimension-free fluctuation scale. In particular, \( \D{X_v}{X_w} \) can be negligible relative to the cluster radius \( \eta \). (Here, $\eta < \delta < r < \rG$ are some parameters will be of order constants in the paper.) 
\end{minipage}
\end{figure}

\step{Multiplicative error distance estimation down to dimension-free fluctuation}
Suppose in addition we have found $d$ clusters $(U_{u,s},u_s)$, $s\in[d]$, around $u$, each satisfying the same size and radius conditions as $(U,u)$, 
and such that the directions from $X_u$ to $X_{u_s}$ for $s\in[d]$ are approximately orthogonal and have roughly equal lengths in the tangent space $T_{X_u}M$.  
Then, for any vertices $v$ and $w$ whose latent positions $X_v, X_w$ are both not too far from $X_u$, 
define the vector
\begin{align}
    \label{eq: intro_phi}
\phi := \Big(\rmp^{-1}(N_{U_{u,s}}{v}) - \rmp^{-1}(N_{U_{u,s}}{w})\Big)_{s\in[d]}.
\end{align}
This vector serves as an approximate distance estimator satisfying
\[
c\,\D{X_v}{X_w} - O(\fe{U}) 
\;\le\; \|\phi\|_2 
\;\le\; C\,\D{X_v}{X_w} + O(\fe{U}),
\]
for some constants $c,C>0$. In this setting, it is advantageous to choose $\eta$ to be a fixed constant fraction of the Accessible radius $\rG$ defined in~\eqref{def: radius_G}, 
which is independent of $n$. 
Then the fluctuation error satisfies
\[
\fe{U} = O\!\left(\sqrt{\frac{\log n}{\sp n}}\right),
\]
which is \emph{dimension-free}.  
In other words, by leveraging differences in the common-neighbor counts across multiple clusters, 
as long as 
\[
 \sqrt{\frac{\log n}{\sp n}} \ll \D{X_v}{X_w} \lesssim \rG,
\]
we can estimate $\D{X_v}{X_w}$ up to a multiplicative error. 
Here, the upper bound ensures that all involved points remain within a locally Euclidean region of $M$, and the proper invertibility of $\rmp$ in this regime.  
Because the resulting error is multiplicative, this method outperforms the cluster edge density method whenever $\D{X_v}{X_w}$ is small comparing to $(\log n / (\sp n))^{1/(d+2)}$.

\step{Two-stage approach}
Our distance estimation algorithm (Algorithm~\ref{alg:distance}) combines the two ideas above into a unified two-stage framework.

\begin{figure}[h!]
\centering
\begin{minipage}{0.5\textwidth}
    \includegraphics[width=\linewidth]{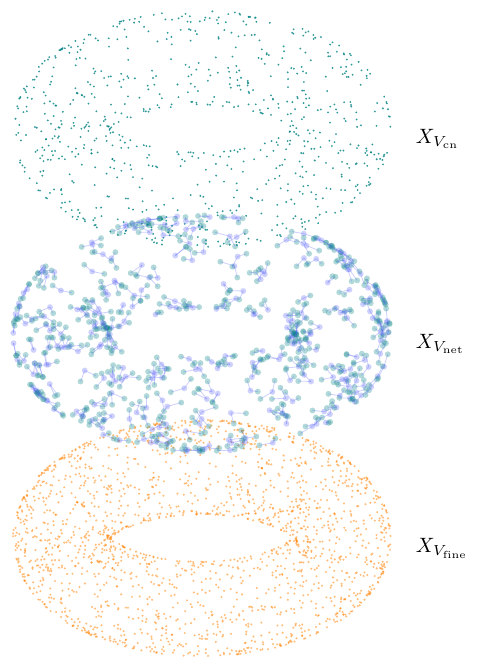}
\end{minipage}\hfill
\begin{minipage}{0.4\textwidth}
    \textbf{Figure 4}:
    In the diagram, the vertex set $\mathbf{V}$ is partitioned into three disjoint subsets: 
$V_{\tr{cn}}$, $V_{\tr{net}}$, and $V_{\tr{fine}}$. 
The first two subsets are used in \textbf{Stage 1}, where a constant-radius cluster net 
is constructed over the latent positions of vertices in $V_{\tr{net}}$ 
based on common-neighbor counts within $V_{\tr{cn}}$. 
In \textbf{Stage 2}, for each vertex $v \in V_{\tr{fine}}$, 
a nearby cluster center $u_k$ from Stage~1 is identified (via $\cn{U_k}{v}$). 
By leveraging the orthogonal clusters around $u_k$, 
one can estimate the distance between $X_v$ and any other vertex’s latent position $X_w$ 
using the \emph{multiple-cluster-difference} method. 
This allows us to collect all $w \in V_{\tr{fine}}$ such that $\D{X_v}{X_w} \le \xi$ to form a cluster around $v$ with balanced-error radius $\xi$.
Once all balanced-error clusters have been constructed via edge counting, 
the distance estimates are obtained through edge counting from fine clusters to vertices. 
    \end{minipage}
\end{figure}


\paragraph{Stage 1: Cluster with constant radius construction.} 
Following the general structure of~\cite{HJM24}, we iteratively extract clusters $(U_k,u_k)$ for $k=1,2,\ldots$ until they form a $\delta$-net of $M$, 
where $\delta$ is a small constant fraction of $\rG$. 
Around each center $u_k$, we further obtain $d$ orthogonal clusters $(U_{k,s},u_{k,s})$ for $s\in[d]$. 
The key differences from~\cite{HJM24} are threefold:
\begin{itemize}
    \item Our cluster radius and $\delta$ are constants independent of $n$, rather than vanishing quantities $o_n(1)$ as in~\cite{HJM24}.
    \item We carefully handle the probabilistic dependence among clusters extracted in different iterations, 
    allowing the entire algorithm to operate using only four disjoint vertex partitions, instead of $n^c$ partitions (for some constant $c>0$) as required in~\cite{HJM24}. 
    This modification removes \emph{one of the two sources of error degradation~\cite{HJM24}, namely the need to breaking the vertex set into  $n^c$ batches}.
    \item Since we extract only constant-radius clusters, for each vertex pair $(v,w)$ we need only compute their number of common neighbors over a subset of vertices $V_{\tr{cn}}$ with size $|V_{\tr{cn}}| = \Theta(\sp^{-2}\log^2(n))$, which is of order $\log^2(n)$ when $\sp=1$, comparing with $|V_{\tr{cn}}| = n^{1-c}$ in~\cite{HJM24}. 
    This yields a substantial computational improvement, reducing the overall runtime from $O(n^3)$ in \cite{HJM24} to $O(n^2\,\mathrm{polylog}(n))$ in our algorithm.
\end{itemize}

\paragraph{Stage 2: Local refinement.}
In the second stage, we use a single new batch of vertices. 
We may assume that for every vertex $v$, there exists a cluster center $u_k$ such that $\D{X_v}{X_{u_k}} \le \delta$. 
We then employ the multiple-cluster-difference method to extract a refined cluster centered precisely at $v$, 
with radius 
\[
\eta = O\!\left(\left(\frac{\log n}{\sp n}\right)^{\!1/(d+2)}\right) 
\gg \sqrt{\frac{\log n}{\sp n}}.
\]
While smaller clusters can be extracted around each vertex, the key is to select $\eta$ that balances the radius error and the fluctuation error in the cluster edge density method. Once such clusters are obtained, we can estimate $\D{X_v}{X_w}$ for every pair of vertices $v,w$ up to an additive error of order 
$O\!\left((\log^2 n / (\sp n))^{1/(d+2)}\right)$, 
whenever $\D{X_v}{X_w}$ lies within the bi-Lipschitz regime of $\rmp$. 
Finally, we apply a shortest-path computation (e.g., Dijkstra’s algorithm) to recover all pairwise distance estimates.
With the second stage, we eliminate \emph{the other source of error degradation present in~\cite{HJM24}, which is not capable to extracting optimal error-balanced clusters}.

\step{Our Algorithm in the Embedded case}
      We believe that our algorithms and analysis carry over also to the embedded case analyzed in~\cite{HJM24}.
    Thus we believe that the statement of our main theorem also holds for random graphs sampled according to the Euclidean distance between embedded points, where we can recover these Euclidean distances under the same assumptions and same error bounds as in our main theorem, with one caveat on the Bi-Lipschitz radius of $\rmp$. We leave the verification of this to future work.

\subsubsection{Triangle comparison, a bridge between the local structure of $M$ and the flat Euclidean space.}
Although the Nash embedding theorem ensures that any Riemannian manifold can be embedded into a higher-dimensional Euclidean space, our setting differs fundamentally from that of~\cite{HJM24}. In their framework, edge probabilities are determined by Euclidean distances in the ambient space, whereas in ours they depend on geodesic distances within the manifold. Consequently, even after embedding, Euclidean and geodesic distances may diverge significantly, which necessitates working directly with the intrinsic Riemannian structure of $M$.

Consider a Riemannian manifold $(M,g)$. For each point $p\in M$, the \emph{exponential map at $p$} is the smooth map
\[
   \exp_p : T_pM \longrightarrow M ,
\]
which sends a vector $v\in T_pM$ to the point reached, at parameter time~$1$, by the geodesic starting at $p$ with initial velocity $v$.
The \emph{injectivity radius} of $M$ is defined by  
\begin{align}
    \label{eq: defi_rinj}
    \rinj(M) \;:=\; \sup\Bigl\{\,r>0 : \exp_p\vert_{\B{r}{\vec{0}}} \text{ is a diffeomorphism for every } p\in M \Bigr\}\, .
\end{align}
\noindent
For convenience, whenever $p\in M$ and $x\in \B{\rinj(M)}{p} \subseteq M$, we set
\begin{align}
    \label{eq: def_tv}
   \tv{x} \;=\; (\tv{x})_p \;:=\; \exp_p^{-1}(x)\in T_pM .
\end{align}
\noindent
A standard consequence of the Gauss lemma (see Proposition~\ref{prop:exp-ball} and see e.g.~\cite{CE08})   is that, for every $p\in M$, 
the exponential map $\exp_p$ maps the open ball $\B{\rinj(M)}{0} \subseteq T_pM$ diffeomorphically onto the open ball $\B{\rinj(M)}{p} \subseteq M$.
And for each $x\in \B{\rinj(M)}{p}$, the curve
\[
   t \mapsto \exp_p \bigl(t\,\tv{x}\bigr), \qquad t\in[0,1],
\]
is the unique minimizing geodesic from $p$ to $x$. In particular, the Riemannian (geodesic) distance between $p$ and $x$, denoted by $\D{p}{x}$, satisfies
\[
   \D{p}{x} \;=\; \bigl\|\tv{x}\bigr\| .
\]

Consider two points $x,y\in \B{\rinj(M)}{p}$. When $M$ is flat, the geodesic distance satisfies $$\D{x}{y} =\|\tv{x}-\tv{y}\|\,.$$ A basic recurring task in the paper is to estimate how much $M$ around $p$ deviates from $T_pM$, and one of them is to estimate how much the geodesic distance $\D{x}{y}$ can differ from its Euclidean proxy under the sectional curvature bound; let us say that there exists $\kappa>0$ such that the sectional curvature $K$ of $M$ at any 2–plane $\sigma$ in $T_pM$ for every $p \in M$ satisfies
\begin{align*}
    -\kappa \;\le\; K(\sigma) \;\le\; \kappa\,.
\end{align*}
(The existence of such a $\kappa$ is guaranteed when $M$ is compact.)
Rather than expanding the exponential map and tracking higher–order error terms, we use triangle–comparison techniques. Specifically, we compare the geodesic triangle with its counterpart in the model space $M_\kappa$ of constant curvature $\kappa$. The Rauch Comparison Theorem and Toponogov’s Triangle Comparison Theorem (see e.g.~\cite{CE08}) allow us to translate the curvature bound into precise inequalities for side lengths and angles.  
\begin{defi}
For any $\kappa \in \mathbb{R}$, let $M_\kappa^d$ be the $d$-dimensional, simply–connected, complete Riemannian manifold whose sectional curvature is identically $\kappa$. The three standard cases are summarized below:


\[
\renewcommand{\arraystretch}{1.3}
\begin{array}{c|l|l|l}
\text{Curvature }\kappa &
\text{Name of model space }M_\kappa^{d} &
\text{Concrete realization} &
\text{Diameter }\varpi_\kappa \\ \hline
\kappa>0 &
\text{$d$-sphere of radius }1/\sqrt{\kappa} &
\frac{1}{\sqrt{\kappa}}\mathbb{S}^{d} \subseteq \mathbb{R}^{\,d+1} &
\frac{\pi}{\sqrt{\kappa}} \\[6pt]
\kappa=0 &
\text{Euclidean $d$-space} &
\mathbb{R}^{d} &
\infty \\[6pt]
\kappa<0 &
\text{Hyperbolic $d$-space of curvature }\kappa &
\mathbb{H}^{d}_{\kappa} &
\infty
\end{array}
\]
\end{defi}


In what follows, when the dimension $d$ is clear from the context, we may write $M_\kappa$ instead of $M^d_\kappa$, omitting the superscript $d$.


\begin{defi}
Let $\kappa \in \bbR$, $\theta\in[0,\pi]$, and $a,b>0$ with $a+b < \varpi_\kappa$. 
We write
\begin{align*}
    \os^\kappa(\theta;a,b)
\end{align*}
for the length of the side opposite the angle~$\theta$ in a geodesic triangle whose two adjacent sides forming the angle $\theta$ have lengths~$a$ and~$b$ in the model space $M^d_\kappa$.
\end{defi}
We remark that the assumption $a+b<\varpi_\kappa$ guarantees that the corresponding geodesic triangle is well-defined when $\kappa > 0$ in the positively curved model space $M^d_\kappa$.
Now suppose that $\kappa > 0$, and suppose that three points $p,x,y\in M$ satisfy 
\[
   \D{p}{x}+\D{p}{y}<\varpi_\kappa.
\]
Let $\theta$ denote the angle between $\tv{x}$ and $\tv{y}$ in $T_pM$. Then by the Rauch Comparison Theorem (see e.g.~\cite{CE08}), the length of the side opposite to $\theta$ obeys
\begin{align}
   \label{eq: intro-opposite-side}
   \os^{\kappa} \bigl(\theta;\D{p}{x},\D{p}{y}\bigr)
   \;\le\; \D{x}{y}
   \;\le\;
   \os^{-\kappa} \bigl(\theta;\D{p}{x},\D{p}{y}\bigr).
\end{align}

Because the opposite side length in each model space can be written explicitly by the spherical and hyperbolic laws of cosines, each bound in~\eqref{eq: intro-opposite-side} may be compared with the flat case
\[
   \os^{0} \bigl(\theta;\D{p}{x},\D{p}{y}\bigr)
   \;=\;
   \|\tv{x}-\tv{y}\|,
\]
which yields estimates of the form
\[
   \D{x}{y}
   \;=\;
   \|\tv{x}-\tv{y}\|
   \;+\;
   \text{higher-order terms in } \max\big\{\D{p}{x},\D{p}{y}\big\}.
\]


In many arguments we must also compare the \emph{angle at~$p$} with its analogue in a model space.  
To avoid repeatedly referring to the angle between the tangent vectors $\tv{x},\tv{y}\in T_pM$, we adopt a notation similar to that of \cite{AKP24}.

\begin{defi}
  Let $p,x,y\in M$ with $x,y\in \B{\rinj(M)/2}{p} \setminus \{p\}$.  
  We define
  \[
     \ang{}{p}{x}{y}\;:=\;\angle \bigl(\tv{x},\tv{y}\bigr) = \arccos\left(\frac{\langle \tv{x},\tv{y}\rangle}{\|\tv{x}\| \, \|\tv{y}\|}\right),
  \]
  the angle at~$p$ of the geodesic triangle $\triangle pxy$ in~$M$.
  Given $\kappa>0$ and assuming $\D{p}{x}+\D{p}{y}<\varpi_\kappa$, we further set
  \[
     \ang{\kappa}{p}{x}{y}
     \quad\text{and}\quad
     \ang{-\kappa}{p}{x}{y}
  \]
  to be, respectively, the angles at~$\tilde p$ of the comparison triangle
  $\triangle\tilde p\tilde x\tilde y \subseteq M_\kappa$ and of the corresponding triangle in $M_{-\kappa}$—that is, the unique triangle in the model space whose side lengths equal
  $\D{p}{x}$, $\D{p}{y}$, and $\D{x}{y}$.
\end{defi}

The requirement $x,y\in \B{\rinj(M)/2}{p}$ guarantees that $\D{p}{x}$, $\D{p}{y}$, and $\D{x}{y}$ are all less than~$\rinj(M)$, so the distance-minimizing geodesics are unique; hence there is no ambiguity in the definition of the corresponding geodesic triangles in~$M$. As the dual inequality to~\eqref{eq: intro-opposite-side}, the angles obey
\begin{align}
    \label{eq: intro-angle}
    \ang{-\kappa}{p}{x}{y}
    \;\le\;
    \ang{}{p}{x}{y}
    \;\le\;
    \ang{\kappa}{p}{x}{y}.
\end{align}
Moreover, because the spherical and hyperbolic laws of cosines express
$\ang{\pm\kappa}{p}{x}{y}$ explicitly in terms of the side lengths, one can compare them with the flat angle
\[
   \ang{0}{p}{x}{y},
\]
i.e., the ordinary Euclidean angle of the triangle with the same side lengths.

Taken together, the comparison bounds~\eqref{eq: intro-opposite-side} and~\eqref{eq: intro-angle} form a quantitative bridge between the local geometry of $M$ and that of its tangent plane of interest. They let us compare the two spaces on a genuinely local—rather than merely infinitesimal—scale. Equipped with these estimates we may treat $M$ as \emph{almost flat} in a sufficiently small neighborhood of~$p$, which in turn allows us to use many arguments from linear algebra.

\step{Outline of the paper}
The section structure of the paper is as follows.
\begin{itemize}
    \item Section~2 introduces some basic notions and events that will be used throughout the paper.
    \item Section~3 presents several geometric estimates on Riemannian manifolds that will be used in the analysis.
    \item Section~4 we state and prove the results on cluster extraction through common-neighbor counts, which form the basis of our distance estimation algorithm.
    \item Section~5 presents an intermediate algorithms for constructing constant-radius cluster nets and for local refinement, respectively. 
    \item Section~6 presents the multi-clusters-difference distance estimator and analyzes its performance.
    \item Section~7-8 illustrate our two stages of distance estimation algorithm. 
    \item Section~9 concludes the proof of our main theorem by combining the results from previous sections.
    \item Appendix~A collects several standard results on Riemannian geometry used in the paper.
    \item Appendix~B shows the proof of the lower bound result (Theorem~\ref{theor: lower-bound-informal}).
\end{itemize}


 \section*{Acknowledgments}
 The authors were partially supported by Vannevar Bush Faculty Fellowship ONR-N00014-20-1-2826 and by Investigator award (622132).
 E.M. was also partially supported by  ARO MURI W911NF1910217, 
 NSF DMS-2031883,  
 and NSF award CCF 1918421. 
 
 We would like to thank Tristan Collins, Tirasan Khandhawit, and Greg Parker for helpful discussions.

\section{Basic Notions and Events}
Without loss of generality, throughout the paper we assume ${\bf V}$ is the vertex set of the random geometric graph $G \sim G(n,M,\mu,\rmp,\sp)$. For simplicity, we assume $|{\bf V}|$ is a constant multiple of $n$, instead of $n$, so that we can partition ${\bf V}$ into several disjoint subsets of equal size when needed. In particular, we assume $|{\bf V}| = 4n$ throughout the paper, which does not affect the statement of our main theorem (Theorem~\ref{thm:main}). Further, throughout the paper, we assume that the sparsity parameter $\sp$ satisfies 
$$
    \sp^2 n = \omega(\log^2 n)\,.
$$

 Below we define some basic events that will be used in the proofs of our main results and show that they hold with high probability.   

\begin{defi}[Normalized common neighbors]\label{def: normalized_number_of_common_neighbors}
Let $U,V\subseteq\bfV$. We define the normalized number of common neighbors of $V$ in $U$ by
\begin{equation}\label{eq:def_normalized_number_of_common_neighbors}
\cn{U}{V}
:= \frac{\big|\{\,u\in U:\ v\sim u\ \text{for all }v\in V\,\}\big|}
{(\sp )^{|V|}\,|U|}\,.
\end{equation}
For a singleton $V=\{v\}$ we write $\cn{U}{v}$, and for a pair $V=\{v,w\}$ we write $\cn{U}{v,w}$.
Further, for a singleton $v$, we also define the (conditional) expectation
\[
\acn{U}{v}:=\frac{1}{|U|}\sum_{u\in U}\rmp\!\big(\D{X_v}{X_u}\big),
\]
so that $\bbE\big[\cn{U}{v}\,\big|\,X_v,X_U\big]=\acn{U}{v}$.
Accordingly, we also define a notion of fluctuation error $\fe{U}$ that depends on $|U|$:
\begin{align}
    \label{eq: def_fe}
    \fe{U} := \frac{\log n}{\sqrt{\sp \,|U|}},
\quad \mbox{ and } 
\quad 
\fe{m} := \frac{\log n}{\sqrt{\sp m} }\,. \quad (m \ge 1).
\end{align}
\end{defi}
\begin{lemma}\label{lem: navi} (Edge fluctuation event)
Let $U,V\subseteq\bfV$ be disjoint with $|V|\le n$. Define
\begin{align}
    \label{eq:def_Enavi}
\Enavi{U}{V}
:=\bigg\{\forall v\in V:\ \big|\cn{U}{v}-\acn{U}{v}\big|\le \fe{U}\bigg\}.
\end{align}
Then for all fixed realizations $X_V=x_V$, $X_U=x_U$,
\[
\Pr\!\big(\Enavi{U}{V}\,^c\ \big|\ X_V=x_V,\ X_U=x_U\big)\ \le\ n^{-\omega(1)}.
\]
\end{lemma}

The above bound follow simply from the union bound and Chernoff bound: 
\begin{lemma}[Chernoff Bound]
Let $Y = \sum_{i=1}^n Y_i$ be the sum of independent Bernoulli random variables with \(\mu = \mathbb{E}[Y]\).
Then for any $0 < t < \mu$,
\[
\Pr\!\left[\,|Y - \mu| \;\ge\; t\,\right] 
\;\le\; 2 \exp\!\left(- \tfrac{t^2}{3\mu}\right).
\]
\end{lemma}

\begin{proof}
Fix $x_V,x_U$. For each $v\in V$,
\[
\cn{U}{v}=\frac{1}{\sp |U|}\sum_{u\in U}\mathbf 1\!\left\{{\cal U}_{v,u}<\sp \,\rmp(\D{X_v}{X_u})\right\},
\]
is a sum of independent Bernoulli random variables conditioned on $X_v,X_U$, normalized by $\Delta:=\sp |U|$ and with mean $\acn{U}{v}$ . 
By Chernoff bound, we have  
\begin{align*}
    \Pr\big(|\cn{U}{v}-\acn{U}{v}|\ge \fe{U}\,|\,X_v,X_U\big)
\le & 
    \Pr\big(|\cn{U}{v}-\acn{U}{v}|\ge \fe{U}\acn{U}{v}\,|\,X_v,X_U\big)\\
\le &
    2\exp\!\left(-\frac{\fe{U}^2\Delta^2}{3\Delta\acn{U}{v}}\right)
= 
    2\exp(-\log^2n/3)\,,
\end{align*}
where we used the fact that $\acn{U}{v} \le 1$. Next, applying the union bound over $v\in V$ yields the claim.
\end{proof}

\begin{defi}(Cluster event) 
    \label{def: event_clu}
    For $u \in U \subseteq {\bf V}$ and $\eta>0$, we define the event  
    \begin{align}
    \label{eq:def_Eclu}
        \Eclu(\eta,U,u) := \Big\{ \forall u' \in U \,,\,\D{X_{u'}}{X_u}  \le \eta \Big\}.    
    \end{align}
\end{defi}

So, if $\Eclu(\eta,U,u)$ and $\Enavi{U}{v}$ hold, then the normalized number of common neighbors $\cn{U}{v}$ can be used to estimate the latent distance $\D{X_v}{X_u}$, with error coming from both the cluster radius $\eta$ and the fluctuation error $\fe{U}$. 

\begin{defi}(Size Calibration margin)
    \label{def: calibration-margin}
    For a set $U\subseteq\bfV$, we say that the size of $U$ satisfies the calibration margin with parameter $\eta>0$ if
    \begin{align}
    \label{eq: eta_fe}
    \fe{U} \le \ell_\rmp \eta\,.
    \end{align}
\end{defi}
\begin{lemma}[Calibration toolkit: single, annulus]
\label{lem:calibration-toolkit}
Let $\rho \in U\subseteq\bfV$. Suppose the size of $U$ satisfies the calibration margin \eqref{eq: eta_fe} with parameter $0<\eta < \rG$, and the events 
$\Eclu(\eta,U,\rho)$ and $\Enavi{U}{w}$ hold. Then:
\begin{enumerate}\itemsep0.4em
\item[\textup{(i)}] Single–cluster calibration. 
\begin{align}
    \label{eq: single_cluster_calibration_neighbors}
\big|\cn{U}{w}-\rmp(\D{X_w}{X_\rho})\big|\ \le\ L_\rmp\,\eta+\fe{U}
\le 2L_\rmp\,\eta. 
\end{align}
If moreover either $\cn{U}{w}\ge \rmp(\rG)$ or $\D{X_w}{X_\rho}\le \rG$, then
\[
\big|\D{X_w}{X_\rho}-\rmp^{-1}(\cn{U}{w})\big|\ \le\ \eta + \ell_\rmp^{-1}\fe{U} \le\ 2\eta.
\]
If $\cn{U}{w}<\rmp(\rG)$, then
$\D{X_w}{X_\rho}>\rG-2\eta$.
\item[\textup{(ii)}] Annulus–threshold equivalence. Fix $r\le \rG$ and $\alpha\in[0,r]$.
\begin{align}
    \label{eq: calibration_annulus_threshold}
 \cn{U}{w}\le \rmp(r-\alpha)
&\ \Rightarrow\ 
\D{X_w}{X_\rho}\ \ge\ r-\alpha-2\eta\,\\
\nonumber
 \cn{U}{w} \ge \rmp(r+\alpha)
&\ \Rightarrow\ 
\D{X_w}{X_\rho} \ \le\ r+ \alpha+2\eta\,\\
\nonumber
\big|\D{X_w}{X_\rho}-r\big|\le \alpha\ 
&\ \Rightarrow\
\rmp(r+\alpha+2\eta)\ \le\ \cn{U}{w}\ \le\ \rmp(r-\alpha-2\eta).
\end{align}
\end{enumerate}
\end{lemma}
\begin{proof}
\step{Preparation}
Throughout we use that $\rmp$ is nonincreasing, $L_\rmp$–Lipschitz on $[0,\diam(M)]$, and bi-Lipschitz on $[0,16\rG]$ with lower constant $\ell_\rmp$.
If $\Eclu(\eta,U,\rho)$ holds, then for any $w$,
\begin{equation}\label{eq:avg-bracket}
  \rmp\!\big(\D{X_w}{X_\rho}+\eta\big)
  \ \le\ \acn{U}{w}
  \ \le\ \rmp\!\big(\D{X_w}{X_\rho}-\eta\big).
\end{equation}
This is immediate from the triangle inequality and the monotonicity of $\rmp$.
If $\Enavi{U}{w}$ holds, then by definition
\begin{equation}\label{eq:navi}
    | \cn{U}{w} - \acn{U}{w} | \le \fe{U}.
\end{equation}

\step{Single–cluster calibration}
Combining \eqref{eq:avg-bracket} and \eqref{eq:navi} and using the $L_\rmp$–Lipschitz property, let $D = \D{X_w}{X_\rho}$. Then
\[
\big|\cn{U}{w}-\rmp(\D{X_w}{X_\rho})\big|
\ \le\ \fe{U}\ +\ \max\Big\{\big|\rmp(D+\eta)-\rmp(D)\big|,\ \big|\rmp(D-\eta)-\rmp(D)\big|\Big\}
\ \le\ \fe{U}+L_\rmp\,\eta.
\]
This proves the first inequality in \eqref{eq: single_cluster_calibration_neighbors}. Since $\fe{U}\le \ell_\rmp\eta\le L_\rmp\eta$, we also get
\(\big|\cn{U}{w}-\rmp(D)\big|\le 2L_\rmp\,\eta\).
Now consider the inversion bound. If $\cn{U}{w} \ge \rmp(\rG)$, then
$$
\rmp(D - \eta) \ge  \acn{U}{w} \ge \cn{U}{w} - \fe{U} \ge \rmp(\rG) - \fe{U}  \quad \Rightarrow \quad D  \le \rG + \eta + \ell_p^{-1}\fe{U} < 16\rG.
$$
On the other hand, if $D \le \rG$, then
$$
\cn{U}{w} \ge \acn{U}{w} - \fe{U} \ge \rmp(D + \eta) - \fe{U} \ge \rmp(\rG+\eta) - \fe{U} \quad \Rightarrow \quad \rmp(\cn{U}{w}) \le
\rG + \eta + \ell_p^{-1}\fe{U} < 16\rG. 
$$  
In either case, we have $D, \cn{U}{w} \in [0,16\rG]$, so the low-Lipschitz property applies:
\[
\big|D-\rmp^{-1}(\cn{U}{w})\big|
\ \le\ \ell_\rmp^{-1}\,\big|\rmp(D)-\cn{U}{w}\big|
\ \le\ \ell_\rmp^{-1}\fe{U}+\ell_\rmp^{-1}L_\rmp\eta
\ \le\ \ell_\rmp^{-1}\fe{U}+\eta
\ \le\ 2\eta.
\]
Finally, if $\cn{U}{w}<\rmp(\rG)$, a similar argument shows 
$\rmp(D+\eta) \le \rmp(\rG) + \fe{U}$ implying $D > \rG - \eta - \ell_\rmp^{-1}\fe{U} > \rG - 2\eta$. 

\step{Annulus–threshold equivalence}
Fix $r\le \rG$ and $\alpha\in[0,r]$.
Assume \(\cn{U}{w}\le \rmp(r-\alpha)\).
Then,
\begin{align*}
\rmp(D+\eta)\ \le \acn{U}{w} 
\le \cn{U}{w} + \fe{U}
\le \rmp(r-\alpha) + \fe{U}
\le \rmp(r-\alpha - \ell_\rmp^{-1}\fe{U})
\quad \Rightarrow\quad
D\ \ge\ r-\alpha - 2\eta\,. 
\end{align*}
Similarly, if \(\cn{U}{w}\ge \rmp(r+\alpha)\), then invoking the inequalities in the other direction leads to  
\begin{align*}
    D \le r+\alpha + 2\eta.  
\end{align*}
This proves the first implication in \eqref{eq: calibration_annulus_threshold}.
Now we assume \(|D-r|\le \alpha\). 
Then, 
\begin{align*}
    \cn{U}{w} \le \acn{U}{w} + \fe{U} \le \rmp(D-\eta) + \fe{U} 
    \le \rmp(r-\alpha - \eta) + \fe{U}
    \le \rmp(r-\alpha - 2\eta)\,,
\end{align*}
and similarly in the other direction. From \eqref{eq:avg-bracket},

\[
\rmp(D+\eta)\ \le\ \frac{1}{|U|}\sum_{u\in U}\rmp(D{X_w}{X_u})\ \le\ \rmp(D-\eta).
\]
Adding/subtracting \(\fe{U}\) via \eqref{eq:navi} and using monotonicity,
\[
\rmp(r+\alpha+\eta)-\fe{U}\ \le\ \cn{U}{w}\ \le\ \rmp(r-\alpha-\eta)+\fe{U}.
\]
This matches the claimed bounds.
\end{proof}

\begin{defi}
[Common neighbor as inner product]
\label{def: ipcn}
For points $x, y \in M$, we define the following inner product-like quantity that measures the common neighbor probability between $x$ and $y$
\begin{align}
    \label{eq:def_ipcn}
    \langle x, y \rangle := \mathbb{E}_{Z\sim \mu} \rmp(\D{x}{Z}) \rmp(\D{y}{Z})
    = \int_M \rmp(\D{x}{z}) \rmp(\D{y}{z}) \, {\rm d}\mu(z)
    \,, 
\end{align} 
where in particular, it is the inner product of the functions $\rmp(\D{x}{\cdot})$ and $\rmp(\D{y}{\cdot})$ in the Hilbert space $L^2(M, \mu)$. Moreover,  
$$
    \langle X_i, X_j \rangle 
$$
is precisely the probability that a third vertex $k$ is in the common neighbor of $i$ and $j$ (Definition \ref{def: random-geometric-graph}) conditioned on the latent positions $X_i$ and $X_j$.
\end{defi}

\begin{lemma} \label{lem: commonNeighborEvent}
   Suppose $V,U$ are two disjoint set of vertices in our random geometric graph model (Definition \ref{def: random-geometric-graph}) with $|V| \le n$.  
   Consider the event 
   \begin{align}
    \label{eq:def_Ecn}
    \Ecn{U}{V} := 
    \bigg\{ 
        \forall v,w \in V \,,
        \Big| 
         \cn{U}{v,w} 
         - \ipcn{X_v}{X_w} \Big|
         \le 
            \frac{\log(n)}{\sqrt{\sp^2|U|}}
    \bigg\}\,,
   \end{align}
   which is an event of ${\cal U}_{U,V}$ and the latent points $X_V,X_U$.
   If 
   $$
    \sp^2 |U| = \omega(\log(n))\,,
   $$
   then
   \begin{align*}
    \Pr\Big( \Ecn{U}{V}^c \, \Big\vert\, X_V =x_V \Big) = n^{-\omega(1)} \quad \forall x_V \in M^{|V|}\,.
   \end{align*}
\end{lemma}


\begin{proof}
    The proof is similar to that of Lemma~\ref{lem: navi}. It is an application of the Chernoff bound together with the union bound over ${|V| \choose 2}$ pairs of vertices in $V$.  
Let us fix any realization of the latent points $X_V = x_V$. 
    Fix $\{v,w\} \in {|V| \choose 2}$.
    Observe 
    \begin{align*}
        \cn{U}{v,w} 
    = 
        \frac{1}{\sp^2|U|} 
        \sum_{u \in U} 
        \underbrace{{\bf 1}\Big({\cal U}_{v,u} < \sp \rmp(\D{X_v}{X_u})\Big) 
                    {\bf 1}\Big({\cal U}_{v,u} < \sp \rmp(\D{X_v}{X_u})\Big)
        }_{:=Z_u}\,.
    \end{align*}
    Notice that $\{Z_u\}_{u \in U}$ are independent Bernoulli random variables with mean 
    $$
        \mathbb{E}_{X_u,{\cal U}_{u,\{v,w\}}} [Z_u \,\vert\, X_v, X_w]
    = 
        \mathbb{E}_{X_u}\Big[ \sp^2\rmp(\D{X_v}{X_u})
                       \rmp(\D{X_w}{X_u}) \,\Big\vert\, X_v, X_w \Big]
    = 
        \sp^2\ipcn{X_v}{X_w}\,,$$
    which implies that 
    $$ 
        \mathbb{E}_{X_u,{\cal U}_{U,\{v,w\}}} [\cn{U}{v,w} \,\vert\, X_v,X_w]
    = 
        \ipcn{X_v}{X_w}\,.
    $$
    The standard Chernoff bound gives
    \begin{align*}
       & \Pr\Big( \Big| \cn{U}{v,w} - \ipcn{X_v}{X_w}\Big| \ge \frac{\log(n)}{\sqrt{\sp^2|U|}}  \,\Big\vert\, X_V = x_V \Big)\\
    =& 
        \Pr\Big( \Big| \sum_{u \in U} Z_u - |U|\sp^2\ipcn{X_v}{X_w}\Big| \ge \log(n) \sqrt{\sp^2|U|} \,\Big\vert\, X_V = x_V \Big) \le  2\exp(-c\log^2(n))  = n^{-\omega(1)} \,,
    \end{align*}
    where we used the fact that 
    $\ipcn{X_v}{X_w} \le 1$. 
    Within the complement of the event described above, we have, after normalization, the event $\Ecn{U}{V}$ holds. 

\end{proof}

\begin{lemma} \label{lem: epsilonNetEvent}
Suppose $\{X_v\}_{v \in V}$ are i.i.d.\ random points sampled from a distribution $\mu$ on a manifold $M$
and $|V| = n$. The event 
\begin{align}
    \label{eq:def_Ept}
   \Ept{V} := 
    \bigg\{ 
        \big| \{X_i\}_{i \in [n]} \cap \B{\varepsilon}{p} \big| 
        \ge \mu_{\rm min}(\varepsilon/3) \frac{n}{2} 
        \quad \forall p \in M, 
        \Big(\frac{\log(n)}{n}\Big)^{1/d} \le \varepsilon \le r_{\mu}
    \bigg\}\,, 
\end{align}
holds with probability $1- n^{-\omega(1)}$.  
\end{lemma}
The above lemma follows from standard volumetric arguments and Chernoff bound; see e.g.~\cite{HJM24} for a similar proof.

\subsection{Additional Notation}
\begin{assump}[Parameter Configurations]
 In the algorithm, we will choose several distance parameters: $\eta, \delta, r$ corresponding to cluster radius, net radius, and orthonormal frame radius respectively. Further, we also have a cluster size parameter $m$ (a desired lower bound on the size of an extracted cluster). 
For any gap parameter $c \in [0,1)$ which might depends on $\ell_\rmp, L_\rmp, \mu_{\min}(\cdot)$, and $d$, but not on $n$,  we say the tuples $(r,\delta,\eta,m)$ satisfy the \emph{parameter configurations} with gap $c$ if 
\begin{align}
    \label{eq: parameter-configuration}
    r \le c \cdot \rG\,, \quad \delta = c\cdot r, \quad \eta = c\cdot \delta
    \quad \text{ and } \quad \fe{m} \le c \eta^2\,.  
\end{align}
\end{assump}
\section{Triangle comparison}
Here we will collect the main tools from Riemannian geometry that allow us to compare distances and angles in a Riemannian manifold with those in model spaces of constant curvature. The main result is the following Triangle Comparison Lemma, which is a direct consequence of the Rauch Comparison Theorem (see e.g.~\cite{CE08}). As most of the statements in this section are standard in Riemannian geometry, we will postpone them to the appendix. Further, in the end of this section, we also include a regularity lemma (Lemma~\ref{lem: regularity-geometry}) for measuring the change of edges count from a cluster to a single vertex when perturbing the latent position of the vertex slightly.

\begin{lemma}
\label{lem: tri_lem}
Let $\kappa > 0$ be a positive constant. Let \((M,g)\) be a compact $d$-dimensional Riemannian manifold whose sectional curvatures satisfy
\[
  -\kappa \;\le\; K(\sigma) \;\le\; \kappa,
\]
for every \(2\)-plane $\sigma$ in the tangent bundle $TM$.
Fix \(p\in M\) and set
\[
  r:=\min \bigl\{\rinj(M),\,\pi/\sqrt{\kappa}\bigr\}.
\]

Consider \(x,y\in \B{r/2}{p}\) and write $\theta:= \ang{}{p}{x}{y}$.
Then
\begin{align}
\label{eq: tri_lem_os}
  \os^{\kappa}(\theta;\D{p}{x},\D{p}{y})
  \;\le\;
  \D{x}{y}
  \;\le\;
  \os^{-\kappa}(\theta;\D{p}{x},\D{p}{y})
\end{align}
and 
\begin{align}
\label{eq: tri_lem_ang}
    \ang{-\kappa}{p}{x}{y}
    \;\le\;
    \ang{}{p}{x}{y}
    \;\le\;
    \ang{\kappa}{p}{x}{y}.
\end{align}
\end{lemma}

It is also worth recalling the \emph{laws of cosines} for model spaces of constant
curvature~$\kappa$, which relate the side lengths and an angle of a geodesic triangle.


\begin{fact}[The laws of cosines for model spaces, see e.g.~\cite{AKP24}]
Let $\kappa>0$ and consider a geodesic triangle in the spherical model space $M_\kappa$ with side lengths $a,b,c$ satisfying
  \[
    a,b,c \;<\;\frac12\dm{\kappa}\;=\;\frac{\pi}{2\sqrt{\kappa}}.
  \]
  If $\theta$ denotes the angle opposite the side of length~$c$, then
   \begin{align}
        \label{eq: spherical-cos-law}
       \cos(\sqrt{\kappa}c) = \cos(\sqrt{\kappa}a)\cos(\sqrt{\kappa}b) + \sin(\sqrt{\kappa}a)\sin(\sqrt{\kappa}b)\cos(\theta).
   \end{align}
   
For a geodesic triangle in the hyperbolic model space $M_{-\kappa}$ ($\kappa>0$) with side lengths $a,b,c$ and $\theta$ again the angle opposite the side of length~$c$, one has
   \begin{align}
    \label{eq: hyperbolic-cos-law}
       \cosh(\sqrt{\kappa}c) = \cosh(\sqrt{\kappa}a)\cosh(\sqrt{\kappa}b) - \sinh(\sqrt{\kappa}a)\sinh(\sqrt{\kappa}b)\cos(\theta).
   \end{align}
\end{fact}

Following from the Triangle Comparison Lemma~\ref{lem: tri_lem} and the laws of cosines above, we have a few lemmas regarding angles and distances which essentially arrive from expanding the laws of cosines in Taylor series.

\begin{lemma}
\label{lem: M-opposite-side}
\MAssump
Suppose $p,q,x$ be points in $M$ such that
$$0 \le \D{p}{x}, \D{p}{q} < \frac{1}{4} \min\{ \dm{\kappa}, \rinj(M) \}
\qquad \text{and} \qquad \D{p}{q} \le \frac{1}{4} \D{p}{x}\,.$$
Expressing 
\begin{align*}
    a:= \D{p}{x}\,, \quad b:= \D{p}{q}\,, \quad c:= \D{q}{x}\,,
    \quad \text{and} \qquad
    \theta := \ang{}{p}{x}{q}\,. 
\end{align*}
Then 
\begin{align*}
 -\frac{7}{6}\frac{b^2}{a} - \frac{\kappa b^2}{3} \cdot b|\cos(\theta)|
\le
    a - c -  b \cos(\theta)  
\le
    \pi \frac{b}{a} \cdot b |\cos(\theta)|\,.
\end{align*}
\end{lemma}

\begin{rem}
In the flat Euclidean plane, the ordinary law of cosines gives  
\[
c^{2}=a^{2}+b^{2}-2ab\cos\theta
\quad\Longrightarrow\quad
c = a - b\cos\theta + O(b^{2})\quad (b\to0,\;a\text{ fixed}).
\]
Hence \(a - c - b\cos\theta\) is the \emph{second-order remainder} in the Taylor expansion of \(c\) with respect to \(b\) at \(b=0\).  
Our goal is to show that, with the \emph{curvature correction}, how this second-order remainder is controlled.  
\end{rem}

\begin{lemma}\label{lem:spherical-angle}
Fix $\kappa>0$. Consider a geodesic triangle in $M_\kappa$ with
adjacent sides $a,b$ and angle $\theta$ between them.  
Assume
\[
      0<a,b<\tfrac14\,\dm{\kappa}
      \qquad\text{and}\qquad
      c:=\os^{\kappa}(\theta;a,b)\;\le\;\tfrac12\,a .
\]
Then
$$
    \theta \le 2 \frac{c}{a}\,.
$$ 
\end{lemma}

Given these lemmas are standard in Riemannian geometry and not to detract from the main flow of the paper, we postpone their proofs to Appendix~\ref{sec:geometry}. 
\begin{lemma}[Edge Counting Regularity]
\label{lem: regularity-geometry}    
\MAssump
Let $p, x, q \in M$ be three points in a Riemannian manifold $M$ with
$$
    \D{p}{x}, \D{p}{q} < \frac{1}{8} \rM
    \qquad \mbox{and} \qquad 
    16 \frac{\D{p}{q}}{\D{p}{x}} \le  
     |\cos(\ang{}{p}{q}{x})| \,.
$$

Suppose $x' \in M$ is a point satisfying 
\begin{align*}
16 \frac{\D{x}{x'}}{\D{p}{x}}  \le |\cos(\ang{}{p}{q}{x})|. 
\end{align*}
Then, the following holds:
\begin{align}
    \frac{1}{2} 
\le \frac{\D{p}{x'} - \D{q}{x'}}{\cos(\ang{}{p}{q}{x}) \D{p}{q}} 
\le 2\,.
\end{align}
Furthermore, with the additional assumption that 
\begin{align*}
    \D{p}{x} \le \frac{1}{2} \rmrp\,,
\end{align*}
we have 
\begin{align*}
    \frac{1}{2} \ell_\rmp 
\le 
    \frac{\rmp(\D{q}{x'}) - \rmp(\D{p}{x'})}{\cos(\ang{}{p}{q}{x}) \D{p}{q}}
\le  
    2L_\rmp\,.
\end{align*}

\end{lemma}

\begin{proof}
\step{Cosine value comparison}
We first invoke Lemma~\ref{lem: tri_lem} to compare the angle $\ang{}{p}{x}{x'}$ with the model angle $\ang{\kappa}{p}{x}{x'}$, and then Lemma~\ref{lem:spherical-angle} to obtain the following estimate on the angle $\ang{}{p}{x}{x'}$:
\begin{align*}
    \ang{}{p}{x}{x'} \le 
    \ang{\kappa}{p}{x}{x'} \le  
    2 \frac{\D{x}{x'}}{\D{p}{x}} \,,
\end{align*}
which in term implies
\begin{align*}
    |\ang{}{p}{q}{x}  - \ang{}{p}{q}{x'} | 
    \le \ang{}{p}{x}{x'} \le  2 \frac{\D{x}{x'}}{\D{p}{x}}\,.
\end{align*}
Relying on the fact that $t \mapsto \cos(t)$ is a Lipschitz function with constant $1$, 
\begin{align*}
    |\cos(\ang{}{p}{q}{x}) - \cos(\ang{}{p}{q}{x'})| 
    \le 2 \frac{\D{x}{x'}}{\D{p}{x}} \le \frac{1}{8} |\cos(\ang{}{p}{q}{x})|\,,
\end{align*} 
where the last inequality follows from the assumption $16 \frac{\D{x}{x'}}{\D{p}{x}}  \le |\cos(\ang{}{p}{q}{x})|$. Therefore, we conclude that $\cos(\ang{}{p}{q}{x})$ and $\cos(\ang{}{p}{q}{x'})$ have \emph{the same sign} and 
\begin{align}
\label{eq: lem-regularity-geometry-02}
    \frac{7}{8}|\cos(\ang{}{p}{q}{x})|
\le 
|\cos(\ang{}{p}{q}{x'})| 
\le 
    \frac{9}{8} |\cos(\ang{}{p}{q}{x})|\,.
\end{align}

 \step{Invoking Lemma~\ref{lem: M-opposite-side}}
 We apply Lemma~\ref{lem: M-opposite-side} to the triple $(p,q,x')$. To do that, we still need to show   
 \begin{align}
     \label{eq:lem-regularity-geometry-01}
     \D{p}{x'} < \frac{1}{4} \rM \qquad \mbox{and} \qquad
     \D{p}{q} < \frac{1}{4} \D{p}{x'}\,.
 \end{align}
 To drive this, we first note that the 
 $
     16 \frac{\D{x}{x'}}{\D{p}{x}}  \le |\cos(\ang{}{p}{q}{x})| \le 1 
 $
 implies
 \begin{align}
    \label{eq: lem-regularity-geometry-04}
     \frac{15}{16} \D{p}{x} \le \D{p}{x'} \le \frac{17}{16} \D{p}{x}\,. 
 \end{align}
 Combining this with the assumption $\D{p}{x} < \frac{1}{8} \rM$ and $\D{p}{q} \le \frac{1}{8} \D{p}{x}$ leads to~\eqref{eq:lem-regularity-geometry-01} and a slightly stronger estimate 
 \begin{align}
    \label{eq: lem-regularity-geometry-05}
    \D{p}{q} \le \frac{2}{15}  \D{p}{x'}\,. 
 \end{align}
 
 Now expressing 
\begin{align*}
    a:= \D{p}{x'}\,, \quad b:= \D{p}{q}\,, \quad c:= \D{q}{x'}\,,
    \quad \text{and} \qquad
    \theta := \ang{}{p}{x'}{q}
\end{align*}
and invoking Lemma~\ref{lem: M-opposite-side}, we have
\begin{align}
    \label{eq: lem-regularity-geometry-06}
 -\frac{7}{6}\frac{b^2}{a} - \frac{\kappa b^2}{3} \cdot b|\cos(\theta)|
\le
    a - c   - 
b \cos(\theta)
\le
    \pi \frac{b}{a} \cdot b |\cos(\theta)|\,.
\end{align}

\step{Simplifying the estimates}
We first claim that
\begin{align*}
\pi \frac{b}{a} \cdot b |\cos(\theta)| \le \frac{b}{2}|\cos(\theta)|
\quad \mbox{and} \quad 
    \frac{7}{6}\frac{b^2}{a} + \frac{\kappa b^2}{3} \cdot b|\cos(\theta)| \le  \frac{b}{2}|\cos(\theta)|  \,.
\end{align*}
The first inequality is a straightforward consequence of~\eqref{eq: lem-regularity-geometry-05} and the fact that $ \frac{2}{15}\pi \le 0.42$. 
By~\eqref{eq:lem-regularity-geometry-01}, 
\begin{align*}
    \frac{\kappa b^2}{3} 
=  
    \frac{\kappa a^2}{3} \frac{b^2}{a^2}
\le 
    \frac{1}{16\cdot 3} \frac{1}{16} \frac{b}{a}\,. 
\end{align*}

Also, our assumption $16 \frac{\D{x}{x'}}{\D{p}{x}}  \le |\cos(\ang{}{p}{q}{x})|$,~\eqref{eq: lem-regularity-geometry-04}, and~\eqref{eq: lem-regularity-geometry-02} imply
\begin{align*}
    \frac{b}{a} 
= 
    \frac{\D{p}{q}}{\D{p}{x'}}  = \frac{\D{p}{q}}{\D{p}{x}} \cdot \frac{\D{p}{x}}{\D{p}{x'}}
\le 
    \frac{1}{16} |\cos(\ang{}{p}{q}{x})| \cdot \frac{16}{15} 
\le 
 \frac{1}{16} \frac{9}{8}|\cos(\theta)| \cdot \frac{16}{15}  = \frac{3}{40} |\cos(\theta)|\,.
\end{align*}
Hence, 
\begin{align*}
    \frac{7}{6} \frac{b^2}{a} \le \frac{7}{80} b |\cos(\theta)|  \,.
\end{align*}
Combining these two estimates for the summands leads to the claim. Finally, the claim together with~\eqref{eq: lem-regularity-geometry-06} implies 
\begin{align*}
    1-0.42 \le  \frac{a-c}{b\cos(\theta)}\le 1+ 0.42 \,,
\end{align*}
regardless of the sign of $\cos(\theta)$. Unwrapping the definitions of $a,b,c,\theta$, 
together with~\eqref{eq: lem-regularity-geometry-02}, 
we obtain the first statement of the lemma. 

\step{Final estimate on $\rmp$}
Notice that 
\begin{align*}
    \max\{\D{p}{x'} , \D{q}{x'}\} 
\le 
    \max\{ \D{p}{x} + \D{x}{x'} , \D{q}{p} + \D{p}{x} +  \D{x}{x'}\}
\le 
     \Big(\frac{1}{16} + 1 + \frac{1}{16}  \Big) \D{p}{x}
\le 
    \rmrp\,,
\end{align*}
where the last inequality follows from the additional assumption $\D{p}{x} \le \frac{1}{8} \rmrp$. Within $[0, \rmrp]$, we can rely on the bi-Lipschitz continuity of $\rmp$ to obtain
\begin{align*}
    \ell_\rmp
\le 
\frac{ \rmp(\D{q}{x'}) - \rmp(\D{p}{x'})}{
    \D{p}{x'} - \D{q}{x'}}
\le 
    L_\rmp\,,
\end{align*}
where we emphasize that due to $\rmp$ is monotone decreasing, $\rmp(\D{q}{x'}) - \rmp(\D{p}{x'})$ and $\D{p}{x'} - \D{q}{x'}$ always have the same sign. 
Substituting this ratio into the first statement of the lemma, we obtain second statement of the lemma.

\end{proof}


\section{Common Neighbor Probability and cluster finding algorithm}
 
\medskip

Using the algebraic identity $(a-b)^2 = a^2 -2ab +b^2$, we obtain the following proposition.
\begin{prop}\label{prop: Kxy-formula-integral}
For any points $x, y \in M$, we have
\begin{equation}\label{eq:Kdiagonal} 
\ipcn{x}{y} = \frac{ \ipcn{x}{x}+\ipcn{y}{y}}{2} - \frac12 \int_M \big(\rmp(\D{x}{z}) - \rmp(\D{y}{z})\big)^2 \, \pd \mu(z).
\end{equation}
\end{prop}

The following lemma derives an quantitative estimate on the integral of the above identity. 

\begin{lemma}\label{lem: Kpq-upper-bound}
Recall from~\eqref{def: radius_G} that
\begin{align*}
    \rG = \frac{1}{16}\min \big\{ \rmrp,\, \rinj(M), \dm{\kappa},\, {\rm r}_\mu \big\}\,.
\end{align*}
   For any $p, q \in M$, we have
\[
\ipcn{p}{q} \le \frac{ \ipcn{p}{p}+\ipcn{q}{q}}{2} 
-
c \cdot \min\{ \rG^2,\,\D{p}{q}^2\}\,.
\] 
where $c = c_{\tref{lem: Kpq-upper-bound}} := \mu_{\rm min}(\rG) \cdot\frac{\lrmp^2}{2\pi^2}$.
\end{lemma}

Let us postpone the proof to the end of the section. With the quantitative estimate, we can collect pairs $(X_i,X_j)$ whose $\ipcn{X_i}{X_j}$ is close to $\frac{\ipcn{X_i}{X_i}+\ipcn{X_j}{X_j}}{2}$, which allows us to deduce that the distance $\D{X_i}{X_j}$ is small. Building from this we have our first algorithm.

\medskip
\begin{algorithm}[H]
\caption{\texttt{GenerateCluster}}
\label{alg:cluster}

\SetKwInOut{Input}{Input}
\SetKwInOut{Output}{Output}

\Input{ 
$U \subseteq {\bf V}$\,. 
$V \subseteq {\bf V}$, a subset of size between $2$ and $n$. \\
$\eta>0$. 
}
\Output{
    $W_{\rm gc} \subseteq V$, a subset of vertices. \\
    $i_{\rm gc} \in W_{\rm gc}$, a vertex.
}
{\em Step 1.} Sort \(\{i,j\} \in { V \choose 2 }\) according to $\cn{U}{i,j}$ from the largest to the smallest, and return a list $\{i_1, j_1\}$, $\{i_2, j_2\}$, $\ldots$, $\{i_L, j_L\}$, where $L = \binom{|V|}{2} = \big| \binom{V}{2} \big|$.

{\em Step 2.} Let $m$ be the largest positive integer such that
\[
\cn{U}{i_m,j_m} \ge  \cn{U}{i_1,j_1} -  \frac{1}{2}c_{\tref{lem: Kpq-upper-bound}} \eta^2\,,
\]
where $c_{\tref{lem: Kpq-upper-bound}}$ is the constant from Lemma~\ref{lem: Kpq-upper-bound}.

{\em Step 3.} Consider a graph 
$\gc{G}$
with vertex set $ \gc{V} := \big( \bigcup_{k \in [m]} \{i_k\}\big) \cup \big( \bigcup_{k \in [m]} \{j_k\} \big)$ and edge set $\gc{E} := \left\{ \{i_k,j_k\} \, : \, k \in [m] \right\}$.

{\em Step 4.} Take a pair $(\gc{W}, \gc{i})$ where $\gc{i} \in \gc{V}$ maximizes the size of neighbors in $\gc{G}$ and 
$\gc{W} = \left\{j \in \gc{V} \, : \, \{\gc{i}, j\} \in \gc{E} \right\} \cup \{\gc{i}\}$.

\Return{$( \gc{W}, \gc{i} )$}
\end{algorithm}
\begin{rem}
\label{rem: cluster-finding-algorithm}
The common neighbor counts $\cn{U}{i,j}$ for all pairs $\{i,j\} \in \binom{V}{2}$ are sub-entries of the matrix $A^\top A$, where $A$ is the adjacency matrix restricted to the block $U \times V$. Using fast sparse matrix--matrix multiplication algorithms, the computation of $A^\top A$ can be performed in time $O(n^{2}\sp^2|U|)$ with probability $1-o(1)$ assuming the concentration of degree on the graph, and converting $A$ to a sparse representation requires $O(n|U|)$ time. Therefore, the overall running time of the algorithm is 
\[
O\!\big(n^{2}(\sp^2|U| + \log n)\big),
\]
which is dominated by either the sorting step or by determining the common neighbor counts $\cn{U}{i,j}$ for all pairs $\{i,j\} \in \binom{V}{2}$.
\end{rem}

\begin{defi}
    [$(\eta,m)$-dense set]
    For $\eta>0$ and a positive integer $m$, a subset of vertices $W \subseteq {\bf V}$ is called \emph{$(\eta,m)$-dense}\,, if for every $w \in W$, the set  
    \begin{align*}
        \Big|\Big\{ w' \in W \, : \, \D{X_w}{X_{w'}} \le \eta \Big\} \Big| \ge m\,.
    \end{align*}
\end{defi}

\begin{prop}
\label{prop: cluster-finding-algorithm}
There exists a universal constant $C_{\tref{prop: cluster-finding-algorithm}}>1$ such that the following holds.
Suppose $V,U$ are two disjoint subset of vertices in our random geometric graph model (Definition \ref{def: random-geometric-graph}) with $|V| \le n$ such that 
the events $\Ecn{U}{V}$ holds. 
Let $\eta$ be a positive parameter satisfying  
\begin{align}
    \label{eq: cluster-finding-eta-assumption-Usize}
 C_{\tref{prop: cluster-finding-algorithm}}\left(\frac{\log(n)}{\sqrt{\sp^2|U|}}\right)^{1/2} \le \eta \le \rG\,.
\end{align}
and set 
\begin{align}
    \label{eq: cluster-finding-assumption}
    m:= \min_{i \in V}\Big|\Big\{j \in V\,: \D{X_i}{X_j} \le c\eta^2 \Big\}\Big| 
\end{align}
where 
$$
    c = c_{\tref{prop: cluster-finding-algorithm}} = \frac{1}{4} \frac{c_{\tref{lem: Kpq-upper-bound}}}{\max\{L_\rmp,\,1\}}
$$
and 
$c_{\tref{lem: Kpq-upper-bound}}$ is the constant from Lemma~\ref{lem: Kpq-upper-bound}.
In other words, $V$ is a $(c\eta^2,m)$-dense set by the definition above.
Then, the output of Algorithm~\ref{alg:cluster} returns a pair $(W_{\rm gc}, i_{\rm gc})$ such that the following holds: 
\begin{align*}
    \forall i \in W_{\rm gc}\,, \quad \D{X_{i_{\rm gc}}}{X_i} \le \eta 
\mbox{ and }
    |W_{\rm gc}| \ge m\,. 
\end{align*}
\end{prop}
\begin{proof}
    First, from Cauchy--Schwarz inequality, we have
    \begin{align*}
        \max_{i,j \in V} \ipcn{X_i}{X_j}
    = 
        \max_{i \in V} \ipcn{X_i}{X_i}\,. 
    \end{align*}
    Let $i_0 \in V$ be a vertex such that $\ipcn{X_{i_0}}{X_{i_0}} = \max_{i \in V} \ipcn{X_i}{X_i}$. 
    For any $(i,j) \in \binom{V}{2}$, by the event $\Ecn{U}{V}$ we have 
    \begin{align*}
        \max_{i,j \in V} 
        \cn{U}{i,j}
    \le 
        \max_{i,j \in V} 
        \ipcn{X_i}{X_j}
        + \frac{\log(n)}{\sqrt{\sp^2|U|}}
    \le 
        \ipcn{X_{i_0}}{X_{i_0}} + \frac{\log(n)}{\sp^2|U|}\,.
    \end{align*}

    \step{Lower bound on $|\gc{W}|$}
    Consider the set 
    $$
        S 
    := 
        \{ j \in V \setminus \{i_0\}  
        : \D{X_{i_0}}{X_j} \le c \eta^2\}\,. 
    $$
    Therefore, for any $j \in S$, we have
    \begin{align}
        \label{eq: cluster-finding-algorithm-00}
        \cn{U}{i_0,j}
    \ge
        \ipcn{X_{i_0}}{X_j}
        - \frac{\log(n)}{\sp^2|U|} 
    \ge& 
        \ipcn{X_{i_0}}{X_{i_0}}
        - L_p c \eta^2 
        - \frac{\log(n)}{\sqrt{\sp^2|U|}} \\
    \nonumber
    &\ge 
        \max_{(i,j) \in \binom{V}{2}} \cn{U}{i,j}
        - L_p c \eta^2
        - 2\frac{\log(n)}{\sqrt{\sp^2|U|}} \\
    \nonumber
    & \ge 
        \max_{(i,j) \in \binom{V}{2}} \cn{U}{i,j}
        - \frac{1}{2}c_{\tref{lem: Kpq-upper-bound}}\eta^2\,,
    \end{align}
    where the last inequality follows from our choice of $c$ and setting $C_{\tref{prop: cluster-finding-algorithm}}$ sufficiently large.

    Therefore, the collection of pairs $\{(i_0,j) : j \in S\}$ will be edges in the graph $\gc{G}$ constructed in Algorithm~\ref{alg:cluster}. This implies that the number of neighbors of $\gc{i}$ is at least the number of $i_0$ in $\gc{G}$, which is at least $|S| \ge m$ from the $(c\eta^2,m)$-density assumption on $V$.  

    \step{Upper bound on $\D{X_{\gc{i}}}{X_j}$ for $j \in \gc{W}$}
    Any pair $(i',j') \in \binom{V}{2}$ appears as an edge in the graph $\gc{G}$ if and only if
    \begin{align*}
        \cn{U}{i',j'} \ge  
        \max_{(i,j) \in \binom{V}{2}} \cn{U}{i,j}
        - \frac{1}{2} c_{\tref{lem: Kpq-upper-bound}} \eta^2\,.
    \end{align*}
    First, replacing $\cn{U}{i',j'}$ by its expected value we have 
    \begin{align*}
        \ipcn{X_{i'}}{X_{j'}}
    \ge 
        \max_{(i,j) \in \binom{V}{2}} \cn{U}{i,j}
        - \frac{1}{2} c_{\tref{lem: Kpq-upper-bound}} \eta^2 - \frac{\log(n)}{\sqrt{\sp^2n}}\,.
    \end{align*}

    From the estimate~\eqref{eq: cluster-finding-algorithm-00}, take any $j'' \in S$, we have 
    \begin{align*}
        \max_{(i,j) \in \binom{V}{2}} \cn{U}{i,j}  -  \frac{\log(n)}{\sqrt{\sp^2n}}
    &\ge 
        \cn{U}{i_0,j''}-  \frac{\log(n)}{\sqrt{\sp^2n}} \\
    & \stackrel{\eqref{eq: cluster-finding-algorithm-00}}{\ge} 
        \ipcn{X_{i_0}}{X_{i_0}} - L_\rmp c \eta^2 - \frac{\log(n)}{\sqrt{\sp^2n}}
        -  \frac{\log(n)}{\sqrt{\sp^2n}} \\
    &\ge 
        \ipcn{X_{i_0}}{X_{i_0}} - \frac{1}{2} c_{\tref{lem: Kpq-upper-bound}} \eta^2\,,
    \end{align*}
    which implies that
    \begin{align*}
        \ipcn{X_{i'}}{X_{j'}}
    \ge  
        \ipcn{X_{i_0}}{X_{i_0}} - c_{\tref{lem: Kpq-upper-bound}} \eta^2
         \ge \frac{\ipcn{X_{i'}}{X_{i'}} + \ipcn{X_{j'}}{X_{j'}}}{2} - c_{\tref{lem: Kpq-upper-bound}} \eta^2\,.
    \end{align*} 
    By Lemma~\ref{lem: Kpq-upper-bound}, the above condition gives  
    \begin{align*}
        \D{X_{i'}}{X_{j'}} \le \eta\,.
    \end{align*}

\end{proof}

For the proof of Lemma~\ref{lem: Kpq-upper-bound}, we need a purely geometric statement which we now present as a separate lemma.

\begin{lemma} \label{lem: differentDistance}
\MAssump
Let $r$ be any positive real number which satisfies
$$
    0< r \le \min   \left\{ \frac14\rinj(M), \frac14\dm{\kappa}\right\}.
$$
Then for any points $p,q\in M$ with $\D{p}{q} < r$ there exists a
geodesic ball $B=\B{r/4}{z} \subseteq M$ for some $z \in M$ satisfying
$\D{p}{z} = r/2$
such that
\[
\D{q}{x} \ge \D{p}{x} + \frac{1}{\pi} \D{p}{q}
   \qquad\text{for every }x\in B\,.
\]
\end{lemma}


\begin{proof} [Proof of Lemma~\ref{lem: differentDistance}]

Consider the geodesic moving from $p$ in the opposite direction from $p$ to $q$ and reaching a point in distance $r/2$. Let us denote this point by \(z\). In our notations, we write
\[
  z \;=\; \exp_p   \left(-\frac{r}{2} \cdot \frac{\tv{q}}{\|\tv{q}\|} \right).
\]
Here we recall that $\tv{q} \in T_pM$ is the unique tangent vector with $\|\tv{q}\| < r$ such that $\exp_p(\tv{q}) = q \in M$. With $r/2 < \rinj(M)$, we know that this geodesic segment between $p$ and $z$ is minimizing: $\D{p}{z} = r/2$. Now consider the geodesic ball of radius $r/4$ centered at $z$ and an arbitrary point
\[
  x\;\in\;\B{r/4}{z} \subseteq M \,.
\]

\step{Lower bound on the angle $\ang{}{p}{q}{x}$} The comparison inequality~\eqref{eq: intro-angle} from Lemma~\ref{lem: tri_lem} applied with the upper curvature bound
gives
\[
   \ang{}{p}{z}{x}\;\le\;\ang{\kappa}{p}{z}{x}.
\]
Since we have
\begin{itemize}
    \item $\D{p}{z} = r/2$,
    \item $\D{z}{x} < r/4$,
    \item $\D{p}{x} \ge \D{p}{z} - \D{z}{x} > r/2 - r/4 = r/4$,
    \item $\D{p}{x} \le \D{p}{z} + \D{z}{x} < r/2 + r/4 = 3r/4$, and
    \item $r \le \varpi_\kappa/4 = \pi/(4\sqrt{\kappa})$,
\end{itemize}
a direct model-space computation (we leave the explicit computation to Lemma~\ref{lem:spherical-angle}) yields
\[
   \ang{\kappa}{p}{z}{x}\;\le\;\frac{\pi}{3}.
\]


Consequently
\[
   \ang{}{p}{q}{x}
   \; = \;
   \ang{}{p}{q}{z}\;-\;\ang{}{p}{z}{x}
   \; \ge \;
   \pi\;-\;\frac{\pi}{3}
   \; = \;
   \frac{2\pi}{3}\,.
\]

\step{Opposite–side comparison for the triangle \(\triangle pqx\)}
By the triangle inequality, we have
$$
    \D{p}{x} \le \D{p}{z} + \D{z}{x} < \frac{r}{2} + \frac{r}{4} = \frac{3r}{4}\,.
$$
Hence,
$$
    \D{p}{x} + \D{p}{q} < \frac{3r}{4} + r \le \varpi_\kappa\,,
$$
where the last inequality follows from the definition of $r$. Consequently, the model angle $\ang{\kappa}{p}{x}{q}$ is well-defined and Lemma~\ref{lem: tri_lem} gives
$$
    \D{q}{x} 
    \ge  \os^{\kappa}(\ang{}{p}{q}{x}; \D{p}{q}, \D{p}{x}) 
    \ge  \os^{\kappa}\Big( \frac{2}{3}\pi; \D{p}{q}, \D{p}{x} \Big)\,,
$$ 
where the last inequality follows from the fact that 
$ \theta \mapsto \os^{\kappa}\Big( \theta ; \D{p}{q}, r/2 \Big)$ is a monotone increasing on $[0,\pi]$ (see e.g.~\cite[Ch.~1]{AKP24}), and we have shown above that $\ang{}{p}{q}{x} \ge \frac{2}{3}\pi$. Finally, a direct estimate in $M_{\kappa}$ (see Lemma~\ref{lem:spherical-opposite-side}) shows that 

$$
\os^{\kappa}\Big( \frac{2}{3}\pi; \D{p}{q}, \D{p}{x} \Big) 
\ge 
\D{p}{x} +  \frac{1}{2} \left|\cos \Big(\frac{2}{3}\pi\Big)\right| \D{p}{q}
=  
\D{p}{x} + \frac{1}{4}\D{p}{q}\,.
$$
\end{proof}


\begin{proof}[Proof of Lemma~\ref{lem: Kpq-upper-bound}]
The proof is similar to the proof of Lemma~3.1 in~\cite{HJM24}, which breaks into two cases depending on $\D{p}{q}$. 
Recall from Proposition~\ref{prop: Kxy-formula-integral} that 
\begin{align*} 
\frac{ \ipcn{x}{x}+\ipcn{y}{y}}{2} - \langle x, y\rangle =  \frac{1}{2}\int_M (\rmp(\D{x}{z}) - \rmp(\D{y}{z}))^2 \, \pd \mu(z)\,,
\end{align*}
so our goal is to estimate the integral from below. 
    
\step{Case 1: $\D{p}{q} < 4\rG$}
    Here we can invoke Lemma~\ref{lem: differentDistance} to show that there exists a ball $B=\B{\rG}{z}$ for some $z \in M$ with $\D{p}{z} = 2\rG$ such that for every $x \in B$, we have
    \begin{align*}
        |\D{p}{x} - \D{q}{x}| \ge  \frac{1}{\pi}  \D{p}{q}\,.
    \end{align*}
    Following triangle inequalities, we have 
    $$
        \max \{ \D{p}{x}, \D{q}{x} \} < (4+2+1)\rG < \radrmp\,.
    $$
    Hence, we have 
    \begin{align*}
        |\rmp(\D{p}{x}) - \rmp(\D{q}{x})| 
        \ge \frac{\lrmp}{\pi} \D{p}{q}\,.
    \end{align*}
    Now we use this bound to estimate the integral
    \begin{align*}
       \int_M (\rmp(\D{p}{x}) - \rmp(\D{q}{x}))^2 \, {\rm d}\mu(x) 
    \ge 
        \int_B  (\rmp(\D{p}{x}) - \rmp(\D{q}{x}))^2 \, {\rm d}\mu(x) 
    \ge 
       \mu_{\min}(\rG) \frac{\lrmp^2}{\pi^2} \D{p}{q}^2\,.
    \end{align*}

    \step{Case 2: $\D{p}{q} \ge 4\rG$}
    Now suppose $\D{p}{q} \ge 4\rG$. Here we consider the ball $B=\B{\rG}{p}$
    For any $x \in B$, simply from triangle inequalities, we have
    $$
        \D{q}{x} \ge \D{p}{q} - \D{p}{x} \ge 4\rG - \rG = 3\rG\,. 
    $$
    Together with $\D{p}{x} \le  \rG$, we have 
    \begin{align*}
        |\rmp(\D{p}{x}) - \rmp(\D{q}{x})| 
        \ge \rmp\Big(3\rG\Big) - \rmp\Big(\rG\Big)  
        \ge \lrmp \cdot 2\rG\,. 
    \end{align*}
    Then, applying the same argument as in Case 1, we have
    \begin{align*}
       \int_M (\rmp(\D{p}{x}) - \rmp(\D{q}{x}))^2 \, {\rm d}\mu(x)
    \ge 
        \int_B  (\rmp(\D{p}{x}) - \rmp(\D{q}{x}))^2 \, {\rm d}\mu(x)
    \ge
        \mu_{\rm min}(\rG) \cdot \lrmp^2 \cdot 4\rG^2\,.
    \end{align*}
\end{proof}

\section{ Scenario 1: Forming an approximate orthogonal basis}
We introduce the first scenario in an intermediate step of the overall algorithm. 
\begin{defi}
\label{def: event-orthogonal-ortho}
For $k \in [0,d]$, and consider $k+1$ pairs $(\rho_i)_{i \in [0,k]} \in \bfV$. For parameters $0 < \eta < \delta < r$, define 
\begin{align*}
    \Eortho\big(r,\delta, \eta, (\rho_i)_{i \in [0,k]}\big) 
\end{align*}
to be the event of $(X_{U_{\rho_i}})_{i \in [0,k]}$ that the following two conditions hold:
 \begin{align}
    \label{eq: event-orthogonal-ortho}
    |\D{X_{\rho_0}}{X_{\rho_i}} - r| \le \delta\,, \qquad \mbox{for } i \in [k]\,,
    \qquad \mbox{and} \qquad
    |\D{X_{\rho_i}}{X_{\rho_j}} - \sqrt{2}r| \le \delta\,, \qquad \mbox{for } i,j \in [k]\,,
 \end{align}

\end{defi}

The main statement of this scenario is as follows:
\begin{prop}
\label{prop: scenerio-1-ortho}
There exists a constant \(c_{\tref{prop: scenerio-1-ortho}} = c(d, L_\rmp, \ell_\rmp) > 0 \) and \(c'_{\tref{prop: scenerio-1-ortho}} = c'(d, L_\rmp, \ell_\rmp) > 0\)
such that the following holds.
For any $c < c_{\tref{prop: scenerio-1-ortho}}$, let $0 < \eta < \delta < r$ and a positive integer $m$ satisfies the parameter configuration \eqref{eq: parameter-configuration}:
\begin{align}
\label{eq: scenario-1-ortho-parameter-assumption}
\fe{m} < c^2\eta, \qquad
\eta = c \delta, \qquad
\delta = c r, \qquad
r \le c \,\rG\,.
\end{align}

Suppose we are given $k+1$ pairs  $(\rho_i, U_{\rho_i})_{i \in [0,k]}$ with \(\rho_i \subseteq U_{\rho_i} \subset \bfV\) and $|U_{\rho_i}| \ge m$ (with $k \le d$). 

Let \(W \subset \bfV\) be disjoint from all $U_{\rho_i}$ and let 
\begin{align*}
   m' := \inf_q \left| \{X_w\}_{w \in W} \cap \B{c'_{\tref{prop: scenerio-1-ortho}} c_{\tref{prop: cluster-finding-algorithm}}\eta^2}{q}  \right|
\end{align*}
where the infimum is taken over all points \( q \in M \) satisfying
\[
|\D{X_{\rho_0}}{q} - r| < 0.92 \delta
\quad \text{and} \quad
|\D{X_{\rho_i}}{q} - \sqrt{2}r| < 0.92 \delta \quad (\forall i \in [k]).
\]
Then, on the event
\[
{\cal E}:= 
\Eortho(r, \delta, \eta, (\rho_i)_{i \in [0,k]})
\cap \bigcap_{i \in [0,k]} \Eclu(\eta,U_{\rho_i},\rho_i) 
\cap \bigcap_{i \in [0,k]} \Enavi{U_{\rho_i}}{W},
\]
the set
\begin{align}
\label{eq: scenario-1-ortho-W'}
W’ :=&
\nonumber
W'(r, \delta, W, (U_{\rho_i}, \rho_i)_{i \in [k]}) \\
= &\left\{ w \in W :
\begin{array}{l}
\rmp(r + 0.91\delta) \le \cn{U_{\rho_0}}{w} \le \rmp(r - 0.91\delta), \\
\rmp(\sqrt{2}r + 0.91\delta) \le \cn{U_{\rho_i}}{w} \le \rmp(\sqrt{2}r - 0.91\delta) \quad (\forall i \in [k])
\end{array}
\right\}
\end{align}
is non-empty, and for every \( w \in W' \), the following holds:
\begin{align*}
    |\B{c_{\tref{prop: cluster-finding-algorithm}}\eta^2}{X_{w'}} \cap W'| \ge m'\,.
\end{align*}
\end{prop}

\begin{rem}
    \label{rem: scenaro-1-W'-distance}
    Within the event ${\cal E}$, by~\eqref{eq: calibration_annulus_threshold} from Calibration Lemma~\ref{lem:calibration-toolkit}, for every $w' \in W'$ we have
    \begin{align*}
        |\D{X_{\rho_0}}{X_w} - r| \le
         0.92 \delta
        \qquad \mbox{and} \qquad 
        |\D{X_{\rho_i}}{X_w} - \sqrt{2}r| \le
         0.92 \delta\,,
    \end{align*}
    provided that $c$ is chosen sufficiently small. 
\end{rem}

\begin{rem}
    In view of Algorithm~\ref{alg:cluster}, the set \( W' \) defined in~\eqref{eq: scenario-1-ortho-W'} is a \((c_{\tref{prop: cluster-finding-algorithm}}\eta^2, m')\)-dense set, which is capable of extracting a cluster in $W'$ with radius $\eta$ and size at least \( m' \) by Proposition~\ref{prop: cluster-finding-algorithm} by measuring common neighbor sizes. 
\end{rem}

\subsection{$W'$ is not empty.}
Our goal in this subsection is to show that the set \( W' \), defined in Proposition~\ref{prop: scenerio-1-ortho}, is non-empty. 
The key idea is to construct a point \( y \in M \) that is approximately orthogonal to the points \( X_{\rho_i} \in M \), where each \( X_{\rho_i} \) belongs to one of the clusters \( U_{\rho_i} \). 
This construction is possible when \( r \) is sufficiently smaller than \( \rM \), which minimizes the effect of curvature, and when \( \delta/r \) is small enough to ensure approximate orthogonality. 
Both conditions can be achieved by choosing a sufficiently small constant \( c \) in the parameter assumptions of~\eqref{eq: scenario-1-ortho-parameter-assumption} from Proposition~\ref{prop: scenerio-1-ortho}.
We now formalize this in the following geometric lemma:

\begin{lemma}
    \label{lem: existence-approximate-orthogonal-basis}
    Let $p,x_1,\dots, x_k \in M$, where $k <d$, be points in $M$ such that
    \begin{align*}
        |\D{p}{x_i} - r| \le \delta\,, \qquad \mbox{for } i \in [k]\,,
    \end{align*}
    and assume $r > \delta > 0$ satisfies the parameter assumption from Proposition~\ref{prop: scenerio-1-ortho}, with a sufficiently small constant $c$. 
    Then, there exists a point $y \in M$ such that
    \begin{align*}
        \D{p}{y} = r\,, \qquad \mbox{and} \qquad
        |\D{y}{x_i} - \sqrt{2}r| \le 0.9\delta \quad \mbox{for}\quad i \in [k]\,.
    \end{align*}
\end{lemma}
We now apply this geometric construction to prove the non-emptiness of \( W' \) under the assumptions of Proposition~\ref{prop: scenerio-1-ortho}. 
The proof of Lemma~\ref{lem: existence-approximate-orthogonal-basis} is postponed to the end of this subsection.

\begin{cor}
    Under the assumptions of Proposition~\ref{prop: scenerio-1-ortho}, $W'$ is not empty. 
\end{cor}
\begin{proof}
    Let  $p =X_{\rho_0} $ and $x_i = X_{\rho_i}$ for $i \in [k]$. These satisfy the assumption of Lemma~\ref{lem: existence-approximate-orthogonal-basis} by construction.
    Hence, there exists a point $q \in M$ such that
    \begin{align*}
        \D{p}{q} = r\,, \qquad \mbox{
        and} \qquad
        |\D{q}{x_i} - \sqrt{2}r| \le 0.9\delta \quad \mbox{for } i \in [k]\,.
    \end{align*}
    This point \( q \) satisfies the geometric constraints of Proposition~\ref{prop: scenerio-1-ortho}. By assumption, this implies the existence of some \( w \in W \) such that
    \begin{align*}
        \D{p}{X_w} \le c_{\tref{prop: cluster-finding-algorithm}} \eta^2 
        \,.
    \end{align*}
Using the triangle inequality:
\[
\left| \D{X_{\rho_0}}{X_w} - r \right| \le c_{\tref{prop: cluster-finding-algorithm}} \eta^2 
\quad \text{and} \quad
\left| \D{X_{\rho_i}}{X_w} - \sqrt{2}r \right| \le 0.9\delta + c_{\tref{prop: cluster-finding-algorithm}} \eta^2.
\]
Under the event \( \Eclu(\eta,U_{\rho_0},\rho_0) \) and the event \( \Enavi{U_{\rho_0}}{W} \), the Calibration Lemma~\ref{lem:calibration-toolkit} applies, yielding 
\[
\rmp(r + 0.9\delta+2\eta) \le \cn{U_{\rho_0}}{w} \le \rmp(r - 0.9\delta-2\eta)\,,
\]
and when \( c \) is sufficiently small, we have \( 2\eta < 0.01\delta \). Further,  an identical argument applies to all \( i \in [k] \), showing that \( w \in W' \). Therefore, \( W' \) is non-empty.
\end{proof}

\begin{proof}[Proof of Lemma~\ref{lem: existence-approximate-orthogonal-basis}]

Recall 
$$
    \tv{x_i} := (\tv{x_i})_p \in T_pM \qquad \mbox{for } i \in [d]
$$
is the unique vector (thanks to $\D{x_i}{p} < c\rM$) such that $\exp_p(\tv{x_i}) = x_i$ and $\|\tv{x_i}\| = \D{p}{x_i}$.

    Since $k < d$, there exists a vector $w \in T_pM$ orthogonal to all $\tv{x_i}$, i.e.,
    \begin{align*}
        \langle \tv{x_i}, w \rangle = 0 \qquad \mbox{for all } i \in [k]\quad  \mbox{and} \qquad
        \|w\| = r\,.
    \end{align*}
    We claim that $$y = \exp_p(w)$$ is the desired point. Clearly,  
    $$\D{p}{y} = \|w\| = r\,.$$
    To estimate $\D{y}{x_i}$, we note that
    $$
        \ang{}{p}{x_i}{y} = \arccos\Big(\frac{\langle \tv{x_i}, w \rangle}{\|\tv{x_i}\|\|w\|}\Big) = \arccos(0) =\frac{\pi}{2}\,.
    $$
    We now apply Lemma~\ref{lem: M-opposite-side-fixed-angle} with $\theta = \frac{\pi}{2}$, $\D{p}{x_i} = \|\tv{x_i}\|$, and $\D{p}{y} = r$ to obtain
    \begin{align*}
        \Big| \D{x_i}{y}^2 - 2r^2 \Big| 
        \le 2\delta r + \delta^2 + 2\kappa (r+\delta)^4 \,.
    \end{align*}
    With the assumption that $r \le c \min\{ \rinj(M), \dm{\kappa_1}\}$ and $\delta = cr$, we have  
    \begin{align*}
        2\kappa (r+\delta)^4 = 2\kappa (1+c)^4 r^4  \le 2 (1+c)^4 c^2r^2 
        \le 2 (1+c)^4 \delta^2 \le 3 \delta^2\,.
    \end{align*}
    Thus, 
    \begin{align*}
        2\delta r + \delta^2 + 2\kappa (r+\delta)^4 
    \le 
       \delta r (2+3c)
    \le 3cr^2\,.
    \end{align*}
    We can first conclude that 
    $$
        \D{x_i}{y} \ge \sqrt{2r^2 -3cr^2} = \sqrt{2-3c}r\,, 
    $$
    and consequently,
    \begin{align*}
        \Big| \D{x_i}{y} - \sqrt{2}r \Big| 
    = 
        \frac{\Big| \D{x_i}{y}^2 - 2r^2 \Big|}{\Big| \D{x_i}{y} + \sqrt{2}r \Big| }
    \le 
        \frac{\delta r(2+3c)}{ \sqrt{2-3c}r + \sqrt{2}r}
    \le 
        0.9\delta\,,
    \end{align*}
    where the last inequality holds when $c$ is sufficiently small.
\end{proof}

\subsection{Proof of Proposition~\ref{prop: scenerio-1-ortho}}

Having established that \( W' \) is non-empty, we now aim to show that for every \( w' \in W' \), there exists a ball of radius \( \eta \) centered at \( X_{w'} \) such that
\[
    \left| \B{\eta}{X_{w'}} \cap W \right| \ge m'\,.
\]

Intuitively, \( W' \) corresponds to the subset of latent points in \( W \) that lie near the intersection of \( k+1 \) fuzzy annuli — “fuzzy” in the sense that membership is defined via edge constraints to cluster cores rather than precise metric distances. The main difficulty arises when a point \( X_{w'} \in W' \) is close to the boundary of this intersection, where small perturbations may break the constraints.

To address this, we consider moving slightly along a geodesic \( \gamma \), starting at \( X_{w'} \) in a direction \( \xi \in T_{X_{w'}} M \), to reach a point \( z = \gamma(s) \) for some small \( s > 0 \), thereby \emph{moving away from the boundary of the intersection region}. 

Strategically, the goal is to find such a direction \( \xi \) so that the resulting point \( z \) satisfies
\[
    \B{c_{\tref{prop: cluster-finding-algorithm}}\eta^2}{z} \cap W \subseteq W',
\]
thus ensuring that a robust neighborhood of \( z \) lies entirely within \( W' \).

For a fixed \( w' \in W' \), define
\[
    \sigma_0 =
    \begin{cases}
        1 & \text{if } \cn{U_{\rho_0}}{w'} < \rmp(r), \\
       -1 & \text{otherwise}.
    \end{cases}
\]
We then choose \( \xi \in T_{X_{w'}}M \) such that
\[
  \mathrm{sign}\big( \langle \xi, \tv{\rho_0} \rangle \big) = \sigma_0,
\]
where \( \tv{\rho_0} \in T_{X_{w'}}M \) is the tangent vector satisfying \( \exp_{X_{w'}}(\tv{\rho_0}) = X_{\rho_0} \).
Suppose \( \cn{U_{\rho_0}}{w'} < \rmp(r) \), this suggests that \( \D{X_{\rho_0}}{X_{w'}} > r \). In this case, setting \( \sigma_0 = 1 \) means choosing \( \xi \) so that \( \langle \xi, \tv{\rho_0} \rangle > 0 \), which moves \( z \) toward \( X_{\rho_0} \) along the geodesic and decreases \( \D{X_{\rho_0}}{z} \), bringing it closer to \( r \).  
The reverse case \( \cn{U_{\rho_0}}{w'} > \rmp(r) \) corresponds to \( \sigma_0 = -1 \), so that moving in the opposite direction increases the distance toward \( r \).

Following the same reasoning, for each \( i \in [k] \) we set
\[
    \sigma_i =
    \begin{cases}
        1 & \text{if } \cn{U_{\rho_i}}{w'} < \rmp(\sqrt{2}r), \\
       -1 & \text{otherwise}.
    \end{cases}
\]
Next, choose \( \xi' \in \mathrm{span}(\{\tilde{\rho_i}\}_{i \in [0,k]}) \) such that
\[
    \mathrm{sign}\big( \langle \xi', \tilde{\rho_i} \rangle \big) = \sigma_i \qquad \text{for all } i \in [0,k],
\]
where
\[
    \tilde{\rho_i} = \frac{\tv{\rho_i}}{\|\tv{\rho_i}\|} \in T_{X_{w'}}M \qquad \text{for } i \in [0,k].
\]
Finally, normalize to obtain
\[
    \xi = \frac{\xi'}{\|\xi'\|}.
\]
In other words, we have 
$$
    \langle \xi , \tilde{\rho_i} \rangle = \frac{\sigma_i}{\|\xi'\|}\,.  
$$

We need to ensure that \( \|\xi'\| \) is not too large, so that moving along the direction \( \xi \) with a small distance still produces a non-trivial change in the distance to each \( X_{\rho_i} \).

\begin{lemma}
\label{lem:xi-prime-bound}
Under the assumptions of Proposition~\ref{prop: scenerio-1-ortho}, for any fixed \( w' \in W' \), the vector \( \xi \) defined above satisfies
\[
    \|\xi' \|_2 \le \sqrt{2(k+1)}, \quad \text{and } \quad
    \frac{\langle \xi , \tilde{\rho_i} \rangle}{\sigma_i} \ge \frac{1}{2(k+1)}
\]
provided \( c_{\tref{prop: scenerio-1-ortho}} \) is chosen sufficiently small.
\end{lemma}
\begin{proof}
Let \( \xi' = \sum_{i=0}^k a_i \tilde{\rho_i} \) with \( a_i \in \mathbb{R} \). Then the coefficients satisfy
\[
B \vec a = \vec \sigma, 
\qquad 
B := \big( \langle \tilde{\rho_i}, \tilde{\rho_j} \rangle \big)_{i,j=0}^k,
\qquad 
\vec a := (a_i)_{i=0}^k,
\qquad 
\vec \sigma := (\sigma_i)_{i=0}^k.
\]
In particular, \( \vec a = B^{-1}\vec \sigma \). Hence
\[
\|\xi'\|_2^2 
= \Big\| \sum_{i=0}^k a_i \tilde{\rho_i} \Big\|_2^2
= \vec a^{\top} B \vec a
= \vec \sigma^{\top} B^{-1} \vec \sigma
\le \|B^{-1}\|_{\rm op} \, \|\vec \sigma\|_2^2
\le \|B^{-1}\|_{\rm op} \,(k+1).
\]
Therefore, it suffices to show that \( s_{\min}(B) \ge \tfrac{1}{2(k+1)} \), i.e. \( \|B^{-1}\|_{\rm op} \le 2(k+1) \), which then yields
\(
\|\xi'\| \le \sqrt{2(k+1)}.
\)

\step{From law of cosine to inner products}
For any \( \{i,j\} \in \binom{[0,k]}{2} \), consider the geodesic triangle with vertices \( (X_{w'}, X_{\rho_i}, X_{\rho_j}) \). Let
\[
\theta_{ij} \;=\; \ang{}{X_{w'}}{X_{\rho_i}}{X_{\rho_j}} \in [0,\pi]
\qquad\text{and notice that }\quad
B_{ij} \;=\; \langle \tilde{\rho_i}, \tilde{\rho_j} \rangle \;=\; \cos(\theta_{ij}).
\]

Now we invoke Lemma~\ref{lem: M-opposite-side-fixed-angle}, 
a comparison version of law of cosines, to get 
\begin{align*}
\Big|
\D{X_{\rho_i}}{X_{\rho_j}}^2 
- \big( \D{X_{w'}}{X_{\rho_i}}^2 
+ \D{X_{w'}}{X_{\rho_j}}^2 
- 2\,\D{X_{w'}}{X_{\rho_i}}\,\D{X_{w'}}{X_{\rho_j}} \cos\theta_{ij} \big)
\Big|
\;\le\;
C\,\kappa\,\big(\D{X_{w'}}{X_{\rho_i}}^4 + \D{X_{w'}}{X_{\rho_j}}^4\big).
\end{align*}

Under the event defining \(W'\) and \( \Eortho \), we have, for \( i,j\in[k] \),
\[
\D{X_{\rho_i}}{X_{\rho_j}} = \sqrt{2}r \pm \delta,
\qquad
\D{X_{w'}}{X_{\rho_i}} = \sqrt{2}r \pm 0.92\delta,
\]
and for pairs involving \(0\),
\[
\D{X_{\rho_0}}{X_{\rho_i}} = r \pm \delta,
\qquad
\D{X_{w'}}{X_{\rho_0}} = r \pm 0.92\delta,
\qquad
\D{X_{w'}}{X_{\rho_i}} = \sqrt{2}r \pm 0.92\delta.
\]
Plugging these into the display above and rearranging yields the following bounds.

\step{Case 1: \(i,j\in[k]\)}
Using \( (\sqrt{2}r)^2 = 2r^2 \) and collecting terms,
\[
4r^2\,(1-\cos\theta_{ij})
\;=\; 2r^2 \;+\; \Err_{ij},
\]
where
\(
|\Err_{ij}| \;\le\; C_1\big(\delta r + \kappa r^4\big).
\)
Hence
\[
\big|\cos\theta_{ij} - \tfrac{1}{2}\big|
\;\le\; C_2\Big(\frac{\delta}{r} + \kappa r^2\Big).
\]

\step{Case 2: one of \(i,j\) equals \(0\)}
Here
\[
3r^2 - 2\sqrt{2}\,r^2 \cos\theta_{0i}
\;=\; r^2 \;+\; \Err_{0i},
\qquad
|\Err_{0i}| \;\le\; C_3\big(\delta r + \kappa r^4\big),
\]
so
\[
\big|\cos\theta_{0i} - 2^{-1/2}\big|
\;\le\; C_4\Big(\frac{\delta}{r} + \kappa r^2\Big).
\]

\noindent On the diagonal we have \( B_{ii}=1 \) for all \( i \).

\medskip
By the parameter regime in Proposition~\ref{prop: scenerio-1-ortho} (namely \( \delta = c r \) and \( r \le c\,\rM \) so that \( \kappa r^2 \lesssim c^2 \)), there is a constant \(C_*\) such that
\[
\varepsilon \;:=\; C_*\Big(\frac{\delta}{r} + \kappa r^2\Big) \;=\; O \big(c_{\tref{prop: scenerio-1-ortho}}\big),
\]
and consequently
\[
B_{ij} \;=\;
\begin{cases}
1 & (i=j),\\[2pt]
\frac{1}{2} \pm \varepsilon & (i,j\in[k],\ i\neq j),\\[2pt]
\frac{\sqrt{2}}{2} \pm \varepsilon & (\text{one of } i,j \text{ is } 0).
\end{cases}
\]
Equivalently, \(B\) is a small perturbation (of size \(\varepsilon\)) of the deterministic Gram matrix with entries \(\{1,\,\tfrac{1}{2},\,\tfrac{\sqrt{2}}{2}\}\), i.e. angles \(\{\pi/3,\ \pi/4\}\) up to \(O(\varepsilon)\).

\step{Least singular value of the baseline matrix}
Let $\bar B$ be the $(k+1)\times(k+1)$ Gram matrix with
\[
\bar B_{ii} = 1,\quad 
\bar B_{0i} = \bar B_{i0} = \alpha := \tfrac{\sqrt{2}}{2} \ \ (i \ge 1),\quad
\bar B_{ij} = \beta := \tfrac12 \ \ (i \neq j,\ i,j \ge 1).
\]
We decompose $\mathbb{R}^{k+1}$ as the orthogonal sum
\[
H := \mathrm{span}\{e_0,\, (0,\tfrac1{\sqrt{k}}\mathbf 1)\}
\quad\text{and}\quad
H^\perp := \{(0,v) : \mathbf 1^\top v = 0\}.
\]

\noindent
On $H^\perp$ (dimension $k-1$), $\bar B$ acts by multiplication with 
$1-\beta = \tfrac12$, so there are $k-1$ eigenvalues equal to $\tfrac12$.
On $H$, with basis $\{e_0,\, (0,\tfrac1{\sqrt{k}}\mathbf 1)\}$, the restriction of $\bar B$ is
\[
D = 
\begin{pmatrix}
1 & \tfrac{\sqrt{2}}{2}\sqrt{k} \\
\tfrac{\sqrt{2}}{2}\sqrt{k} & \tfrac{k+1}{2}
\end{pmatrix}.
\]
This $2\times 2$ matrix has trace $\tfrac{k+3}{2}$ and determinant $\tfrac12$, hence its eigenvalues are
\[
\lambda_{\pm}(D) = \frac{ \tfrac{k+3}{2} \pm \sqrt{ \big(\tfrac{k+3}{2}\big)^2 - 2 } }{2}.
\quad \Rightarrow \quad
\lambda_{\min}(\bar B) = \lambda_{-}(D) 
= \frac{ \tfrac{k+3}{2} - \sqrt{ \big(\tfrac{k+3}{2}\big)^2 - 2 } }{2}
\ \ge\ \frac{1}{k+3}.
\]

\step{Perturbation to the actual $B$}
The actual Gram matrix satisfies $B = \bar B + E$, where
\[
|E_{ij}| \ \le\ \varepsilon, 
\quad \varepsilon = C \left(\frac{\delta}{r} + \kappa r^2\right)
= O \big(c_{\tref{prop: scenerio-1-ortho}}\big).
\]
Thus $\|E\|_{\mathrm{op}} \lesssim k\varepsilon$, and Weyl’s inequality gives
\[
\sigma_{\min}(B) \ \ge\ \sigma_{\min}(\bar B) - \|E\|_{\mathrm{op}}
\ \ge\ \frac{1}{k+3} - C'k\varepsilon \ \ge \ \frac{1}{2(k+1)}\,, 
\]
provided that $c_{\tref{prop: scenerio-1-ortho}}$ is chosen small enough. 
\end{proof}

\begin{lemma}
\label{lem: scenaro-1-ortho-z}
Under the assumptions of Proposition~\ref{prop: scenerio-1-ortho} and fix \(w' \in W'\).
Let \(\xi = \frac{\xi'}{\|\xi'\|} \in T_{X_{w'}}M\) be the unit direction chosen as in the construction above, i.e.
\[
\xi' \in {\mathrm{span}(\{\tilde{\rho_i}\}_{i \in [0,k]})}
\quad\text{satisfies}\quad
\langle \xi, \tilde{\rho_i} \rangle =\sigma_i \qquad \text{for all } i \in [0,k],
\]
with \(\tilde{\rho_i}=\tv{\rho_i}/\|\tv{\rho_i}\|\) and \(\exp_{X_{w'}}(\tv{\rho_i})=X_{\rho_i}\).
Set \(z=\exp_{X_{w'}}(t\,\xi)\) for $0 < t \le \delta$. 
When \(c_{\tref{prop: scenerio-1-ortho}}\) is chosen sufficiently small, we have 
\begin{equation}\label{eq:rmp-diff-avg}
\frac{\ell_\rmp}{2}\,\frac{t}{\|\xi'\|}
\;\le\;
\sigma_i\cdot\frac{1}{|U_{\rho_i}|}\sum_{v\in U_{\rho_i}}
\Big(\rmp(\D{z}{X_v}) - \rmp(\D{X_{w'}}{X_v})\Big)
\;\le\;
2L_\rmp\,\frac{t}{\|\xi'\|}.
\end{equation}
\end{lemma}

\begin{proof}
Fix $w' \in W'$ and let $\xi \in T_{X_{w'}}M$ be the unit vector chosen above.
For $t>0$ small, set
\[
    z := \exp_{X_{w'}}(t\,\xi).
\]
For each $i\in[0,k]$, we apply Lemma~\ref{lem: regularity-geometry} with
\[
    p := X_{w'}, \qquad x := X_{\rho_i}, \qquad q := z.
\]

\step{Verification of the conditions of Lemma~\ref{lem: regularity-geometry}}
We have $\D{p}{q} = \D{X_{w'}}{z} = t$ and
\[
    \big|\cos \big(\ang{}{p}{q}{x}\big)\big|
    = \big|\langle \xi, \tilde{\rho_i}\rangle\big|
    \ \ge\ \frac{1}{2(k+1)}
\]
by Lemma~\ref{lem:xi-prime-bound}.  
Moreover, by Remark~\ref{rem: scenaro-1-W'-distance} we have
\[
    \D{X_{w'}}{X_{\rho_i}} = 
    \begin{cases}
    r \pm 0.92\delta, & i=0,\\
    \sqrt{2}\,r \pm 0.92\delta, & i\ge 1,
    \end{cases}
\]
and both quantities are $O(r)$. Thus, provided $t \le \delta$ and $c_{\tref{prop: scenerio-1-ortho}}$ is chosen sufficiently small, the two conditions from Lemma~\ref{lem: regularity-geometry} are satisfied:
\[
    16\,\frac{t}{\D{X_{w'}}{X_{\rho_i}}} 
    \ \le\ \frac{1}{2(k+1)}
    \ \le\ 
    \big|\cos \big(\ang{}{X_{w'}}{z}{X_{\rho_i}}\big)\big|,
\]
and 
\[\D{X_{w'}}{X_{\rho_i}},\ t < \tfrac18\,\rM\,.\] 

\step{Choose $x' = X_v$ for $v \in U_{\rho_i}$}
Now we take $x'$ to be any $X_v$ with $v \in U_{\rho_i}$ in the setting of Lemma~\ref{lem: regularity-geometry}.

Within the event $\Eclu(U_{\rho_i}, \rho_i)$,
\[
    \D{X_{\rho_i}}{X_v} \le \eta
    \quad\text{for all}\quad v \in U_{\rho_i}.
\]
Since $\D{p}{x} = \D{X_{w'}}{X_{\rho_i}} \ge r - 0.92\delta$ and
$\eta = c_{\tref{prop: scenerio-1-ortho}}\,\delta$ with $c_{\tref{prop: scenerio-1-ortho}}$ small, we obtain
\[
    16\,\frac{\D{x}{x'}}{\D{p}{x}}
    \ \le\ 16\,\frac{\eta}{r - 0.92\delta}
    \ \le\ 16\,\frac{c_{\tref{prop: scenerio-1-ortho}}\,\delta}{r - 0.92\delta}
    \ \ll\ \frac{1}{2(k+1)}
    \ \le\ \big|\cos \big(\ang{}{p}{q}{x}\big)\big|,
\]
where the last inequality uses Lemma~\ref{lem:xi-prime-bound}.
Thus the second hypothesis of Lemma~\ref{lem: regularity-geometry} is also satisfied for this choice of $x'$.

\step{Consequence}
Recalling $\cos(\ang{}{p}{q}{x})=\langle \xi,\tilde{\rho_i}\rangle=\sigma_i/\|\xi'\|$, Lemma~\ref{lem: regularity-geometry} yields
\begin{equation}\label{eq:lem-scenaro-1-ortho-z-00}
\frac{\ell_\rmp}{2}\,\frac{t}{\|\xi'\|}
\;\le\;
\sigma_i\,\Big(\rmp(\D{q}{x'}) - \rmp(\D{p}{x'})\Big)
\;\le\;
2L_\rmp\,\frac{t}{\|\xi'\|}.
\end{equation}
Averaging~\eqref{eq:lem-scenaro-1-ortho-z-00} over $v\in U_{\rho_i}$ gives
\begin{align*}
\frac{\ell_\rmp}{2}\,\frac{t}{\|\xi'\|}
\;\le\;
\sigma_i\cdot\frac{1}{|U_{\rho_i}|}\sum_{v\in U_{\rho_i}}
\Big( \rmp(\D{z}{X_v}) -  \rmp(\D{X_{w'}}{X_v})\Big)
\;\le\;
2L_\rmp\,\frac{t}{\|\xi'\|}\,.
\end{align*}
\end{proof}

\begin{proof}[Proof of Proposition~\ref{prop: scenerio-1-ortho}]

Choose $t$ so that
\begin{equation}\label{eq:t-choice}
   \frac{\ell_\rmp}{2} \frac{t}{\|\xi'\|} 
   \;=\;  
    2 \fe{m} + L_\rmp c'_{\tref{prop: scenerio-1-ortho}} c_{\tref{prop: cluster-finding-algorithm}}\eta^2.
\end{equation}
By \eqref{eq: scenario-1-ortho-parameter-assumption} and $\|\xi'\| \le \sqrt{2(k+1)} \le \sqrt{2(d+1)}$ from Lemma~\ref{lem:xi-prime-bound}, we can assume this choice satisfies 
$$t  \le \frac{1}{2}c_{\tref{prop: cluster-finding-algorithm}}\eta^2 
\le 0.01\delta,$$ provided that $c_{\tref{prop: scenerio-1-ortho}}$ and $c'_{\tref{prop: scenerio-1-ortho}}$ are sufficiently small. In particular, 
it is admissible in Lemma~\ref{lem: scenaro-1-ortho-z}.

Let $z$ be the point from Lemma~\ref{lem: scenaro-1-ortho-z} corresponding to this $t$. And by Lemma~\ref{lem: scenaro-1-ortho-z},
It is rather straightforward to see that $z$ satisfies the conditions in Proposition~\ref{prop: scenerio-1-ortho}:
\[
|\D{X_{\rho_0}}{z} - r| < 0.92 \delta
\quad \text{and} \quad
|\D{X_{\rho_i}}{z} - \sqrt{2}r| < 0.92 \delta \quad (\forall i \in [k])\,,
\]
since $X_{w'}$ satisfies the same conditions with $0.92\delta$ replaced by $0.91\delta$ and $\D{X_{w'}}{z} = t \le 0.01\delta$.

Take any $w \in W$ with 
$$\D{z}{X_w} \le c_{\tref{prop: scenerio-1-ortho}}'c_{\tref{prop: cluster-finding-algorithm}}\eta^2\,,$$
whose existence is guaranteed by the definition of $W$.

\medskip
\step{Example case $i = 0$ and $\sigma_0 = 1$}
From $\Enavi{U_{\rho_0}}{w}\supseteq \Enavi{U_{\rho_0}}{W}$, we have
\begin{align*}
    \cn{U_{\rho_0}}{w} 
\ge 
    \acn{U_{\rho_0}}{w} - \fe{m}
\ge 
    \frac{1}{|U_{\rho_0}|} \sum_{v\in U_{\rho_0}} \rmp(\D{z}{X_v}) 
     - L_\rmp\,\D{z}{X_w} - \fe{m}\,.
\end{align*}
Then, invoke Lemma~\ref{lem: scenaro-1-ortho-z} with $i=0$ and $\sigma_0=1$ to get
\begin{align*}
   (*) \ge \acn{U_{\rho_0}}{w'}  + \frac{\ell_\rmp}{2} \frac{t}{\|\xi'\|} - L_\rmp\,\D{z}{X_w} - \fe{m}
   \ge \acn{U_{\rho_0}}{w'} +\fe{m} 
   \ge \cn{U_{\rho_0}}{w'} \ge \rmp(r+0.91\delta),
\end{align*}
where the second last inequality uses the definition of $W'$.
For the bound on the other side, notice $\sigma_0 = 1$ also implies $\cn{U_{\rho_0}}{w'} < \rmp(r)$, the analogous application of Lemma~\ref{lem: scenaro-1-ortho-z} to the upper bound yields
\[
    \cn{U_{\rho_0}}{w} \le \cn{U_{\rho_0}}{w'} + 2L_\rmp \frac{t}{\|\xi'\|} + L_\rmp\,\D{z}{X_w} + \fe{m}
    \le 
    \rmp(r) + 0.1\delta 
\]
when $c_{\tref{prop: cluster-finding-algorithm}}$ is sufficiently small.

\step{Conclusion}
We have shown that for $i=0$ and $\sigma_0=1$, $w$ satisfies the $U_{\rho_0}$ constraints in the definition of $W'$. The same argument works verbatim for all $i\in[0,k]$ and both signs of $\sigma_i$. Hence
\[
    \B{c'_{\tref{prop: scenerio-1-ortho}}c_{\tref{prop: cluster-finding-algorithm}}\eta^2}{z} \cap W \ \subseteq\ W'.
\]
Moreover, $\D{z}{X_w} \le t + c'_{\tref{prop: scenerio-1-ortho}}c_{\tref{prop: cluster-finding-algorithm}}\eta^2 \le 
c_{\tref{prop: cluster-finding-algorithm}}\eta^2$, so
\[
    |\B{c_{\tref{prop: cluster-finding-algorithm}}\eta^2}{z} \cap  W'| \
\ge\ |\B{c'_{\tref{prop: scenerio-1-ortho}}c_{\tref{prop: cluster-finding-algorithm}}\eta^2}{z} \cap W| 
    \ \ge\ m'.
\]
\end{proof}

\section{Scenario 2: Point Navigation via Approximate Orthogonal Basis}
This is a continuation of the previous scenario, where we will focus on point navigation once we established the approximate orthogonal basis. The main result of this section is the following proposition
\begin{prop}
\label{prop: scenario-2-navigation}
There exists a constant \(  c_{\tref{prop: scenario-2-navigation}} = c(d, L_\rmp, \ell_\rmp) > 0 \) such that the following holds.
Suppose $ 0 < c < c_{\tref{prop: scenario-2-navigation}}$, $0 < \eta < \delta < r$ and a positive integer $m$ satisfies the parameter configuration~\eqref{eq: parameter-configuration}: 
\begin{align}
\label{eq: scenario-2-navigation}
\fe{m} := \frac{\log n}{\sqrt{\sp  m}} < c\eta^2, \qquad
\eta = c \delta, \qquad
\delta = c r, \qquad
r \le c \rMrp\,.
\end{align}

Suppose we are given $d+1$ pairs  ${\bf F} = (\rho_i, U_{\rho_i})_{i \in [0,d]}$ with \(\rho_i \subseteq U_{\rho_i} \subset \bfV\) and $|U_{\rho_i}| \ge m$. Further, let $w_1,w_2 \in {\bf V}$ are two vertices disjoint from $U_{\rho_i}$ for all $i \in [0,d]$. Let 
\[
 \psi_{\bf F}(w_1,w_2) : = \sqrt{\sum_{i \in [d]} (\cn{U_{\rho_i}}{w_2}-\cn{U_{\rho_i}}{w_1})^2}\,.
\]
Then, on the event
\[
{\cal E}:= 
\Eortho(r, \delta, \eta, (\rho_i)_{i \in [0,d]})
\cap \bigcap_{i \in [0,d]} \Eclu(\eta,U_{\rho_i},\rho_i) 
\cap \bigcap_{i \in [0,d]} \Enavi{U_{\rho_i}}{\{w_1,w_2\}},
\]
the following implication holds: 
\begin{align*}
&&    &\max\{\cn{U_{\rho_0}}{w_1}, \cn{U_{\rho_0}}{w_2}\}
\ge  
    \rmp(9\delta) \\
&\Rightarrow&
&\frac{1}{3}\,L_\rmp^{-1}(\psi_{\bf F}(w_1,w_2) -  2\sqrt{d} \fe{m})\,
\ \le\ 
\D{X_{w_1}}{X_{w_2}}
\ \le\ 
3\,\ell_\rmp^{-1} (\psi_{\bf F}(w_1,w_2) +  2\sqrt{d} \fe{m})\,.
\end{align*}

\end{prop}
In other words, if we have two vertices $w_1$ and $w_2$ whose latent distances are sufficiently close to $X_{\rho_0}$, then we can estimate their latent distance 
using $(\cn{U_{\rho_i}}{w_j})_{i\in [0,k], j\in [1,2]}$ with a high precision.

\begin{lemma}
\label{lem: scenario-2-approximate-identity}
Let $0 < \delta < r$ satisfy the same condition as in Proposition~\ref{prop: scenario-2-navigation}.
Suppose $p, x_0, x_1, \ldots, x_d \in M$ satisfy
\[
    |\D{x_0}{x_i} - r| \le \delta\,, \quad
    \big| \D{x_i}{x_j} - \sqrt{2}\,r \big| \le \delta \ \ \forall\, i\neq j \in [d], \quad
    \D{p}{x_0} \le 10\delta\,. 
\]
Let 
\[
    \tv{x_i} = (\tv{x_i})_p, 
    \quad \vec{y}_i := \frac{\tv{x_i}}{\|\tv{x_i}\|} \in T_p M,
    \quad i \in [d],
\]
and define the Gram matrix
\[
    A = (a_{ij})_{i,j\in[d]} = (\langle \vec{y}_i, \vec{y}_j \rangle_p)_{i,j\in[d]}.
\]
Then
\[
    \|A - I_d\| \ \le\ C d\left(\frac{\delta}{r} + \kappa r^2\right)\,,
\]
for some universal constant $C>0$. 
\end{lemma}

\begin{proof}
    To present the proof without focus on the computation, we will use $C,C',C''$ to denote universal constants that may change from line to line.

\step{Off-diagonal entries}
For $i\neq j$, by definition
\[
\langle \vec{y}_i, \vec{y}_j \rangle_p = \cos(\ang{}{p}{x_i}{x_j}).
\]
By triangle inequality, we know that 
\begin{align*}
    |\D{p}{x_i} - r| \le \D{p}{x_0} + |\D{x_0}{x_i} - r|
    \le 10 \delta + \delta \le 11 \delta\,. 
\end{align*}
Applying Lemma~\ref{lem: M-opposite-side-fixed-angle} to the triangle $\triangle p x_i x_j$ gives
\[
\left| \cos(\ang{}{p}{x_i}{x_j}) - 
\frac{\D{p}{x_i}^2 + \D{p}{x_j}^2 - \D{x_i}{x_j}^2}{2\,\D{p}{x_i}\,\D{p}{x_j}}\right|
\ \le\ \frac{\kappa\,(r+11\delta)^4}{(r-11\delta)^2}
\ \le\ 2\,\kappa r^2.
\]
From the distance assumptions,
\[
\Big| \D{p}{x_i}^2 + \D{p}{x_j}^2 - \D{x_i}{x_j}^2\Big| \le C \delta r\,.
\]
Then,
\[
\left| \frac{\D{p}{x_i}^2 + \D{p}{x_j}^2 - \D{x_i}{x_j}^2}{2\,\D{p}{x_i}\,\D{p}{x_j}}\right|
\ \le\ \frac{C\,\delta r}{2(r-11\delta)^2}
\ \le\ C'\,\frac{\delta}{r}.
\]
Combining with the curvature error term above yields
\[
\big|\langle \vec{y}_i, \vec{y}_j \rangle_p\big| \ \le\ C''\Big(\frac{\delta}{r} + 2\,\kappa r^2\Big),
\qquad i\neq j\,,
\]

\step{Diagonal entries}
By construction $\|\vec{y}_i\|_p=1$, so $a_{ii}=1$ for all $i$.

\step{Norm bound}
The off-diagonal bound gives 
\[
\|A - I_d\| \ \le\ \|A - I_d\|_{\mathrm{HS}} 
\ \le\ C'' d\left(\,\frac{\delta}{r} + \,\kappa r^2\right),
\]
as claimed.
\end{proof}

\begin{cor}
\label{cor: scenario-2-ortho-matrix-A}
Under the assumptions of Lemma~\ref{lem: scenario-2-approximate-identity}, the vectors 
$\{\vec{y}_i\}_{i=1}^d$ form a basis of $T_pM$ and, for all $\vec{q} \in T_pM$,
\[
\left(1 - \varepsilon\right) \sum_{i=1}^d \langle \vec{y}_i, \vec{q} \rangle_p^2
\ \le\ \|\vec{q}\|^2
\ \le\ \left(1 + \varepsilon\right) \sum_{i=1}^d \langle \vec{y}_i, \vec{q} \rangle_p^2,
\]
where $\varepsilon := 2Cd\left(\frac{\delta}{r} + \kappa r^2\right) = \Theta(c_{\tref{prop: scenario-2-navigation}})$.
\end{cor}

\begin{proof}
Let $A$ be the Gram matrix from Lemma~\ref{lem: scenario-2-approximate-identity}. That lemma gives 
$$\|A-I_d\|\le Cd(\frac{\delta}{r}+\kappa r^2) :=\alpha \lesssim d c_{\tref{prop: scenario-2-navigation}}\,.$$ When $c_{\tref{prop: scenario-2-navigation}}$ is chosen small enough, we have $A$ is positive definite and $\{\vec{y}_i\}$ is a basis. For $\vec{q}=\sum_i q^i\vec{y}_i$ we have
$\|\vec{q}\|^2 = q^\top A q$ and $\sum_i \langle \vec{y}_i, \vec{q} \rangle_p^2 = q^\top q$.
The bound on $\|A-I_d\|$ implies $\|A^{-1}-I_d\|\le \frac{\alpha}{1-\alpha} \le 2\alpha$,
hence 
\[
|(q^\top q) - (q^\top A^{-1} q)| \le 2\alpha\,\|q\|^2,
\]
which is exactly the stated inequality with $\varepsilon=2\alpha$.
\end{proof}

\begin{proof}[Proof of Proposition~\ref{prop: scenario-2-navigation}]
Define
\[
{\bf b} = (b_i)_{i\in[d]} 
:= \left( \acn{U_{\rho_i}}{w_2} - \acn{U_{\rho_i}}{w_1} \right)_{i\in[d]}\,.
\]
We {\bf claim} that 
\begin{align}
\frac{1}{\sqrt{5}}\,L_\rmp^{-1} \|{\bf b}\|_2 
\ \le\ 
\D{X_{w_1}}{X_{w_2}}
\ \le\ 
\sqrt{5}\,\ell_\rmp^{-1}
\|{\bf b}\|_2
\end{align}
Once the claim is established, we can simply compare $\psi_{\bf F}(w_1,w_2)$ and $\|{\bf b}\|_2$.  

Let us fix an instance of $X_{U_{\rho_i}} = x_{U_{\rho_i}}$ for $i \in [0,d]$, $X_{w_j} =x_{w_j}$ for $j \in \{1,2\}$, and $ {\cal U}_{U_{\rho_i},\{w_1,w_2\}} = {\frak u}_{U_{\rho_i},\{w_1,w_2\}}$ such that the events described in Proposition~\ref{prop: scenario-2-navigation} holds. We want to apply Corollary~\ref{cor: scenario-2-ortho-matrix-A} with 
$$
    p = x_{w_1},\, q = x_{w_2},\, x_i = x_{\rho_i} \quad \forall i \in [0,d]\,.
$$

\step{From edge count to distance}
By the Calibration Lemma~\ref{lem:calibration-toolkit}, 
\begin{align*}
    \rmp(\D{x_{w_j}}{x_{0}})\ \ge\ 
    \cn{U_{\rho_0}}{w_j}\ - 2\eta  
    \ge
    \rmp(9\delta) - 2\eta  
    \ge \rmp(10\delta), 
\end{align*}
where the last inequality holds when $c_{\tref{prop: scenario-2-navigation}}$ is chosen small enough. Consequently,  
 we have $\D{x_{w_1}}{x_0} \le 10\delta$, which satisfies the assumptions of Lemma~\ref{lem: scenario-2-approximate-identity} (and its Corollary \ref{cor: scenario-2-ortho-matrix-A}). Now let us write $p = x_{w_1}$, $q = x_{w_2}$ and $x_i = x_{\rho_i}$.  

\step{Notation in $T_pM$}
For each $i\in[d]$ and $v\in U_{\rho_i}$, let
\[
\tv{x_i} := (\tv{x_i})_p,\qquad
\tv{x_v} := (\tv{x_v})_p
\mbox{ for } v\in U_{\rho_i},
\qquad
\tv{q} := (\tv{x_{w_2}})_p    \,.
\]
Define unit directions
\[
\vec{y}_i := \frac{\tv{x}_i}{\|\tv{x}_i\|},\qquad
\vec{y}_i^{\,v} := \frac{\tv{x}_v}{\|\tv{x}_v\|}\quad (v\in U_{\rho_i}),
\]
and set
\[
t_i := \langle \vec{y}_i, \tv{q}\rangle_p
= \| \tv{q}\|\,\cos \big(\ang{}{p}{q}{x_i}\big),
\qquad
{\bf t} := (t_i)_{i\in[d]}.
\]

\step{Equivalence between $\|{\bf t}\|_2^2$ and $\| \tv{q}\|^2$}
By Corollary~\ref{cor: scenario-2-ortho-matrix-A},
\begin{align}
\label{eq:pw-eq}
\Big(1-\varepsilon_{\ref{cor: scenario-2-ortho-matrix-A}}\Big) \sum_{i=1}^d t_i^2
\ \le\ \| \tv{q}\|^2\ \le\
\Big(1+\varepsilon_{\ref{cor: scenario-2-ortho-matrix-A}}\Big)\sum_{i=1}^d t_i^2,
\end{align}
where \( \varepsilon_{\ref{cor: scenario-2-ortho-matrix-A}} = \Theta(c_{\tref{prop: scenario-2-navigation}})\). 

\step{Split indices by angle}
Let 
\[
I := \Big\{\,i\in[d]:\ |\cos(\ang{}{p}{q}{x_i})|\ \ge\ \frac{1}{10\sqrt{d}}\,\Big\}.
\]
Then
\[
\sum_{i\notin I} t_i^2
= \| \tv{q}\|^2 \sum_{i\notin I}\cos^2(\ang{}{p}{q}{x_i})
\ \le\ d\cdot \frac{1}{100d}\,\| \tv{q}\|^2
= \frac{1}{100}\,\| \tv{q}\|^2,
\]
and together with~\eqref{eq:pw-eq} this yields
\begin{equation}\label{eq:tI-dominates}
\frac{1}{2} \sum_{i\in I} t_i^2 \ \le\ \| \tv{q}\|^2\ \le\ 2\sum_{i\in I} t_i^2\,,
\end{equation}
provided that $c_{\tref{prop: scenario-2-navigation}}$ is sufficiently small.

\step{Relating $b_i$ to $t_i$ for $i\in I$}
Fix $i\in I$.
Because $|\cos(\ang{}{p}{q}{x_i})|\ge \frac{1}{10\sqrt d}$, and for all $v\in U_{\rho_i}$ we have 
$\D{x_i}{x_v}\le \eta$ and $\D{p}{x_i}\asymp r$, the hypotheses of Lemma~\ref{lem: regularity-geometry}
(with $p,q,x=x_i,x'=x_v$) are satisfied. Using its second conclusion,
\[
\frac{\ell_\rmp}{2} \ \le\
\frac{\rmp \big(\D{q}{x_v}\big)-\rmp \big(\D{p}{x_v}\big)}{\cos(\ang{}{p}{q}{x_i})\,\D{p}{q}}
\ \le\ 2L_\rmp.
\]
Averaging over $v\in U_{\rho_i}$ and recalling $t_i=\cos(\ang{}{p}{q}{x_i})\,\D{p}{q}$,
\begin{equation}\label{eq:bi-ti-I}
\frac{\ell_\rmp}{2}\,|t_i|
\ \le\ |b_i| \ \le\ 2L_\rmp\,|t_i|,
\qquad i\in I.
\end{equation}
Summing~\eqref{eq:bi-ti-I} over $i\in I$ and combining with~\eqref{eq:tI-dominates},
\begin{equation}\label{eq:q2-vs-bI2}
\frac{1}{4}L_\rmp^{-2}\sum_{i\in I} b_i^2 \ \le\ \| \tv{q}\|^2 \ \le\ 4\,\ell_\rmp^{-2}\sum_{i\in I} b_i^2.
\end{equation}

\step{Bounding the complement $i\notin I$}
Fix $i\notin I$ and $v\in U_{\rho_i}$. 
Applying Corollary~\ref{cor: M-opposite-side} to the triangle $\Delta p q x_v$, 
obtaining
\begin{align*}
    \big|\D{p}{x_v}-\D{q}{x_v}\big|
\le & 
    |\cos(\ang{}{p}{q}{x_v})|\cdot \D{p}{q} + 4 \frac{\D{p}{q}^2}{\D{p}{x_v}} \\
\le & 
    (|\cos(\ang{}{p}{q}{x_i})| + \ang{}{p}{x_i}{x_v}) \cdot \D{p}{q} + 4 \frac{\D{p}{q}^2}{\D{p}{x_v}} 
\end{align*}
Using triangle comparison from Lemma~\ref{lem: tri_lem}, and then Lemma~\ref{lem:spherical-angle} yields
\begin{align*}
   \ang{}{p}{x_i}{x_v}
\le 
   \ang{\kappa}{p}{x_i}{x_v}
\le 
   2\frac{\D{x_i}{x_v}}{\D{p}{x_i}}\,.
\end{align*}
Together with $|\cos(\ang{}{p}{q}{x_i})| \le \frac{1}{10\sqrt{d}}$,  
\[
\big|\D{p}{x_v}-\D{q}{x_v}\big|
\ \le\
\Big(|\cos(\ang{}{p}{q}{x_i})| + 2\,\tfrac{\D{x_i}{x_v}}{\D{p}{x_i}}
+ 4\,\tfrac{\D{p}{q}}{\D{p}{x_v}}\Big)\,\D{p}{q}
\le 
\Big(\frac{1}{10\sqrt{d}}+ \Theta(c_{\tref{prop: scenario-2-navigation}})\Big)\,\D{p}{q}
\ \le\ \frac{1}{5\sqrt d}\,\D{p}{q},
\]
under the small $c_{\tref{prop: scenario-2-navigation}}$ assumption (using $\D{x_i}{x_v}\le \eta$, $\D{p}{x_i}\gtrsim r$, and $\D{p}{q} \le \delta$).
By Lipschitzness of $\rmp$,
\[
\big|\rmp \big(\D{q}{x_v}\big)-\rmp \big(\D{p}{x_v}\big)\big|
\ \le\ L_\rmp\,\frac{1}{5\sqrt d}\,\| \tv{q}\|.
\]
Averaging over $v\in U_{\rho_i}$ gives
\begin{equation}\label{eq:bIcomp}
|b_i| \ \le\  L_\rmp\,\frac{1}{5\sqrt d}\,\| \tv{q}\|,
\qquad i\notin I,
\qquad \Rightarrow\qquad
\sum_{i\notin I} b_i^2 \ \le\ \frac{L_\rmp^2}{25}\,\| \tv{q}\|^2.
\end{equation}

\step{Assembling the bounds for the claim}
From~\eqref{eq:q2-vs-bI2} and~\eqref{eq:bIcomp},
\[
\| \tv{q}\|^2\ \ge\ \frac{1}{4}L_\rmp^{-2}\sum_{i\in I} b_i^2
\ \ge\ \frac{1}{4}L_\rmp^{-2}\left(\sum_{i=1}^d b_i^2 - \frac{L_\rmp^2}{25}\,\| \tv{q}\|^2\right).
\]
Rearranging,
\[
\left(1+\frac{1}{100}\right)\| \tv{q}\|^2 \ \ge\ \frac{1}{4}L_\rmp^{-2}\sum_{i=1}^d b_i^2
\quad\Rightarrow\quad
\| \tv{q}\|^2 \ \ge\ \frac{1}{5}\,L_\rmp^{-2}\sum_{i=1}^d b_i^2.
\]
The upper bound follows from~\eqref{eq:q2-vs-bI2} and $\sum_{i\in I} b_i^2 \le \sum_{i=1}^d b_i^2$:
\(
\| \tv{q}\|^2 \ \le\ 4\,\ell_\rmp^{-2}\sum_{i=1}^d b_i^2.
\)
Therefore, the claim holds. 

\step{Comparing $\sqrt{\sum_{i \in [d]}b_i^2}$}
Within the event $\Enavi{U_{\rho_i}}{w_j}$, we have 
\begin{align*}
    |b_i - \underbrace{(\cn{U_{\rho_i}}{w_2}- \cn{U_{\rho_i}}{w_1})}_{:= a_i}| 
 &
\le 
    2 \fe{m}\,,
\end{align*}
Viewing ${\bf b} = (b_i)_{i\in[d]}$ and ${\bf a} = (a_i)_{i\in[d]}$ as vectors in $\mathbb{R}^d$, the standard triangle inequality gives
$$
    |\|{\bf b}\|_2 - \|{\bf a}\|_2| \le 2\sqrt{d} \fe{m}
    \quad \Leftrightarrow \quad
    |\|{\bf b}\|_2 - \psi_{\bf F}(w_1,w_2)|
    \le 2\sqrt{d} \fe{m}\,.
$$

\end{proof}

\newpage

\begin{algorithm}[H]
\small
\caption{\texttt{cluster-net}}
\label{alg:cluster-net}

\SetKwInOut{Input}{Input}
\SetKwInOut{Output}{Output}

\Input{
$V_{\tr{cn}}$, $V_{\mathrm{net}}, V_{\mathrm{ortho}} \subseteq \bfV$ with $V_{\mathrm{net}}\cap V_{\mathrm{ortho}}=\varnothing$, $|V_{\mathrm{net}}|=|V_{\mathrm{ortho}}|=n$. \\
Parameters $r,\delta,\eta$ coupled by the standing choice with $c_{\tref{prop:cluster-net-algorithm}}$. \\
Integers $m,s_{\max}$ (minimum cluster size; max \# orthogonal clusters to extract). 
}
\Output{
Clusters $\{(U_i,u_i)\}_{i\in[k]} \subseteq V_{\mathrm{net}}$  (a $2\delta$-net). \\
Clusters $\{(V_j,v_j)\}_{j\in[s]} \subseteq V_{\mathrm{ortho}}$. \\
Maps $\axis_i:[d]\to[\ell]$ for $i\in[k]$ (orthogonal frames).
}
{\em Initialization:} 
$t \gets 1$;
$k \gets 1$;$\ell \gets 1$;
$(U_1,u_1) \gets {\rm Algorithm\,} \ref{alg:cluster}(V_{\mathrm{cn}},V_{\mathrm{net}},\eta)$;\\
\BlankLine

\While{true}{
    $t \gets t + 1$;\\
\If{ $\axis_k$ is not defined for all $s \in [d]$}
{
$s \gets$ the smallest index such that $\axis_k(s)$ is not defined.  
Define the candidate set
\[
\mathcal{J}_{k,s} \;:=\; 
\Big\{\, j\in[\ell]\ :\
\begin{aligned}
&\rmp(r+0.93\delta)\ \le\ {\rm n}(u_k,v_j)\ \le\ \rmp(r-0.93\delta),\\
&\rmp(\sqrt{2}r+0.93\delta)\ \le\ {\rm n}(v_{\axis_k(\alpha)},v_j)\ \le\ \rmp(\sqrt{2}r-0.93\delta)\quad \forall\,\alpha\in[s-1]
\end{aligned}
\Big\}.
\]
\If{$\mathcal{J}_{k,s}\neq\varnothing$}{
    pick any $j^\star\in\mathcal{J}_{k,s}$; set $\axis_k(s)\gets j^\star$; 
}
\Else{
    Define
    $$
    W_{k,s} = W'\big(r,\delta; V_{{\rm ortho}} \setminus L_{\preceq t-1}
     (U_k,u_k), (V_{\axis_k(1)},v_{\axis_k(1)}),\ldots,(V_{\axis_k(s-1)},v_{\axis_k(s-1)})\big)\,.
$$
from \eqref{eq: scenario-1-ortho-W'}.\\
$(V_{\ell+1},v_{\ell+1}) \gets {\rm Algorithm\,} \ref{alg:cluster}(V_{\tr{cn}} , W_{k,s},\eta)$;\\
$\ell \gets \ell + 1$.
}
}
\Else{
    Define the candidate set  
    $$
        {\cal V}_{k+1} = \bigg\{ (v,i_v) \in (V_{\mathrm{net}} \setminus L_{\preceq t-1}) \times [k] :
            \cn{U_i}{v} \le \rmp(1.1\delta)
            \quad \forall i \in [k] \quad \mbox{and } 
            \cn{U_{i_v}}{v} \ge \rmp(1.9\delta) 
        \bigg\}
    $$
    \If{${\cal V}_{k+1} \neq \varnothing$}
    {
        pick any $(u^*,i^*) \in {\cal V}_{k+1}$; 
        $u_{k+1} \gets u^*$;
        Define 
        \begin{align*}
           U^* := 
        \Big\{
            v \in V_{\mathrm{net}} \setminus L_{\preceq t-1} :
             \psi_{i^*}(u_{k+1},v) \le  \frac{\ell_p}{6} \eta   \mbox{ and }  N_{U_{i^*}}(v) \ge \rmp(3\delta)
        \Big\}
        \end{align*}
        where 
        \begin{align*}
           \psi_{i^*}(u_{k+1},v) := \sqrt{\sum_{s \in [d]} (\cn{V_{\axis_{i^*}(s)}}{v}-\cn{V_{\axis_{i^*}(s)}}{u_{k+1}})^2}\,.
        \end{align*}
        $(U_{k+1},u_{k+1}) \gets (U^*,u^*)$;\\
        $k \gets k + 1$;
    }
    \Else{
        \Return $((U_i,u_i), \axis_i)_{i \in [k]}$ and $((V_j,v_j), j \in [\ell])$.
    }
}
}

\end{algorithm}

\section{Stage 1: Algorithm \texttt{cluster-net} for extracting clusters nets with orthogonal frames}

In this section we introduce the algorithm~\texttt{cluster-net}, the goal of the algorithm is to extracts which form a $2\delta$-net of $M$ where each element in a net are equipped with an orthogonal frame of clusters.
The main result of this section is the following:
\begin{prop}
\label{prop:cluster-net-algorithm}
There exists a constant $0< c_{\tref{prop:cluster-net-algorithm}} \le \min\!\big\{c_{\tref{prop: scenerio-1-ortho}},\,c_{\tref{prop: scenario-2-navigation}}\big\}$ so that the following holds. For any $0 < c_\star \le  c_{\tref{prop:cluster-net-algorithm}}$,

Set the parameters
\[
    r \;:=\; c_{\tref{prop:cluster-net-algorithm}}\, \rG 
    \qquad 
    \delta \;=\; c_{\tref{prop:cluster-net-algorithm}}\, r,
    \qquad 
    \eta \;=\; c_{\tref{prop:cluster-net-algorithm}}\, \delta,
\]
and
\[   
    m := \mu_{\min}\big( c'_{\tref{prop: scenerio-1-ortho}}c_{\tref{prop: cluster-finding-algorithm}}\eta^2/3\big) \frac{n}{2}\,.
\]
   Suppose $V_{\tr{cn}}$, $V_{\tr{net}}$, and $V_{\tr{ortho}}$ are disjoint vertex subsets of ${\bf V}$ with  
   \begin{align}
    \label{def: n-cn}
    |V_{\tr{cn}}| = n_{\tr{cn}} := \bigg\lceil \frac{C^4_{\tref{prop: cluster-finding-algorithm}}\log^2(n)}{\sp^2\eta^4}  \bigg\rceil
    = \Theta(  \sp^{-2}\log^2n)
    \qquad \text{ and } \qquad 
    |V_{\tr{net}}| = |V_{\tr{ortho}}| = n\,, 
   \end{align}

   Then with high probability, the algorithm~\texttt{cluster-net} extracts a $2\delta$-net 
    $(U_i,u_i)_{i \in[k]}$ from $V_{\mathrm{net}}$ and a collection of orthogonal clusters $(V_j,v_j)_{j\in[\ell]}$ from $V_{\mathrm{ortho}}$, such that the following holds: 
    \begin{enumerate}
        \item The set $\{u_k\}_{k\in[k]}$ form a $2\delta$-net of $M$ with $\delta$-separation: 
        \begin{align*}
            \forall p \in M, \exists i \in [k] \text{ such that } \D{p}{X_{u_i}} \le 2\delta, \quad \text{ and } \quad \D{X_{u_i}}{X_{u_j}} > \delta, \forall i \neq j \in [k]\,.
        \end{align*}
           
        \item For each $i \in [k]$ and $j \in [\ell]$, the cluster $(U_i,u_i)$ satisfies the event $\Eclu(\eta,U_{i},u_i)$ with $|U_i| \ge m$ and, similarly, $\Eclu(\eta, V_j,v_j)$ hold with $|V_j| \ge m$.
        \item For each $i \in [k]$, the clusters $(U_i,u_i),(V_{\axis_i(1)},v_{\axis_i(1)}),\ldots,(V_{\axis_i(d)},v_{\axis_i(d)})$ satisfies the events  
        \begin{align*}
            \Eortho(r,\delta,\eta, (u_i, v_{\axis_i(1)},\ldots,,v_{\axis_i(d)})).
        \end{align*}
    \end{enumerate}
    
\end{prop}

\begin{rem}
The choice of $m$ was arrives from Proposition~\ref{prop: scenerio-1-ortho}. And the choice of $|V_{\tr{cn}}|$ is due to~\eqref{eq: cluster-finding-eta-assumption-Usize} from Proposition~\ref{prop: cluster-finding-algorithm}. 
\end{rem}
\begin{rem}
We further remark that the choice of $c_{\tref{prop:cluster-net-algorithm}},r,\delta,\eta$ and $m$ satisfy the assumptions of Propositions~\ref{prop: scenerio-1-ortho} and~\ref{prop: scenario-2-navigation}.
\end{rem}

\paragraph{Algorithm~\texttt{cluster-net} Outline}
The algorithm attempts to extract clusters in a sequential process.  
We index the steps by 
$${\rm t}=1,2,\dots,T,$$
where for concreteness we may assume $T \le n(d+1)$.  
At each step we might extract a new cluster from $V_{\mathrm{net}}$ or from $V_{\mathrm{ortho}}$. 

At the end of the process we obtain two sequences of clusters:
\[
\{(U_i,u_i)\}_{i\in[k]}, \quad k \le d, \qquad U_i\subseteq V_{\mathrm{net}},
\]
and
\[
\{(V_j,v_j)\}_{j\in[\ell]}, \qquad V_j\subseteq V_{\mathrm{ortho}}.
\]

\step{Cluster property}
By a \emph{cluster} $(U_i,u_i)$ we mean that the cluster event
\[
\Eclu(\eta,U_i,u_i)
\]
holds; i.e., every $v\in U_i$ satisfies $\D{X_{u_i}}{X_v}\le \eta$.  
Similarly, for each $(V_j,v_j)$ we assume $\Eclu(\eta,V_j,v_j)$ holds.

\step{Orthogonal frames}
For each $(U_i,u_i)$ we construct an index map
\[
\axis_i : [d] \longrightarrow [\ell],
\]
such that the collection
\[
(U_i,u_i),\ (V_{\axis_i(1)},v_{\axis_i(1)}),\ \ldots,\ (V_{\axis_i(d)},v_{\axis_i(d)})
\]
forms an \emph{orthogonal frame of clusters}. Precisely, they satisfy the event
\[
\Eortho \Big(r,\delta,\eta,m;\ (U_i,u_i),\ (V_{\axis_i(1)},v_{\axis_i(1)}),\ldots,(V_{\axis_i(d)},v_{\axis_i(d)})\Big).
\]
Moreover, the set of centers $\{u_i\}_{i\in[k]}$ forms a $2\delta$-net of $M$.

\paragraph{Extraction times and notation.}
For each $u_i$ (resp.\ $v_j$) we denote by $\st(u_i)$ (resp.\ $\st(v_j)$) the step at which it is extracted.  
We set
\[
L_{\st} :=
\begin{cases}
    U_i, & \text{if } \st(u_i)=\st, \\
    V_j, & \text{if } \st(v_j)=\st, \\
    \varnothing, & \text{if no cluster is extracted at step } \st,
\end{cases}
\qquad
\ell_{\st} :=
\begin{cases}
    u_i, & \text{if } \st(u_i)=\st, \\
    v_j, & \text{if } \st(v_j)=\st\,.\\
    \varnothing, & \text{if no cluster is extracted at step } \st.
\end{cases}
\]
Thus $(L_{\st},\ell_{\st})$ is the cluster extracted at step $\st$ (when one exists).  
We also define the revealed sets
\[
L_{\prec \st} := \bigcup_{s<\st} L_s,
\qquad
L_{\preceq \st} := \bigcup_{s\le \st} L_s.
\]

\step{Edge revealing process}
Initially the algorithm only observes the edges connecting $V_{\tr{cn}}$ to $V_{\mathrm{net}}$ and $V_{\mathrm{ortho}}$.  
Whenever at step $\st$ a cluster $L_{\st}$ is extracted, we reveal all edges from $L_{\st}$ to 
\[
V_{\mathrm{net}} \cup V_{\mathrm{ortho}} \setminus L_{\preceq \st-1}.
\]
(Note that edges from $L_\st$ to $L_{\preceq \st-1}$ have already been revealed when previous clusters were extracted.)  
This convention makes independence transparent: {\it The data generated up to step $\st$ does not rely on edges that have not been revealed yet, so that these data are independent of the edges that will be revealed in the future.}
In particular, whenever we reveal the edges from $L_\st$ to $V_{\mathrm{net}}\cup V_{\mathrm{ortho}}\setminus L_{\preceq \st-1}$, we expect the fluctuation event
\[
\Enavi{L_\st}{V_{\mathrm{net}} \cup V_{\mathrm{ortho}} \setminus L_{\preceq \st-1}}
\]
to hold. This means that for every $v\in V_{\mathrm{net}}\cup V_{\mathrm{ortho}}\setminus L_{\preceq \st-1}$, if $\cn{L_\st}{v} \ge \rmp(\rG)$, then $\rmp^{-1}(\cn{L_\st}{v})$ estimate the distance  
of $\ell_\st$ to $v$ up to an error $2\eta$. 

\step{Distance between centers}
Given extracted cluster centers $u_k,v_i,v_j$, we define 
\begin{align}
{\rm n}(u_k,v_j)
:= 
\begin{cases}
\cn{U_{u_k}}{v_j}, & \text{if } \st(u_k) < \st(v_j),\\[2pt]
\cn{V_j}{u_k},     & \text{if } \st(v_j) < \st(u_k).
\end{cases}
\quad \mbox{ and } \quad
{\rm n}(v_i,v_j)
:=
\begin{cases}
\cn{V_i}{v_j}, & \text{if } \st(v_i) < \st(v_j),\\[2pt]
\cn{V_j}{v_i}, & \text{if } \st(v_j) < \st(v_i).
\end{cases}
\end{align}
That is, ${\rm n}(u_k,v_j)$ is always measured using the \emph{earlier} cluster’s vertex set, so that the corresponding fluctuation event (of the form $\Enavi{L_{\st(\,\cdot\,)}}{\cdot}$) was assumed, but not the later, to avoid justification of independence.

\step{Step $\st$ of the algorithm}
At step $\st$, suppose we have extracted $(U_i,u_i)$ for $i\in[k]$ and $(V_j,v_j)$ for $j\in[\ell]$, where $k = k(\st)$ and $\ell = \ell(\st)$ are the number of clusters extracted so far from $V_{\mathrm{net}}$ and $V_{\mathrm{ortho}}$, respectively. 
If $\axis_k$ is not defined for all $s\in[d]$. Then the goal of $\st$-th step is to define $\axis_k(s)$ where $s$ is the smallest index such that $\axis_k(s)$ is not yet defined:
\begin{itemize}
\item It first checks whether there exists a cluster $(V_j,v_j)$ already extracted from $V_{\mathrm{ortho}}$ that is a valid candidate for $\axis_k(s)$. If so, we set $\axis_k(s)=j$.  
\item If not, then we extract a new cluster $(V_\ell,v_\ell)$ from $V_{\mathrm{ortho}}$ and set $\axis_k(s)=\ell$. (This corresponds to \emph{Scenario 1}.)  
\end{itemize}
On the other hand, if $\axis_k$ is already defined for all $s\in[d]$, then the goal of $\st$-th step is to extract a new cluster from $V_{\mathrm{net}}$. This new cluster must be at least $2\delta$ away from all previously extracted clusters in $V_{\mathrm{net}}$, but within $3\delta$ of some previously extracted cluster in $V_{\mathrm{net}}$. (This corresponds to \emph{Scenario 2}.)


\step{Proof roadmap}
We define a sequence of ``good'' events $\{\Omega(\st)\}_{\st\ge 0}$ adapted to the sequential extraction, where
$\Omega(0)$ summarizes the initial conditions, and $\Omega(\st)$ collects the desired properties after step $\st$
(e.g., cluster validity, edge fluctuation bounds on all revealed edges, and the orthogonal frame properties).
Write $\Omega(\preceq \st):=\bigcap_{t'\le \st}\Omega(t')$.
We first show that
\[
\Pr\!\big(\Omega(0)\big)\;=\;1-n^{-\omega(1)}.
\]
Inductively, assume $\Omega(\preceq t)$ holds at the beginning of step $t+1$. If the algorithm does \emph{not} terminate at step $\st+1$, we prove the one-step estimate
\begin{equation}\label{eq:one-step}
\Pr\!\big(\Omega(\st+1)\mid \Omega(\preceq \st)\big)\;=\;1-n^{-\omega(1)}.
\end{equation}
If the algorithm \emph{does} terminate at step $\st+1$, then $\Omega(\preceq \st)$ will imply that the output satisfies the conclusion of the proposition.

Since in each iteration we either (i) extend a frame coordinate and/or (ii) add a new net cluster, the procedure terminates after at most $T\le n(d+1)$ steps. A union bound over the at most $T$ exposures in \eqref{eq:one-step} yields
\[
\Pr\!\big( \Omega(\preceq T)\big)
\;\ge\;
1 - T\cdot n^{-\omega(1)}
\;=\;
1 - n^{-\omega(1)},
\]
Consequently, with probability $1-n^{-\omega(1)}$ either the algorithm terminates early with the stated output, or it maintains $\Omega(\st)$ at every step up to termination, proving the proposition.

\step{Definition of the events $\Omega(\st)$}
Now let us define the events $\Omega(\st)$ more precisely.
First, the initial event $\Omega(0)$ is defined as follows:
$$
    \Omega(0) := \Ecn{V_{\tr{cn}}}{V_{\tr{net}}} \cap \Ecn{V_{\tr{cn}}}{V_{\tr{ortho}}} \cap \Ept{V_{\tr{net}}} \cap \Ept{V_{\tr{ortho}}}\,.
$$
As for $\Omega(\st)$, it contains two components: 
The first event concerns about a new extracted cluster
\begin{align}
    \label{eq: def-Oclu-t}
    \Oclu{\st} 
    := &
    \begin{cases}
        \{|L_\st| \ge m \}\cap
        \Eclu(\eta, L_\st,\ell_\st) \cap
        \Enavi{L_\st}{V_{\mathrm{net}} \cup V_{\mathrm{ortho}} \setminus L_{\preceq \st-1}} & \text{if } L_\st \neq \varnothing\,,\\
        {\rm True} & \text{otherwise.} 
    \end{cases}
\end{align}
The second event concerns the orthogonal frame property:
\begin{align}
    \label{eq: def-Oortho-t}
\Oortho{\st}:= \begin{cases}
    \Eortho(r,\delta,\eta,(u_{k},v_{\axis_{k}(1)}, \dots, v_{\axis_k(s)})_{j \in [s]})
    & \text{if } \axis_k(s) \text{ is defined at step }\st\,,\\
    {\rm True} & \text{otherwise.}
\end{cases}
\end{align}
The third event is a net property that is only checked when the algorithm terminates:
\begin{align}
    \label{eq: def-Onet-t}
    \Onet{\st} :=
    \begin{cases}
        \Big\{\{u_i\}_{i\in[k]} \text{ is a $2\delta$-net of $M$}\Big\} & \text{if the algorithm terminates at step }\st\,,\\
        {\rm True} & \text{otherwise.}
    \end{cases}
\end{align}

The $\Omega(\st)$ event is then defined as the intersection of these three events:
\begin{align*}
    \Omega(\st) := \Oclu{\st} \cap \Oortho{\st} \cap \Onet{\st},
\end{align*}

\step{Remark on randomness and data structure}
It is worth to remark that all $(U_i,u_i)$, $(V_j,v_j)$, and $\axis_i$ are all random. In particular, the revealing order matters. 
To analyze the algorithm, we will define ${\cal D}_{\st}$ to be a set of possible data that can be generated by the algorithm up to step $t$: 
\begin{defi}
    \label{def: data-set-Dt} 
    A data set ${\cal D}_{\st}$ is a collection of tuples 
$$
    {\cal D}_{\st} = \Big\{ 
    (U_i^\sharp,u_i^\sharp, \st_{u_i^\sharp})_{i \in [k^\sharp]}\,,\quad (V_j^\sharp,v_j^\sharp, \st_{v_j^\sharp})_{j \in [\ell^\sharp]}\,,\quad \axis_i^\sharp : [d] \to [\ell^\sharp]\,,\quad i \in [k^\sharp-1]\,, \axis_{k^\sharp}^\sharp : [s^\sharp] \to [\ell^\sharp] \Big\}\,.
$$
\end{defi}
Further, we say that a data set ${\cal D}_{\st}$ is compatible with the Algorithm \ref{alg:cluster-net} at step $\st$ if  
\begin{enumerate}
    \item $k(t) = k^\sharp$ and $\ell(t) = \ell^\sharp$;
    \item $(U_i,u_i) = (U_i^\sharp,u_i^\sharp)$ and $t(u_i) = t_{u_i^\sharp}$ for all $i \in [k^\sharp]$;
    \item $(V_j,v_j) = (V_j^\sharp,v_j^\sharp)$ and $t(v_j) = t_{v_j^\sharp}$ for all $j \in [\ell^\sharp]$.
    \item $\axis_i = \axis_i^\sharp$ for all $i \in [k^\sharp-1]$ and $\axis_k = \axis^\sharp_{k^\sharp}$ for all $i \in [s^\sharp]$. In particular, $\axis_k$ is only defined for $[s^\sharp]$. 
\end{enumerate}

Here, $k^\sharp$ and $\ell^\sharp$ indicates the number of clusters extracted from $V_{\mathrm{net}}$ and $V_{\mathrm{ortho}}$, respectively, and $s^\sharp$ indicates the number of coordinates of the orthogonal frame that have been defined for the last extracted cluster from $V_{\mathrm{net}}$.
For each $i\in[k^\sharp]$ and $j\in[\ell^\sharp]$, the pair $(U_i^\sharp,u_i^\sharp)$ (resp.\ $(V_j^\sharp,v_j^\sharp)$) is a cluster extracted from $V_{\mathrm{net}}$ (resp.\ $V_{\mathrm{ortho}}$) at step $\st_{u_i^\sharp}$ (resp.\ $\st_{v_j^\sharp}$).

\step{Basic properties}
Before we proceed, let us make a few observations.
\begin{fact}
    \label{fact: Omega-prec-t}
    From the Algorithm \ref{alg:cluster-net}, whenever $L_\st$ is extracted at step $\st$, we have
    $$
        L_{\st+1} \subseteq V_{\mathrm{net}} \cup V_{\mathrm{ortho}} \setminus L_{\preceq \st}\,,
    $$
    whenever $L_{t+1} \neq \varnothing$. 
\end{fact}

Before proceeding, we also make a remark on the running time. 
\begin{rem}
    \label{rem: cluster-net-running-time}
   First, we remark that given $\delta$ is a constant. The $\delta$-separation set of $M$ has size at most $O(1)$ (independent of $n$). Therefore, the algorithm \texttt{cluster-net} will terminate in $O(1)$ steps.
   In each iteration, the most time-consuming step is to run the cluster-finding algorithm (Algorithm \ref{alg:cluster}) which takes time $O(n^2 (\log(n) + \sp^2|V_{\tr{cn}}|)) = O(n^2 \log^2 n)$. Therefore, the total running time of Algorithm \ref{alg:cluster-net} is $O(n^2 \log^2 n)$.
\end{rem}

\subsection{Main Technical Lemma}
This subsection is devoted to prove the following lemma. 
\begin{lemma}
    For each step $\st \ge 0$, 
        \begin{align*}
        \Pr\big(\Omega(\st+1) \mid \Omega(\preceq t)\big) = 1 - n^{-\omega(1)}\,.
    \end{align*}
\end{lemma}
    First of all, if the algorithm is already terminated before or at $\st$, then $\Omega(\st+1)$ is trivially true. Therefore, we only need to consider the case that the algorithm does not terminate at step $\st$. 
    
    \step{A realization of partial information of the graph}
    From now on, we assume that the algorithm does not terminate at step $\st$ and we fix a data structure ${\cal D}_{\st}$, we will first condition on the event $\Omega(\preceq \st)$ and the data set ${\cal D}_{\st}$. 
    If we run the algorithm with $\st$ iterations, then the edges that have been revealed are 
    \begin{align*}
       \Edge{V_{\tr{cn}}}{V_{\mathrm{net}}\cup V_{\tr{ortho}}}
    \cup 
        \bigcup_{\alpha \le \st} \Edge{L_\alpha}{V_{\mathrm{net}} \cup V_{\mathrm{ortho}} \setminus L_{\preceq \alpha-1}}\,,
    \end{align*}
    which are functions of the independent random variables
    \begin{align*}
        &&&X:= (X_{V_{\tr{cn}}}, X_{V_{\mathrm{net}}}, X_{V_{\mathrm{ortho}}}), \quad
        \quad {\cal U}_0 :=
         \{{\cal U}_{u,v}\,:\, u \in V_{\tr{cn}}, v \in V_{\mathrm{net}} \cup V_{\mathrm{ortho}}\} \\
        \mbox{ and } &&&
        {\cal U}_\alpha := \{{\cal U}_{u,v}\,:\, u \in L_\alpha, v \in V_{\mathrm{net}} \cup V_{\mathrm{ortho}} \setminus L_{\preceq \alpha-1}\}\,,
        \mbox{ for } \alpha \in [\st]\,.
    \end{align*}
    We remark that the sets are $L_\alpha$ are all determined by ${\cal D}_{\st}$ for $\alpha \le \st$. 
    At the same time, $\Omega(\preceq \st)$ is an event of $X, {\cal U}_0, {\cal U}_1,\ldots, {\cal U}_{\st}$.
    Now, let us \emph{fix} a realization 
    $$
        X=x \quad \mbox{ and } \quad {\cal U}_{[0,\st]} = {\cal U}_{[0,\st]}^\sharp\,,
    $$
    of the random variables so that it is within the event $\Omega(\preceq \st)$, and with this realization, the data set generated by the algorithm up to step $\st$ is exactly ${\cal D}_{\st}$.

    From the data set ${\cal D}_{\st}$, we can determine what's the goal of the algorithm at step $\st+1$. There are two cases: either the algorithm tries to define $\axis_k(s)$ for some $s \in [d]$, which may involve extracting a new cluster from $V_{\mathrm{ortho}}$, or it tries to extract a new cluster from $V_{\mathrm{net}}$. Let us consider it case by case by following the structure of the algorithm. 

    \step{Special case: $\st=0$}
    At the first iteration, we simply generate the first cluster $(U_1,u_1)$ by  
    $$
        (U_1,u_1) = {\rm Algorithm\,} \ref{alg:cluster}(V_{\mathrm{net}},V_{\mathrm{ortho}},\eta)\,.
    $$
    In this case, 
    $$
        \Omega(1) = \Oclu{1}. 
    $$
    From $\Ept{V_{\mathrm{net}}} \supseteq \Omega(0)$,
    we know that for every $q \in M$,
    $$
        \big| \B{c_{\tref{prop: cluster-finding-algorithm}\eta^2}}{q} \cap \{X_v\}_{v \in V_{\mathrm{net}}} \big| \ge m\,.
    $$
    Together with $\Ecn{V_{\tr{cn}}}{V_{\tr{net}}} \supseteq \Omega(0)$, we can apply Proposition \ref{prop: cluster-finding-algorithm} to conclude that with probability $1-n^{-\omega(1)}$, the event $\Eclu(\eta,U_1,u_1)$ holds with $|U_1| \ge m$.

    So it remains to check the event $\Enavi{L_1}{V_{\mathrm{net}} \cup V_{\mathrm{ortho}} \setminus L_1}$ where $L_1=U_1$. Given that all the edges from $L_1$ to $V_{\mathrm{net}} \cup V_{\mathrm{ortho}} \setminus L_1$ have not been revealed yet, we simply apply Lemma \ref{lem: navi}, which gives 
    $$
        \Pr\big(\Enavi{L_1}{V_{\mathrm{net}} \cup V_{\mathrm{ortho}} \setminus L_1} \mid X=x, {\cal U}_{[0,\st]} = {\cal U}_{[0,\st]}^\sharp\big) = 1 - n^{-\omega(1)}\,.
    $$
    Since this probability estimate is true for any realization $X=x$ and ${\cal U}_{[0,\st]} = {\cal U}_{[0,\st]}^\sharp$ within the event $\Omega(\preceq \st)$, and also every data set ${\cal D}_{\st}$ compatible with the algorithm at step $\st$ and $\Omega(0)$, we have  
    $$
        \Pr\big(\Omega(1) \mid \Omega(0)\big) = 1 - n^{-\omega(1)}\,.
    $$
    This concludes the proof of the case $\st=0$. From now on, we assume $\st \ge 1$.

    \step{Case 1: $\axis_k$ is not defined for all $s \in [d]$}
    In this case, the algorithm tries to define $\axis_k(s)$ where $s$ is the smallest index such that $\axis_k(s)$ is not yet defined.
    The candidate set 
    \[
\mathcal{J}_{k,s} \;:=\; 
\Big\{\, j\in[\ell]\ :\
\begin{aligned}
&\rmp(r+0.93\delta)\ \le\ {\rm n}(u_k,v_j)\ \le\ \rmp(r-0.93\delta),\\
&\rmp(\sqrt{2}r+0.93\delta)\ \le\ {\rm n}(v_{\axis_k(\alpha)},v_j)\ \le\ \rmp(\sqrt{2}r-0.93\delta)\quad \forall\,\alpha\in[s-1]
\end{aligned}
\Big\}
\]
    is then constructed in the algorithm.

    \step{Case 1.1: ${\mathcal J}_{k,s} \neq \emptyset$}
    First, the event $\Oclu{\st+1}$ is trivially true since no new cluster will be extracted. Now one needs to verify the event $$\Oortho{\st}=\Eortho(r,\delta,\eta, (u_{k}, v_{f_1(1)},\dots,v_{\axis_{k}(s-1)},v_{\axis_{k}(s)}) ).$$
    To see this, we first remark that 
    $$
    \Eortho(r,\delta,\eta, (u_{k}, v_{f_1(1)},\dots,v_{\axis_{k}(s-1)}) )
     \supseteq \Omega(\preceq t)$$
    It remains to verify that 
    \begin{align*}
        \big| \D{X_{u_k}}{X_{v_{\axis_k(s)}}} -r \big| 
    \le 
        \delta  
    \quad \mbox{ and } \quad 
        \big| \D{X_{v_{\axis_k(\alpha)}}}{X_{v_{\axis_k(s)}}} - \sqrt{2}r \big|
    \le  
        \delta  
    \quad \mbox{ for } \alpha \in [s-1]\,,
    \end{align*}
    which follows directly from  the definition of ${\mathcal J}_{k,s}$ and~\eqref{eq: calibration_annulus_threshold} in Calibration Lemma~\ref{lem:calibration-toolkit}
    provided that $c_{\tref{prop:cluster-net-algorithm}}$ is sufficiently small. 
    In other words, 
    $$
        \Prob\Big( 
            \Omega(\st+1) \mid {\mathcal J}_{k,s} \neq \emptyset\,,  X = x , {\cal U}_{[0,\st]} = {\cal U}_{[0,\st]}^\sharp 
        \Big) = 1\,.
    $$
    

    \step{Case 1.2: $\axis_k(s)$ is defined by extracting a new cluster $(V_{\ell+1},v_{\ell+1})$ from $V_{\mathrm{ortho}}$}
    In this case, 
    \eqref{eq: calibration_annulus_threshold} in Calibration Lemma~\ref{lem:calibration-toolkit}, we know that for every $j \in [\ell] \backslash \{\axis_k(1),\ldots,\axis_k(s-1)\}$, and any vertex $v' \in V_{j}$ DOES NOT satisfies one of the following condition 
    \begin{align*}
        \big| \D{X_{u_k}}{X_{v'}} - r \big| 
    \le & 
        0.928\delta  
    \quad \mbox{ or } \quad 
        \big| \D{X_{v_{\axis_k(\alpha)}}}{X_{v'}} - \sqrt{2}r \big| 
    \le
        0.928\delta 
    \quad \mbox{ for some } \alpha \in [s-1]\,,
    \end{align*}
    provided that $c_{\tref{prop:cluster-net-algorithm}}$ is sufficiently small.  
    Further, the above property also holds for $v' \in V_{\axis_k(\alpha)}$ for every $\alpha \in [s-1]$, since it clearly violates the second condition. Therefore, for any $q \in M$ such that 
    \begin{align*}
        \D{X_{u_k}}{q} \in [r-0.92\delta,r+0.92\delta] \quad \mbox{ and } \quad \D{X_{v_{\axis_k(\alpha)}}}{q} \in [\sqrt{2}r-0.92\delta,\sqrt{2}r+0.92\delta] \quad \mbox{ for } \alpha \in [s-1]\,,
    \end{align*}
    every $q' \in B_{c_{\tref{prop: scenerio-1-ortho}} \eta^2}(q)$ satisfies the above inequality with $0.92$ replaced by $0.921$, by triangle inequalities. 
    Hence, we conclude that  
    \begin{align*}
        \Big(V_{\mathrm{ortho}} \setminus L(\preceq \st)\Big) \cap B_{c_{\tref{prop: scenerio-1-ortho}} \eta^2}(q) 
        = V_{\mathrm{ortho}} \cap B_{c_{\tref{prop: scenerio-1-ortho}} \eta^2}(q) \,,
    \end{align*}
    which in term implies that
    \begin{align*}
         \Big|\Big(V_{\mathrm{ortho}} \setminus L(\preceq \st)\Big) \cap B_{c_{\tref{prop: scenerio-1-ortho}} \eta^2}(q)\Big|
    \ge 
        m\,,
    \end{align*}
    from $\Ept{V_{\mathrm{ortho}}}\supseteq \Omega(0)$ and our choice of $m$.
    In particular, the condition in Proposition~\ref{prop: scenerio-1-ortho} is satisfied with $m= m'$, and so does the event 
    $\Eortho(r,\delta,\eta,(u_{k},v_{\axis_{k}(1)},\dots, v_{f_k(s-1)})).$
    Therefore, by Proposition~\ref{prop: scenerio-1-ortho}, Proposition~\ref{prop: cluster-finding-algorithm}, and Remark~\ref{rem: scenaro-1-W'-distance}, the output $(L_{\st+1},\ell_{\st+1})=(V_{\ell+1}, v_{\ell+1})$ of the subalgorithm Algorithm \ref{alg:cluster}($V_{\tr{cn}},W',\eta$) satisfies the event $\Eclu(\eta,V_{\ell+1},v_{\ell+1})$ with $|V_{\ell+1}| \ge m$ and 
    $$
        \Eortho(r,\delta,\eta,(u_{k},v_{\axis_{k}(1)},\dots, v_{f_k(s-1)},v_{\ell+1})).
    $$ 
    In particular, all events in the definition of $\Oclu{\st+1}$ and $\Oortho{\st+1}$ hold with the exception of the event $\Enavi{L_{\st+1}}{V_{\mathrm{net}} \cup V_{\mathrm{ortho}} \setminus L_{\preceq \st-1}}$. 
    For the latter, we note that all edges from $L_{\st+1}$ to $V_{\mathrm{net}} \cup V_{\mathrm{ortho}} \setminus L_{\preceq \st}$ have not been revealed yet. To be precise,  the variables ${\cal U}_{u,v}$ for $u \in L_{\st+1}$ and $v \in V_{\mathrm{net}} \cup V_{\mathrm{ortho}} \setminus L_{\preceq \st}$ have not been revealed yet.
    Therefore, by Lemma \ref{lem: navi}, we have 
    \begin{align*}
    &\Prob\big(\Omega(\st+1) \mid 
       {\mathcal J}_{k,s} = \emptyset\,,
            X = x , {\cal U}_{[0,\st]} = {\cal U}_{[0,\st]}^\sharp  
       \big) \\
    =& 
     \Prob\big( \Enavi{L_{\st+1}}{V_{\mathrm{net}} \cup V_{\mathrm{ortho}} \setminus L_{\preceq \st}} \mid X_{L_{\st+1}} = x_{L_{\st+1}} , X_{V_{\mathrm{net}} \cup V_{\mathrm{ortho}} \setminus L_{\preceq \st}} = x_{V_{\mathrm{net}} \cup V_{\mathrm{ortho}} \setminus L_{\preceq \st}} \big) \\
    =&
    1 - n^{-\omega(1)}\,.
    \end{align*}
    
    \step{Case 2: $f_k$ is already defined for all $s \in [d]$ and the algorithm tries to extract a new cluster from $V_{\mathrm{net}}$}
    In this case, the candidate set
        $$
        {\cal V}_{k+1} = \bigg\{ (v,i_v) \in (V_{\mathrm{net}} \setminus L_{\preceq t-1}) \times [k] :
            \cn{U_i}{v} \le \rmp(1.1\delta)
            \quad \forall i \in [k] \quad \mbox{and } 
            \cn{U_{i_v}}{v} \ge \rmp(1.9\delta) 
        \bigg\}
    $$ 
    is defined. 
    For each $(v,i_v) \in {\cal V}_{k+1}$,  by $\Enavi{L_{\alpha}}{V_{\mathrm{net}} \cup V_{\mathrm{ortho}} \setminus L_{\preceq \alpha-1}} \supseteq \Omega(\preceq \st)$ 
    for all $\alpha \le \st$. 
    Then invoking the annulus calibration~\eqref{eq: calibration_annulus_threshold} in Calibration Lemma~\ref{lem:calibration-toolkit}  we have 
    \begin{align*}
        \D{X_{u_i}}{X_v} \ge \delta \quad \mbox{ for all } i \in [k]\,, \quad \mbox{and} \quad
        \D{X_{u_{i_v}}}{X_v} \le 2\delta,
    \end{align*}
    provided that $c_{\tref{prop:cluster-net-algorithm}}$ is sufficiently small.

    \step{Case 2.1: ${\mathcal V}_{k+1} \neq \emptyset$}
    For each $v \in U^*$, 
    the Calibration Lemma \ref{lem:calibration-toolkit} implies that  
    \begin{align*}
    \D{X_{v}}{X_{u_{i^*}}} 
        \le 3.1\delta    \quad \Rightarrow \quad
        U^* \subseteq (V_{\tr{net}} \setminus L_{\preceq \st}) \cap \B{3.1\delta}{u_{i^*}}\,.
    \end{align*}
    where we use $\B{3.1\delta}{u_{i^*}}$ to denote the set of vertices $v \in \bfV$ such that $\D{X_v}{X_{u_{i^*}}} \le 3.1\delta$.
    Now, we want to apply Proposition \ref{prop: scenario-2-navigation} for 
    \begin{align*}
        ((U_{\rho_i},\rho_i))_{i \in [0,d]}
    = 
        ((U_{i^*},u_{i^*}),(V_{\axis_{i^*}(1)},v_{\axis_{i^*}(1)}),\dots,(V_{\axis_{i^*}(d)},v_{\axis_{i^*}(d)}) )\\
     \quad w_1 = u^* \quad \mbox{ and } \quad w_2 \in 
     (V_{\tr{net}} \setminus L_{\preceq \st}) \cap \B{3.1\delta}{u_{i^*}}\,.
    \end{align*}

    Clearly, the assumed events in Proposition~\ref{prop: scenario-2-navigation} are satisfied from $\Omega(\preceq \st)$. 
    Further, we also have, for each $v \in (V_{\tr{net}} \setminus L_{\preceq \st}) \cap \B{4\delta}{X_{u_{i^*}}}$, 
    we have 
    \begin{align*}
        \rmp(N_{U_{i^*}}(v)) \ge \rmp(4\delta + \eta ) - \fe{m}  \ge \rmp(9\delta)\,,
    \end{align*}
    satisfying the condition in Proposition~\ref{prop: scenario-2-navigation}.
    Therefore, we conclude that each $v \in U^*$ satisfies  
    $$
        \D{X_{u^*}}{X_v} 
    \le  
        \frac{3}{\ell_\rmp} (\psi_{i^*}(u^*, v) + 2\sqrt{d}\fe{m}) 
    \le     
        \frac{3}{\ell_\rmp} \Big( \frac{\ell_\rmp}{6}\eta + 2\sqrt{d}\fe{m}\Big) 
    \le 
        \eta\,,
    $$
    provided that $c_{\tref{prop:cluster-net-algorithm}}$ is sufficiently small. Therefore, the event 
    $$
        \Eclu(\eta,U_{k+1},u_{k+1}) =   
        \Eclu(\eta,U^*,u^*) 
    $$ holds.

    Next, consider 
    \begin{align*}
        V_{\tr{net}} \setminus L_{\preceq \st}     =  
        V_{\tr{net}} \setminus \big(\bigcup_{i \in [k]} U_i\big) \,.
    \end{align*}
    For each $u \in U_i$ for some $i \in [k]$, by the event $\Eclu(\eta,U_i,u_i) \supseteq \Omega(\preceq \st)$, we have
    \begin{align*}
        \D{X_u}{X_{u^*}} \ge 
        \D{X_{u_i}}{X_{u^*}} - \D{X_u}{X_{u_i}} \ge \delta - \eta \ge 0.9\delta.
    \end{align*}
    Thus, 
    $$
        \B{0.9\delta}{u^*} \cap V_{\tr{net}} \subseteq  V_{\tr{net}} \setminus L_{\preceq \st} \,.    
    $$
    Fix any $v \in V_{\tr{net}}$ such that 
    $$\D{X_v}{X_{u^*}} \le \frac{\ell_\rmp}{18L_\rmp}\eta.$$ 
    First, $v \in V_{\tr{net}} \setminus L_{\preceq \st}$.
    Second, by triangle inequalities, we have 
    $$
        \D{X_v}{X_{u_{i^*}}} 
    \le 
        \D{X_{u^*}}{X_{u_{i^*}}} + \frac{\ell_\rmp}{18L_\rmp}\eta 
    \le 
        2.1\delta\,,
    $$ 
    and therefore, by the event $\Enavi{L_{i^*}}{V_{\tr{net}} \cup V_{\tr{ortho}} \setminus L_{\preceq i^*-1}} \supseteq \Omega(\preceq \st)$ and~\eqref{eq: calibration_annulus_threshold} in Calibration Lemma~\ref{lem:calibration-toolkit},
    we have
    $$
        \rmp(N_{U_{i^*}}(v)) \ge \rmp(3\delta)\,.
    $$
      
    Third, by Proposition~\ref{prop: scenario-2-navigation}, we have
    \begin{align*}
        \psi_{i^*}(u^*,v) \le& 3L_\rmp \D{X_{u^*}}{X_v} \le \frac{\ell_\rmp}{6}\eta\,.
    \end{align*}
    Together we conclude that 
    $$
        B_{\ell_\rmp \eta/18L_\rmp}(u^*) \cap V_{\tr{net}} \subseteq U^*\,,
    $$
    and therefore, by $\Ept{V_{\tr{net}}} \supseteq \Omega(0)$ and our choice of $m$, we have  
    $$
        |L_{\st+1}| = |U^*| \ge |B_{\ell_\rmp \eta/18L_\rmp}(u^*) \cap V_{\tr{net}}|
        \ge m\,,
    $$
    provided that $c_{\tref{prop:cluster-net-algorithm}}$ is sufficiently small. Therefore, the only missing event in the definition of $\Oclu{\st+1}$ and $\Oortho{\st+1}$ is the event $\Enavi{L_{\st+1}}{V_{\mathrm{net}} \cup V_{\mathrm{ortho}} \setminus L_{\preceq \st-1}}$.
    As in Case 1.2, all edges from $L_{\st+1}$ to $V_{\mathrm{net}} \cup V_{\mathrm{ortho}} \setminus L_{\preceq \st}$ have not been revealed yet. Therefore, the same analysis as in Case 1.2 shows that 
    \begin{align*}
    \Prob\big(\Omega(\st+1) \mid
         {\mathcal V}_{k+1} \neq \emptyset\,,
                X = x , {\cal U}_{[0,\st]} = {\cal U}_{[0,\st]}^\sharp  
         \big) 
    =&
    1 - n^{-\omega(1)}\,.
    \end{align*}

    \step{Case 2.2: ${\mathcal V}_{k+1} = \emptyset$}
    Suppose there exists $v \in V_{\tr{net}} \setminus L_{\preceq \st}$ such that
    \begin{align*}
        \D{X_v}{X_{u_i}} \ge 2\delta \quad \mbox{ for all } i \in [k]\,.
    \end{align*}
    Let $\gamma$ be the geodesic from $X_v$ to $X_{u_1}$. Let $p$ be the first time on $\gamma$ such that $\D{p}{X_{u_i}} = 1.5\delta$ for some $i' \in [0,k]$. By the event $\Ept{V_{\tr{net}}} \supseteq \Omega(0)$, there exists $v' \in V_{\tr{net}}$ such that $\D{X_{v'}}{p} \le 0.1\delta$. Therefore, we have 
    \begin{align*}
        \D{X_{v}}{X_{u_i}} \ge 1.4 \delta \quad \mbox{ for all } i \in [k]\,,
        \quad \mbox{and} \quad
        \D{X_{v}}{X_{u_{i'}}} \le 1.6\delta\,. 
    \end{align*}
    In particular, from the event $\Eclu(\eta, U_i,u_i) \supseteq \Omega(\preceq \st)$ for all $i \in [k]$, we have
    $$
        v \in V_{\tr{net}} \setminus L_{\preceq \st}  
    $$
    and consequently, $(v,i') \in {\mathcal V}_{k+1}$, contradicting the assumption that ${\mathcal V}_{k+1} = \emptyset$. Therefore, we conclude that when ${\mathcal V}_{k+1} = \emptyset$, there is no $p \in M$ such that 
    \begin{align*}
        \D{p}{X_{u_i}} \ge 2\delta \quad \mbox{ for all } i \in [k]\,.
    \end{align*} 
    Therefore, the event $\Onet{\st+1}$ is true.

    In short, we conclude that 
    \begin{align*}
    &\Prob\big(\Omega(\st+1) \mid 
        X = x , {\cal U}_{[0,\st]} = {\cal U}_{[0,\st]}^\sharp  
       \big) 
    \end{align*}
    for every realization $x$ and ${\cal U}_{[0,\st]}^\sharp$ such that $\Omega(\preceq \st)$ holds and ${\cal D}_{\st}$ is the data set generated by the algorithm up to step $\st$. 
    Since the above holds for every such realization and ${\cal D}_{\st}$ that is compatible with $\Omega(\preceq \st)$, we conclude that
    \begin{align*}
    \Prob\big(\Omega(\st+1) \mid \Omega(\preceq \st) \big) 
    =& 
    1 - n^{-\omega(1)}\,.
    \end{align*}

    \begin{proof}[Proof of Proposition~\ref{prop:cluster-net-algorithm}]
        By the Lemma, it is clear that  
        $$
            \Prob\big(\Omega(\preceq (d+1)n) \big) = 1 - n^{-\omega(1)}\,.
        $$
        Within the event $\Omega(\preceq (d+1)n)$, the algorithm terminates at step $\st \le (d+1)n$. This is because at each step $\st$, either a new cluster $(U_i,u_i)$ from $V_{\tr{net}}$ is extracted, or a cluster $(V_j,v_j)$ is assigned to form an orthogonal pair. There can be at most $n$ clusters extracted from $V_{\tr{net}}$, given the size of $V_{\tr{net}}$, and thus the algorithm must terminate by step $(d+1)n$.    
    \end{proof}

\section{Stage 2: Distilling fine clusters}

\begin{prop}
\label{prop:refine-fine}
There exists \(c_{\tref{prop:refine-fine}}>0\) such that the following holds. Recall the parameter choice
\[
    r \;:=\; c_{\tref{prop:cluster-net-algorithm}}\, \rG, 
    \qquad 
    \delta \;=\; c_{\tref{prop:cluster-net-algorithm}}\, r,
    \qquad 
    \eta \;=\; c_{\tref{prop:cluster-net-algorithm}}\, \delta,
\]
and
\[
    m \;:=\; \mu_{\min}\!\big( c_{\tref{prop: cluster-finding-algorithm}}\eta^2/3\big)\,\frac{n}{2}.
\]
Let \(V_{\tr{ocn}}\subseteq\bfV\) with \(|V_{\tr{ocn}}|=2n+ n_{\tr{cn}}\), where \(n_{\tr{cn}}\) is introduced in Algorithm~\ref{alg:cluster-net} 
 Let \(\Omega_{\rm ocn}\) be the there event 
running partitioning \(V_{\tr{ocn}}\) into 3 parts \(V_{\tr{cn}}\), \(V_{\tr{net}}\) and \(V_{\tr{ortho}}\) and then running Algorithm~\ref{alg:cluster-net} on \(V_{\tr{cn}} \sqcup V_{\tr{net}} \sqcup V_{\tr{ortho}} \) with parameters \(r,\delta,\eta,m\) outputs  
\begin{itemize}
    \item A family \(\{(U_i,u_i)\}_{i=1}^k\subseteq V_{\tr{ocn}}\) with \(|U_i|\ge m\) and
\(\Eclu(\eta,U_i,u_i)\). Their centers form a \(2\delta\)-net of \(M\):
\(\{X_{u_i}\}_{i=1}^k\) is \(2\delta\) net of \(M\) with $\delta$-separation.
    \item Let \(\{(V_j,v_j)\}_{j=1}^\ell\subseteq V_{\tr{ocn}}\) with \(|V_j|\ge m\) and \(\Eclu(\eta,V_j,v_j)\).
For each \(i\in[k]\) there is a map \(\axis_i:[d]\to[\ell]\) such that the orthogonal-frame event $
    \Eortho(r,\delta,\eta,(u_i,v_{\axis_i(1)},\dots,v_{\axis_i(d)}))$ holds, namely: 
    \[
        \begin{cases}
            |\D{X_{u_i}}{X_{v_{\axis_i(\alpha)}}}-r|\le\delta & \forall \alpha\in[d]\\
            |\D{X_{v_{\axis_i(\alpha)}}}{X_{v_{\axis_i(\beta)}}}-\sqrt{2}r|\le\delta & \forall \alpha,\beta\in[d],\alpha\neq\beta
        \end{cases}
    \]
\end{itemize}
Let \(V_{\tr{fine}}\subseteq\bfV\) with \(|V_{\tr{fine}}|=n\) be disjoint from \(V_{\tr{ocn}}\).
Assume \(\Omega_{\rm ocn}\), \(\Enavi{U_i}{V_{\tr{fine}}}\) and \(\Enavi{V_j}{V_{\tr{fine}}}\) for all \(i\in[k]\), \(j\in[\ell]\), and \(\Ept{V_{\tr{fine}}}\). We set 
\begin{align}
    \xi := c_{\tref{prop:refine-fine}} \left( \frac{\log n}{\sp n}\right)^{1/(d+2)}.
\end{align}
The Algorithm~\ref{alg:refine-fine}, which takes the edges 
$$
    E(U_i,V_{\tr{fine}}),\quad E(V_j,V_{\tr{fine}})\quad(i\in[k],j\in[\ell]),
$$
as input, returns clusters \(\{(W_w,w)\}_{w\in V_{\tr{fine}}}\) with \(W_w\subseteq V_{\tr{fine}}\) such that each \(W_w\) satisfies \(\Eclu(\xi,W_w,w)\) 
and its size satisfies the size calibration margin~\eqref{eq: eta_fe}: 
$$
    \fe{W_w} \;\le\; \ell_\rmp \xi.
$$
\end{prop}
\begin{rem}
    \label{rem: fine-prob}
    If we denote by \(\Omega_{\rm fine}\) the event that the composition of Algorithm~\ref{alg:cluster-net} and Algorithm~\ref{alg:refine-fine} satisfies the properties in Proposition~\ref{prop:refine-fine}, 
    then we have 
    \begin{align*}
        \Prob( \Omega_{\rm fine}) \;=\; 1 - n^{-\omega(1)}.
    \end{align*}
    This is simply due to conditional independence. First, $\Omega_{\rm ocn}$ holds with probability \(1-n^{-\omega(1)}\) by Proposition~\ref{prop:cluster-net-algorithm}, and is an event of $X_{V_{\tr{ocn}}}$ and ${\cal U}_{V_{\tr{ocn}},V_{\tr{ocn}}}$. 
    Second, $\Ept{V_{\tr{fine}}}$, as an event of $X_{V_{\tr{fine}}}$, is independent of $\Omega_{\rm ocn}$ and holds with probability $1 - n^{-\omega(1)}$ by Lemma~\ref{lem: epsilonNetEvent}, which implies 
    $$
        \Pr(\Omega_{\rm ocn} \cap \Ept{V_{\tr{fine}}}) = 1 - n^{-\omega(1)}\,.
    $$
    Third, the events $\Enavi{U_i}{V_{\tr{fine}}}$ and $\Enavi{V_j}{V_{\tr{fine}}}$ for all $i \in [k]$ and $j \in [\ell]$ are all events of ${\cal U}_{U_i,V_{\tr{fine}}}$ and ${\cal U}_{V_j,V_{\tr{fine}}}$, which are independent, conditioning of $X_{V_{\tr{ocn}} \cup V_{\tr{fine}}}$ and ${\cal U}_{V_{\tr{ocn}},V_{\tr{ocn}}}$. Therefore, by Lemma~\ref{lem: navi} and union bound, we have  
    \begin{align*}
        \Pr\Big( 
            \bigcap_{i=1}^k \Enavi{U_i}{V_{\tr{fine}}} \cap \bigcap_{j=1}^\ell \Enavi{V_j}{V_{\tr{fine}}} 
            \,\Big\vert\, \Omega_{\rm ocn} \cap \Ept{V_{\tr{fine}}}\Big) = 1 - n^{-\omega(1)}\,,
    \end{align*}
    which implies $\Omega_{\rm fine}$ holds with probability $1 - n^{-\omega(1)}$. 
\end{rem}
The algorithm is straightforward. For each \(w\in V_{\tr{fine}}\) and choose \(i(w)\in[k]\) such that 
\[
    N_{U_{i(w)}}(w) \;\ge\; \rmp(3\delta),
\]
and set 
\[
u^w:=u_{i(w)},\quad U^w:=U_{i(w)},\qquad 
v^w_\alpha:=v_{\axis_{i(w)}(\alpha)},\quad V^w_\alpha:=V_{\axis_{i(w)}(\alpha)}\quad(\alpha\in[d]).
\]
Define, for \(v\in V_{\tr{fine}}\),
\[
   \psi_w(v) \;:=\; \sqrt{\sum_{\alpha=1}^d \big(\cn{V^w_\alpha}{v} - \cn{V^w_\alpha}{w}\big)^2},
\]
and for a threshold \(\lambda\) (to be specified below),
\[
    W_w \;:=\; \Big\{ v \in V_{\tr{fine}} : \cn{U^w}{v} \ge \rmp(9\delta)\ \text{ and }\ \psi_w(v) \le \lambda \Big\}.
\]
We assume throughout that
\[
   4\sqrt{d}\,\fe{m} \;\le\; \lambda \;\le\; \eta^2,
\]
which is feasible for all large \(n\) under our parameter configuration (since \(\eta\) is a fixed small constant and \(\fe{m}=o(1)\)). The precise choice of \(\lambda\) is given in the algorithm and derived in the proof below.

\begin{algorithm}[H]
\small
\caption{\texttt{refine-fine}: axis-guided local clustering on $V_{\tr{fine}}$}
\label{alg:refine-fine}

\SetKwInOut{Input}{Input}
\SetKwInOut{Output}{Output}

\Input{
$\{(U_i,u_i)\}_{i\in[k]}$, $\{(V_j,v_j)\}_{j\in[\ell]}$, and axis maps $\{\axis_i:[d]\to[\ell]\}_{i\in[k]}$ \\ 
Fresh set $V_{\tr{fine}}\subseteq \bfV$ with $|V_{\tr{fine}}|=n$, disjoint from all clusters.\\
Edges $E(U_i,V_{\tr{fine}})$ and $E(V_j,V_{\tr{fine}})$ for all $i\in[k]$, $j\in[\ell]$.\\
Parameters $r,\delta,\eta$ (standing choice tied to $c_{\tref{prop:cluster-net-algorithm}}$), integer $m$.\\
}
\Output{
Refined local clusters $\{(W_w,w)\}_{w\in V_{\tr{fine}}}$ with $W_w\subseteq V_{\tr{fine}}$.
}

{\em Initialization:}\\
$\lambda \gets  \left(
\frac{\sqrt{2}\,(18L_\rmp)^{d/2}}{6\sqrt{c_\mu}}
\cdot
\frac{\log n}{\sqrt{\sp n}}
\right)^{\!\frac{2}{d+2}}$ \;
\ForEach{$w\in V_{\tr{fine}}$}{
    \tcp{Choose reference net cluster via a coarse proximity test}
    $i(w)\gets \arg\max_{i\in[k]}\ \mathbf 1\{\cn{U_i}{w}\ge \rmp(3\delta)\}$;

    \If{no such $i$ exists}{
        \Return \texttt{failure} (no cluster found around $w$);
    }
    $U^w \gets U_{i(w)}$; \ $u^w \gets u_{i(w)}$\;
    \For{$\alpha\in[d]$}{ $V^w_\alpha \gets V_{\axis_{i(w)}(\alpha)}$; \ $v^w_\alpha\gets v_{\axis_{i(w)}(\alpha)}$ }
    \BlankLine
    \tcp{Define the axis-based difference functional $\psi_w$}
    \ForEach{$v\in V_{\tr{fine}}$}{
        $\psi_w(v)\ \gets\ \sqrt{\sum_{\alpha=1}^d \big(\cn{V^w_\alpha}{v}-\cn{V^w_\alpha}{w}\big)^2}$\;
    }
    \BlankLine
    \tcp{Select the refined cluster around $w$}
    $W_w \gets \big\{\,v\in V_{\tr{fine}}:\ \cn{U^w}{v}\ge \rmp(9\delta)\ \wedge\ \psi_w(v)\le \lambda\,\big\}$\;
    }

\Return $\{(W_w,w)\}_{w\in V_{\tr{fine}}}$.

\end{algorithm}
\begin{rem}
    The running time of Algorithm~\ref{alg:refine-fine} is \(O(n^2)\) where the dominant step is to compute the edges $\cn{U^w}{v}$ and $\cn{V^w_\alpha}{v}$ for all \(w,v\in V_{\tr{fine}}\) and \(\alpha\in[d]\).
\end{rem}
\begin{proof}
\step{Step 1: From $N_{U^w}(w)\ge \rmp(3\delta)$ to a bound on \(\D{X_w}{X_{u^w}}\)}
First, we know there exists $i \in [k]$ such that $\D{X_w}{X_{u_i}} \le 2\delta$ by the event that $\{X_{u_i}\}_{i=1}^k$ is a $2\delta$-net of $M$. Fix such $i$, with $\Eclu(\eta,U_i,u_i)$ and $\Enavi{U_i}{V_{\tr{fine}}}$, the Calibration Lemma~\ref{lem:calibration-toolkit} gives 
$$
    N_{U_i}(w) \;\ge\; \rmp(2\delta+\eta) - \fe{U_i} \;\ge\; \rmp(3\delta),
$$
which guarantees \(i(w)\) is well-defined. Now, with the choice of \(i(w)\), 
the Calibration Lemma~\ref{lem:calibration-toolkit} gives
\[
\D{X_w}{X_{u^w}} \;\le\; 3\delta + 2\eta \;\le\; 4\delta,
\]
for \(c_{\tref{prop:cluster-net-algorithm}}\) small enough. 

\step{Step 2: A convenient geometric inclusion}
For any \(v\in V_{\tr{fine}}\) with \(\D{X_v}{X_w}\le \delta\), the triangle inequality and Step~1 give
\[
\D{X_v}{X_{u^w}} \;\le\; \D{X_v}{X_w} + \D{X_w}{X_{u^w}}
\;\le\; \delta + 4\delta \;=\; 5\delta.
\]
By monotonicity of \(\rmp\), the bracket \(\rmp(\D{X_v}{X_{u^w}}+\eta)\le \acn{U^w}{v}\le \rmp(\D{X_v}{X_{u^w}}-\eta)\) (see~\eqref{eq:avg-bracket}) and \(\fe{U^w}\le \ell_\rmp\eta\), we obtain
\[
\cn{U^w}{v}
\;\ge\;
\acn{U^w}{v}-\fe{U^w}
\;\ge\;
\rmp(5\delta+\eta)-\fe{U^w}
\;\ge\;
\rmp(9\delta),\,.
\]
Therefore,
\begin{equation}\label{eq:ball-inclusion}
B_{5\delta}(w)\cap V_{\tr{fine}}\ \subseteq\ \big\{v\in V_{\tr{fine}}:\ \cn{U^w}{v}\ge \rmp(9\delta)\big\}.
\end{equation}

\step{Step 3: Navigation via the orthogonal frame (Scenario~2).}
On \(\Omega_{\rm ocn}\), the family \(\{(U^w,u^w),(V^w_\alpha,v^w_\alpha)\}_{\alpha\in[d]}\) is an approximate orthogonal frame. 
Thus Proposition~\ref{prop: scenario-2-navigation} applies (using \(\Enavi{U^w}{\{w,v\}}\) and \(\Enavi{V^w_\alpha}{\{w,v\}}\)), yielding for all \(v\in V_{\tr{fine}}\) with \(\cn{U^w}{v}\ge \rmp(9\delta)\),
\begin{equation}\label{eq:psi-bounds}
\frac{\ell_\rmp}{3}\,\D{X_w}{X_v}\ -\ 2\sqrt{d}\,\fe{m}
\ \le\ 
\psi_w(v)
\ \le\
3L_\rmp\,\D{X_w}{X_v}\ +\ 2\sqrt{d}\,\fe{m}.
\end{equation}

\step{Step 4: Inner and outer radii for \(W_w\)}
Assume \(\lambda\ge 4\sqrt{d}\,\fe{m}\).
Then from the upper bound in~\eqref{eq:psi-bounds},
\[
\D{X_w}{X_v}\ \le\ \frac{\lambda}{6L_\rmp}\ \Longrightarrow\ \psi_w(v)\ \le\ \frac{\lambda}{2}+2\sqrt{d}\,\fe{m}\ \le\ \lambda,
\]
so every \(v\in V_{\tr{fine}}\) with \(\D{X_v}{X_w}\le \lambda/(6L_\rmp)\) and \(\cn{U^w}{v}\ge \rmp(9\delta)\) lies in \(W_w\).
Since \(\lambda\le \eta\ll \delta\), this ball is contained in \(B_{5\delta}(w)\), hence by \eqref{eq:ball-inclusion},
\begin{equation}\label{eq:inner-ball}
B_{\lambda/(6L_\rmp)}(w)\cap V_{\tr{fine}}\ \subseteq\ W_w.
\end{equation}

Conversely, if \(\D{X_v}{X_w}\ge \dfrac{6\lambda}{\ell_\rmp}\), then the lower bound in~\eqref{eq:psi-bounds} and \(\lambda\ge 4\sqrt{d}\,\fe{m}\) imply
\[
\psi_w(v)\ \ge\ 2\lambda - 2\sqrt{d}\,\fe{m}\ \ge\ \tfrac{3}{2}\lambda\ >\ \lambda,
\]
so such \(v\) cannot belong to \(W_w\).
Therefore,
\begin{equation}\label{eq:outer-ball}
W_w\ \subseteq\ B_{6\lambda/\ell_\rmp}(w)\cap V_{\tr{fine}}.
\end{equation}
In particular, \(\Eclu\big(6\lambda/\ell_\rmp,\ W_w,\ w\big)\) holds.

\step{Step 5: Size lower bound for \(W_w\)}
By~\eqref{eq:inner-ball} and the point–count event \(\Ept{V_{\tr{fine}}}\),
\begin{equation}\label{eq:size-lb}
|W_w|\ \ge\ \big| B_{\lambda/(6L_\rmp)}(w)\cap V_{\tr{fine}}\big|
\ \ge\ \mu_{\min}\!\Big(\frac{\lambda}{18L_\rmp}\Big)\,\frac{n}{2}.
\end{equation}
With \(\lambda/(18L_\rmp)\le \frac{\eta}{18L_\rmp} \ll {\rm r}_\mu\), the standing assumption \(\mu_{\min}(r)\ge c_\mu r^d\) yields
\[
|W_w|\ \ge\ \frac{c_\mu}{2}\,\Big(\frac{\lambda}{18L_\rmp}\Big)^d\,n.
\]


\step{Step 6: Fluctuation level inside \(W_w\) and the balanced choice of \(\lambda\)}
Using the definition \(\fe{U}=\dfrac{\log n}{\sqrt{\sp|U|}}\) and~\eqref{eq:size-lb}, we get
\[
\fe{W_w}
\ \le\ 
\
\frac{\sqrt{2}\,\log n}{\sqrt{c_\mu \sp n}\,}\,
\Big(\frac{18L_\rmp}{\lambda}\Big)^{\!d/2}.
\]
Recall that if we treat $(W^w,w)$ as a cluster and within the event $\Enavi{W^w}{v}$ for some vertex $v$, then    
$$
    |\cn{W^w}{v} - \acn{W^w}{v}| \le \frac{\fe{W^w}}{\ell_\rmp} + \frac{6\lambda}{\ell_\rmp}\,.
$$
So a convenient ``balanced'' choice of \(\lambda\) is obtained by imposing \(\fe{W_w} = 6\lambda\).
Define
\begin{equation}\label{eq:lambda-bal}
\lambda_{\rm bal}
\ :=\
\left(
\frac{\sqrt{2}\,(18L_\rmp)^{d/2}}{6\sqrt{c_\mu}}
\cdot
\frac{\log n}{\sqrt{\sp n}}
\right)^{\!\frac{2}{d+2}}.
\end{equation}
Finally set
\begin{equation}\label{eq:lambda-final}
\lambda\ :=\ \min\Big\{\,\eta,\ \max\{\,4\sqrt{d}\,\fe{m},\ \lambda_{\rm bal}\,\}\Big\}
= \max\{\,4\sqrt{d}\,\fe{m},\ \lambda_{\rm bal}\,\}
= \lambda_{\rm bal},\,
\end{equation}
where the last inequality holds for all sufficiently large \(n\) since  
$$
    \lambda_{\rm bal} = \Theta\left(  \left( \frac{\log n}{\sqrt{\sp n}}\right)^{2/(d+2)} \right) 
    \mbox{ while } 
    \fe{m} = \Theta\left( \frac{\log n}{\sqrt{\sp n}} \right) \,.
$$
For all sufficiently large \(n\), \(\lambda\) satisfies \(4\sqrt{d}\,\fe{m}\le \lambda\le \eta\), hence Steps~4–6 apply and give
\[
\Eclu\!\Big(\frac{6\lambda}{\ell_\rmp},\,W_w,\,w\Big),
\qquad 
|W_w|\ \ge\ \mu_{\min}\!\Big(\frac{\lambda}{18L_\rmp}\Big)\,\frac{n}{2},
\qquad 
\fe{W_w}\ \le\ 6\lambda.
\]
Replacing $\lambda$ by $\frac{6\lambda}{\ell_\rmp}$ concludes the proof of Proposition~\ref{prop:refine-fine}.
\end{proof}

We conclude this section by stating a byproduct which will be used later estimate the pairwise distances between any two points.  

\begin{algorithm}[H]
\small
\caption{\texttt{refine-fine-net}: greedy extraction of a $\zeta$-net on $V_{\tr{fine}}$}
\label{alg:refine-fine-net}

\SetKwInOut{Input}{Input}
\SetKwInOut{Output}{Output}

\Input{
$\{(U_i,u_i)\}_{i\in[k]}$, $\{(V_j,v_j)\}_{j\in[\ell]}$, and axis maps $\{\axis_i:[d]\to[\ell]\}_{i\in[k]}$;\\
Fresh set $V_{\tr{fine}}\subseteq \bfV$ with $|V_{\tr{fine}}|=n$, disjoint from $V_{\tr{ocn}}$;\\
Edges $E(U_i,V_{\tr{fine}})$ and $E(V_j,V_{\tr{fine}})$ for all $i\in[k],j\in[\ell]$;\\
Parameters $r,\delta,\eta,m$ (standing choice tied to $c_{\tref{prop:cluster-net-algorithm}}$);\\
A scale $\zeta$ with $C_{\tref{lem:refine-fine-net}}\fe{m}\le \zeta \le \eta$.}
\Output{
A subset $S\subseteq V_{\tr{fine}}$ such that $\{X_w:w\in S\}$ is a $\zeta$-net of $M$ with $\frac{1}{18L_\rmp}\zeta$-separation (on $\Omega_{\rm ocn}$).}

{\em Precomputation (axis setup and local gauges):}\\
\ForEach{$w\in V_{\tr{fine}}$}{
  \tcp{Choose a reference net cluster for $w$}
  $i(w)\gets \arg\max_{i\in[k]}\ \mathbf 1\{\,\cn{U_i}{w}\ge \rmp(3\delta)\,\}$; \\
  $U^w \gets U_{i(w)}$;\quad $u^w \gets u_{i(w)}$;\\
  \For{$\alpha\in[d]$}{
     $V^w_\alpha \gets V_{\axis_{i(w)}(\alpha)}$;\quad $v^w_\alpha\gets v_{\axis_{i(w)}(\alpha)}$;
  }
  \tcp{Define the axis-based difference functional $\psi_w$}
  \ForEach{$v\in V_{\tr{fine}}$}{
     $\psi_w(v)\ \gets\ \sqrt{\sum_{\alpha=1}^d\big(\cn{V^w_\alpha}{v}-\cn{V^w_\alpha}{w}\big)^2}$;
  }
}
\BlankLine

{\em Greedy selection:}\\
$\tau\ \gets\ 10L_\rmp \zeta $ \tcp*{separation threshold from Lemma~\ref{lem:refine-fine-net}}
$S\gets\emptyset$;\quad $A\gets V_{\tr{fine}}$; \tcp*{$A$ = uncovered points}
$t \gets 0$
\While{$A\neq\emptyset$}{
   $t \gets t+1$;
   pick any $w_t\in A$; \tcp*{arbitrary choice}
   $S\gets S\cup\{w_t\}$; \\
   \tcp{Mark points \emph{covered} by $w_t$ (close in the coarse test \& small axis discrepancy)}
   $\mathrm{Cover}(w_t)\ \gets\ \big\{\,v\in A:\ \cn{U^{w_t}}{v}\ge \rmp(9\delta)\ \wedge\ \psi_{w_t}(v)<\tau\,\big\}$;\\
   $A\gets A\setminus \mathrm{Cover}(w_t)$;
}
\Return $(S, (\mathrm{Cover}(w))_{w \in S})$.

\end{algorithm}

\begin{lemma}
\label{lem:refine-fine-net}
Under the same notation and assumptions as in Proposition~\ref{prop:refine-fine},
there exists a constant \(C_{\tref{lem:refine-fine-net}}>0\), depending on \(L_\rmp, \ell_\rmp,\) and \(d\), such that the following holds.  
For any 
\[
C_{\tref{lem:refine-fine-net}} \, \fe{m} \le \zeta \le \eta,
\]
on the event \(\Omega_{\rm ocn}\) and \(\Ept{V_{\tr{fine}}}\), the Algorithm~\ref{alg:refine-fine-net} returns a collection $(\mathrm{Cover}(w), w)_{w \in S}$ for some subset \(S \subseteq V_{\tr{fine}}\) such that
$S$  is a \(2\zeta\)-separation.
Further, for each \(w \in S\) and \(v \in \mathrm{Cover}(w)\),
\[
\D{X_w}{X_v} \le C_{\tref{lem:refine-fine-net}} \zeta.
\]
\end{lemma}

\begin{proof}[Proof of Lemma~\ref{lem:refine-fine-net}]
As established in the proof of Proposition~\ref{prop:refine-fine}, for any \(w,v \in V_{\tr{fine}}\) with \(\D{X_w}{X_v} \le \delta\), we have  
\[
\cn{U^w}{w'} \ge \rmp(9\delta),
\qquad \text{and} \qquad
\frac{\ell_\rmp}{3}\,\D{X_w}{X_v} - 2\sqrt{d}\,\fe{m}
\ \le\ 
\psi_w(v)
\ \le\
3L_\rmp\,\D{X_w}{X_v} + 2\sqrt{d}\,\fe{m}.
\]

We now construct \(S\) via a greedy selection procedure.  
Start with an arbitrary singleton \(\{w_1\} \subseteq V_{\tr{fine}}\), and iteratively add \(w_{i+1} \in V_{\tr{fine}}\) such that for all \(s \in [i]\),
\[
   \psi_{w_s}(w_{i+1}) \ge  10 L_\rmp \zeta
      \quad \text{or} \quad
   \cn{U^{w_s}}{w_{i+1}} < \rmp(9\delta).
\]
This construction ensures that for any two distinct points \(w, w' \in S\),
\[
\D{X_w}{X_{w'}} 
\ \ge\ 
2\zeta\,
\]
where the last inequality follows from the assumption \(\zeta \ge C\fe{m}\).
(In the case when $\cn{U^{w_s}}{w_{i+1}} < \rmp(9\delta)$, we simply have 
$\D{X_{w_s}}{X_{w_{i+1}}} \ge 8\delta > \zeta$.)

On the other hand, for each $w' \in \mathrm{Cover}(w)$, we have
\begin{align*}
\D{X_{w'}}{X_w}
&\le C \zeta\,, 
\end{align*}
when $C$ is a sufficiently large constant depending on $L_\rmp$ and $\ell_\rmp$.

Next, suppose that there exists a vertex \(w \in V_{\tr{fine}} \setminus S\) such that \(\D{X_{w'}}{X_w} \ge \zeta/2\) for all \(w' \in S\).  
Then for each \(w' \in S\), we must have either \(\psi_{w'}(w) \le \frac{\ell_\rmp}{3}\frac{\zeta}{2} - 2\sqrt{d}\,\fe{m}\) or \(\cn{U^{w'}}{w} \le \rmp(9\delta)\), contradicting the maximality of \(S\).  
Therefore, \(S\) is a \(\zeta/2\)-net of \(\{X_v\}_{v\in V_{\tr{fine}}}\).

Finally, by the event \(\Ept{V_{\tr{fine}}}\), every point \(q \in M\) lies within distance  
\[
\left(\frac{\log^2 n}{n}\right)^{1/d} \ll \frac{\zeta}{2}
\]
of some \(X_v\), implying that \(S\) is indeed a $\zeta$-net of \(M\).
\end{proof}

\section{From clusters to distance estimates}
\paragraph{Setup and objective.}
From now on assume $|\bfV|= 4n$. Suppose we fix a partition $\bfV=V_{\tr{ocn}}\,\dot\cup\,V_{\tr{fine}}$ with $|V_{\tr{ocn}}|=3n$ and $|V_{\tr{fine}}|=n$.
Applying the previous two stages (Algorithm~\ref{alg:cluster-net} on $V_{\tr{ocn}}$ and the refinement procedure on $V_{\tr{fine}}$) yields a family of clusters
\[
\{(W_w,w)\}_{w\in V_{\tr{fine}}},\qquad W_w\subseteq V_{\tr{fine}},
\]
such that each $W_w$ satisfies  
\[
\Eclu(\xi,W_w,w)\qquad\text{and}\qquad \fe{W_w}\le \ell_\rmp\,\xi,
\]
These clusters allow us to estimate pairwise distances within $V_{\tr{fine}}$ with high accuracy whenever the latent distance lies in the bi-Lipschitz window of $\rmp$ (via the Calibration Lemma). The goal of this section is to push beyond these two limitations:
\begin{enumerate}\itemsep0.2em
\item to extend the estimator \emph{beyond} the bi-Lipschitz regime of $\rmp$, by stitching together local distance estimates along paths, and
\item to produce a \emph{fine distance estimator} for \emph{every} pair $v,w\in \bfV$ (rather than just those in $V_{\tr{fine}}$).
\end{enumerate}

The resolution of the both issues are rather straightforward. Suppose we have a distance estimator ${\rm d}(w,v)$ for all $w,v\in V_{\tr{fine}}$ whose latent distances are within $\rG$. Then, we build a weighted graph on $V_{\tr{fine}}$ with edge weights ${\rm d}(w,v)$ for all $w,v\in V_{\tr{fine}}$ with $\D{X_w}{X_v}\lesssim \rG$, and $\infty$ otherwise. The shortest-path distance on this graph will be a good distance estimator for all pairs $w,v\in V_{\tr{fine}}$, with no restriction on the latent distance.  

For the second issues, we will simply partition $\bfV$ into $8$ equal parts with each part of size $n$.  
Then we run the above two-stage algorithms on every combination of $V_{\tr{fine}}$ being the union of two parts and $V_{\tr{ocn}}$ being the union of the other six parts. This way, every pair $v,w\in \bfV$ will be contained in some version of $V_{\tr{fine}}$. Overall, we will run the two-stage algorithms $8\choose 2$ times. The overall procedure is summarized in Algorithm~\ref{alg:distance}.

\begin{algorithm}[H]
\small
\caption{\texttt{fine distance}}
\label{alg:distance}

\SetKwInOut{Input}{Input}
\SetKwInOut{Output}{Output}

\Input{
    $\bfV$ with $|\bfV|=4n$.\\
    Edges $E(\bfV,\bfV)$.\\
    Parameters $r,\delta,\eta$, integer $m$.\\
}
\Output{
    s
}
$\xi \gets c_{\tref{prop:refine-fine}} \left( \frac{\log n}{\sp n}\right)^{1/(d+2)}$\;
Initialize ${\rm w}^{\alpha,\beta}: \bfV \to\R_{\ge0}$ by ${\rm w}^{\alpha,\beta}  (w,v)\gets\infty$ for all $w,v \in \bfV$ and $\{\alpha,\beta \} \in \binom{[8]}{2}$\;
Partition $\bfV$ into $8$ equal parts with each part of size $n$: $\bfV_1,\dots,\bfV_8$\;
\ForEach {$\{\alpha,\beta\}\subseteq \binom{[8]}{2}$}{
    $V_{\tr{fine}} \gets \bfV_\alpha \cup \bfV_\beta$\;
    $V_{\tr{ocn}} \gets \bfV\setminus V_{\tr{fine}}$\;
Randomly select disjoint subsets $V_{\tr{cn}}, V_{\tr{net}}, V_{\tr{ortho}}$ from $V_{\tr{ocn}}$ with $|V_{\tr{cn}}|= n_{\tr{cn}}$ (see \eqref{def: n-cn}), $|V_{\tr{net}}|=|V_{\tr{ortho}}|=n$\;
$(\{(U_i,u_i)\}_{i\in[k]},\{(V_j,v_j)\}_{j\in[\ell]},\{\axis_i:[d]\to[\ell]\}_{i\in[k]}) \gets$ \texttt{cluster-net}$(V_{\tr{cn}},V_{\tr{net}}, V_{\tr{ortho}},r,\delta,\eta,m, E(V_{\tr{ocn}}))$ \;

$\{(W^{\alpha,\beta}_w,w)\}_{w\in V_{\tr{fine}}} \gets$ \texttt{refine-fine}$(\{(U_i,u_i)\}_{i\in[k]},\{(V_j,v_j)\}_{j\in[\ell]},\{\axis_i\}_{i\in[k]}, V_{\tr{fine}}, E(U_i,V_{\tr{fine}}), E(V_j,V_{\tr{fine}}), r,\delta,\eta,m)$ \;
$(S, (\mathrm{Cover}(w))_{w \in S}) \gets$ \texttt{refine-fine-net}$(\{(U_i,u_i)\}_{i\in[k]},\{(V_j,v_j)\}_{j\in[\ell]},\{\axis_i\}_{i\in[k]}, V_{\tr{fine}}, E(U_i,V_{\tr{fine}}), E(V_j,V_{\tr{fine}}), r,\delta,\eta,m, \zeta = \xi)$ \;

\ForEach{$w\in V_{\tr{fine}}$}{
    $w' \gets$ the unique element in $S$ such that $w\in \mathrm{Cover}(w')$\;
    \ForEach{$v\in V_{\tr{fine}}$}{
        $v' \gets$ the unique element in $S$ such that $v\in \mathrm{Cover}(v')$\;
        \If {$v'=w'$}{
            ${\rm w}^{\alpha,\beta}(w,v) \gets C_{\tref{lem:refine-fine-net}} \xi$\;
        } 
        \ElseIf{$\cn{W^{\alpha,\beta}_w}{v}\ge \rmp(\rG)$}{
            ${\rm w}^{\alpha,\beta}(w,v) \gets \rmp^{-1}(\cn{W^{\alpha,\beta}_{w'}}{v})+2\xi+2C_{\tref{lem:refine-fine-net}} \xi$\;
        }
    }
}
}
${\rm w}(w,v) \gets \min\{{\rm w}^{\alpha,\beta}(w,v),{\rm w}^{\alpha,\beta}(v,w)\}_{\{\alpha,\beta\} \in \binom{[8]}{2}} \quad \forall w,v\in \bfV$\;

$({\rm d}(w,v))_{w,v\in \bfV} \gets$ Dijkstra$(({\rm w}(w,v))_{w,v\in \bfV})$\; 
\Return $({\rm d}(w,v))_{w,v\in \bfV}$\;
\end{algorithm}

\begin{rem}
We remark that the running time of Algorithm~\ref{alg:distance} is $O(n^2\log^2n)$ where the dominant step is to run Algorithm~\ref{alg:refine-fine} for ${8\choose 2}$ times. 
\end{rem}

Here we restate Theorem~\ref{thm:main} and provide its proof.
\begin{theorem*}
Assuming that $\sp n = \omega(\log^2 n)$. The following holds when $n$ is sufficiently large.
Running Algorithm~\ref{alg:distance} on $\bfV$ with $|\bfV|=4n$ and edges $E(\bfV,\bfV)$ returns a distance estimator $({\rm d}(w,v))_{w,v\in \bfV}$ such that, with probability $1 - n^{-\omega(1)}$, for all $w,v\in \bfV$,
\[
    |\D{X_w}{X_v} - {\rm d}(w,v)| \;\le\; C \left( \frac{\log^2 n}{\sp n}\right)^{\!\frac{1}{d+2}}  
\]
where $C$ is a constant depending only on $d$, $L_\rmp$, $\ell_\rmp$, $c_\mu$, $\rG$, and ${\rm diam}(M)$. 
\end{theorem*}

\begin{proof}[Proof of Theorem~\ref{thm:main}]
    \step{Good event for each run of the two-stage algorithms}
    Fix $\{\alpha,\beta\}\in \binom{[8]}{2}$. Let 
    $$
        \Omega_{\rm fine}^{\alpha,\beta} 
    $$ 
    be the event of $X_{\bfV}$ and ${\cal U}_{\bfV \setminus {\bf V}_\alpha \cup \bfV_\beta,\bfV}$ (excluding ${\cal U}_{\bfV_\alpha \cup \bfV_\beta,\bfV_\alpha \cup \bfV_\beta}$) that 
    $$
    \Ept{\bfV_\alpha \cup \bfV_\beta}, \quad 
        \Eclu(\xi,W^{\alpha,\beta}_w,w),\qquad \fe{W^{\alpha,\beta}_w}\le \ell_\rmp \xi\, \qquad \mbox{for all } w\in \bfV_\alpha \cup \bfV_\beta,
    $$
    where 
    $$
        \xi = c_{\tref{prop:refine-fine}} \left( \frac{\log n}{\sp n}\right)^{1/(d+2)}
    $$
    was defined in Proposition~\ref{prop:refine-fine}.
    This event holds with probability $1 - n^{-\omega(1)}$ by Proposition~\ref{prop:refine-fine} and Remark \ref{rem: fine-prob}. 
    Now, fix any realization of $X_{\bfV} =x_{\bfV}$ and ${\cal U}_{\bfV \setminus {\bf V}_\alpha \cup \bfV_\beta,\bfV} = {\mathsf u}_{\bfV \setminus {\bf V}_\alpha \cup \bfV_\beta,\bfV}$ such that $\Omega_{\rm fine}^{\alpha,\beta}$.  
    Given that ${\cal U}_{\bfV_\alpha \cup \bfV_\beta,\bfV_\alpha \cup \bfV_\beta}$ 
is independent of $X_{\bfV}$ and ${\cal U}_{\bfV \setminus {\bf V}_\alpha \cup \bfV_\beta,\bfV}$,  by Lemma~\ref{lem: navi} and union bound, we have
        \begin{align*}
        \Pr\Big(
            \bigcap_{w \in \bfV_\alpha \cup \bfV_\beta} \Enavi{W^{\alpha,\beta}_w}{\bfV_\alpha \cup \bfV_\beta \setminus W^{\alpha,\beta}_w}
            \, \Big\vert \, X_{\bfV} = x_{\bfV},\ {\cal U}_{\bfV \setminus {\bf V}_\alpha \cup \bfV_\beta,\bfV} = {\mathsf u}_{\bfV \setminus {\bf V}_\alpha \cup \bfV_\beta,\bfV}
            \Big) 
        \;=\; 1 - n^{-\omega(1)}.
    \end{align*}
    Given this conditional probability holds for all realizations of $X_{\bfV}$ and ${\cal U}_{\bfV \setminus {\bf V}_\alpha \cup \bfV_\beta,\bfV}$ such that $\Omega_{\rm fine}^{\alpha,\beta}$, we have
    $$
        \Omega_{\rm dist}^{\alpha,\beta} := \Omega_{\rm fine}^{\alpha,\beta} \cap \bigcap_{w \in \bfV_\alpha \cup \bfV_\beta} \Enavi{W^{\alpha,\beta}_w}{\bfV_\alpha \cup \bfV_\beta \setminus W^{\alpha,\beta}_w}
    $$
    holds with probability $1 - n^{-\omega(1)}$.

    \step{Distance estimates with ${\rm w}^{\alpha,\beta}$}
    Fix any realization of the graph such that $\Omega_{\rm dist}^{\alpha,\beta}$ holds.
    From $\Eclu(\xi,W^{\alpha,\beta}_w,w)$, $\fe{W^{\alpha,\beta}_w}\le \ell_\rmp \xi$, and $\Enavi{W^{\alpha,\beta}_w}{\bfV_\alpha \cup \bfV_\beta \setminus W^{\alpha,\beta}_w}$,
    the Calibration Lemma~\ref{lem:calibration-toolkit} gives, for all $w,v\in \bfV_\alpha \cup \bfV_\beta$ with $v \notin W^{\alpha,\beta}_w$ and $\cn{W^{\alpha,\beta}_w}{v}\ge \rmp(\rG)$,
    \begin{align*}
        |\rmp^{-1}(\cn{W^{\alpha,\beta}_w}{v}) - \D{X_w}{X_v}|
        &        \ \le\ 2\xi.
    \end{align*}
    In the case $v \in W^{\alpha,\beta}_w$, we have $\D{X_w}{X_v} \le \xi$ by $\Eclu(\xi,W^{\alpha,\beta}_w,w)$. Therefore, from our definition
    \begin{align*}
        {\rm w}^{\alpha,\beta}(w,v) 
        &\ =\ 
        \begin{cases}
            \xi, & v \in W^{\alpha,\beta}_w,\\[6pt]
            \rmp^{-1}(\cn{W^{\alpha,\beta}_w}{v}) + 2\xi, & v \notin W^{\alpha,\beta}_w, \cn{W^{\alpha,\beta}_w}{v}\ge \rmp(\rG),\\[6pt]
            +\infty, & \text{otherwise},
        \end{cases}
    \end{align*}
    we have the implication
    \begin{align}
        \label{eq: rmw-implication}
        {\rm w}^{\alpha,\beta}(w,v) < +\infty 
        &\implies
        {\rm w}^{\alpha,\beta}(w,v) \ge \D{X_w}{X_v}  \ge {\rm w}^{\alpha,\beta}(w,v) - 4\xi. 
    \end{align}
    
    \step{Small distance guarantees for ${\rm w}^{\alpha,\beta}$}
    Further, whenever $\D{X_w}{X_v} \le \rG/2$, if $v \notin W^{\alpha,\beta}_w$, then our Calibration Lemma~\ref{lem:calibration-toolkit} gives
    \begin{align*}
        \cn{W^{\alpha,\beta}_w}{v}
        &\ \ge\ 
        \rmp(\D{X_w}{X_v} + 2\eta) \ge \rmp(\rG) \quad \text{(since $\eta \ll \rG$)}.  
    \end{align*}
    Therefore,  
    \begin{align}
        \label{eq: rmw-noninfinity} 
        \D{X_w}{X_v} \le \rG/2
        &\implies
        {\rm w}^{\alpha,\beta}(w,v) < +\infty.
    \end{align}

    \step{Conclusion for ${\rm w}(w,v)$}
    Now, we assume the good event
    $$
        \Omega_{\rm dist} := \bigcap_{\{\alpha,\beta\} \in \binom{[8]}{2}} \Omega_{\rm dist}^{\alpha,\beta}
    $$
    holds, which happens with probability $1 - n^{-\omega(1)}$ by union bound.
    Then, from \eqref{eq: rmw-implication}, the definition   
    $$
        {\rm w}(w,v) := \min\{{\rm w}^{\alpha,\beta}(w,v),{\rm w}^{\alpha,\beta}(v,w)\}_{\{\alpha,\beta\} \in \binom{[8]}{2}} \quad \forall w,v\in \bfV,
    $$
    implies that 
    \begin{align}
        \label{eq: rmw-implication-final}
        {\rm w}(w,v) < +\infty 
        &\implies
        {\rm w}(w,v) \ge \D{X_w}{X_v}  \ge {\rm w}(w,v) - 4\xi.
    \end{align}
    Furthermore, whenever $w,v \in \bfV$ satisfy $\D{X_w}{X_v} \le \rG/2$, there exists $\{\alpha,\beta\} \in \binom{[8]}{2}$ such that $w,v \in \bfV_\alpha \cup \bfV_\beta$. Then, from \eqref{eq: rmw-noninfinity}, we have 
    \begin{align}
        \label{eq: rmw-noninfinity-final}
        \D{X_w}{X_v} \le \rG/2
        &\implies
        {\rm w}(w,v) < +\infty.
    \end{align}

\step{Final distance estimator ${\rm d}(w,v)$}
Finally, we define the distance estimator ${\rm d}(w,v)$ for all $w,v \in \bfV$ as the shortest-path distance with respect to the edge weights ${\rm w}(w,v)$.
Fix any $w,v \in \bfV$. Let $\gamma$ be a shortest geodesic from $X_w$ to $X_v$ in $M$ and let $t = \Big\lceil \frac{\D{X_w}{X_v}}{\rG} \Big\rceil$. Let $p_0 = X_w, p_1, \dots, p_t = X_v$ be equidistant points on $\gamma$ such that $\D{p_{i-1}}{p_i} = \D{X_w}{X_v}/t \le \rG/2$ for all $i=1,\dots,t$. By the event $\Ept{\bfV}$, there exist $w=u_0, u_1, \dots, u_t=v$ in $\bfV$ such that 
$$\D{X_{u_i}}{p_i} \le \Big(\frac{\log^2(n)}{n}\Big)^{1/d} = o(\xi)$$ 
for all $i=1,\dots,t-1$. Then, by triangle inequality,
\begin{align*}
    \D{X_w}{X_v} 
    &=
    \sum_{i=1}^t \D{p_{i-1}}{p_i} 
    \ \ge\ 
    \sum_{i=1}^t \Big( \D{X_{u_{i-1}}}{X_{u_i}} - 2\Big(\frac{\log^2(n)}{n}\Big)^{1/d} \Big)\\
    &\ge 
    \sum_{i=1}^t \Big( {\rm w}(u_{i-1},u_i) - 5\xi \Big)
    \ge 
    {\rm d}(w,v) -  \frac{{\rm diam}(M)}{\rG/2} \cdot 5\xi\,.
\end{align*}
On the other hand, let $w=z_0,z_1,\dots,z_s=v$ in $\bfV$ be the path that attains the shortest-path distance ${\rm d}(w,v)$. Then by~\eqref{eq: rmw-noninfinity-final}, 
we have 
$$
    \D{X_w}{X_v} 
    \ \le\ 
    \sum_{i=1}^s \D{X_{z_{i-1}}}{X_{z_i}} 
    \ \le\ 
    \sum_{i=1}^s  {\rm w}(z_{i-1},z_i) 
    \ =\ 
    {\rm d}(w,v)\,. 
$$
Together we conclude that under the event $\Omega_{\rm dist}$, for all $w,v \in \bfV$,
\begin{align*}
    \D{X_w}{X_v} 
    &\le {\rm d}(w,v) 
    \le \D{X_w}{X_v} + 10 \frac{{\rm diam}(M)}{\rG} \xi\,.
\end{align*}

\end{proof}
\newpage 

\bibliographystyle{alpha}
\bibliography{ref}

\newpage

\appendix
\section{Auxiliary Geometric Results}
\label{sec:geometry}
This section provides some auxiliary results from Riemannian geometry. We start with a standard proposition about the exponential map.

\begin{prop}[{\cite[Corollary~1.9]{CE08}}]    
    \label{prop:exp-ball}
Let \(\B{r}{\vec{0}} \subseteq T_pM\) be the open ball of radius \(r\) on which the exponential map \(\exp_p\) is an embedding. Then
\begin{enumerate}
  \item For every \(v\in \B{r}{\vec{0}}\), the curve
        \[
            \gamma_v:[0,1]\;\longrightarrow\; M,\qquad
            \gamma_v(t)=\exp_p(tv),
        \]
        is the \emph{unique} curve satisfying
        \[
            L[\gamma_v]\;=\;\rho\bigl(p,\exp_p(v)\bigr)\;=\;\|v\|
        \]
        such that $\gamma_v(0) = p$ and $\gamma_v(1) = \exp_p(v)$.
        In particular, if a curve \(c\) satisfies
        $L[c]= \left| c(0) \, c(1) \right|$, then---up to reparameterization---\(c\) is a smooth geodesic.

  \item If \(q\notin\exp_p \bigl(\B{r}{\vec{0}}\bigr)=\B{r}{p}\), then there exists
        \(q'\in\partial \B{r}{p}\) (the boundary of \(\B{r}{p}\)) such that
        \[
            |p \, q| = r + |q' \,  q|.
        \]
        In particular, $|p \, q| \ge r$.
\end{enumerate}
\end{prop}


\subsubsection{Proof of Lemma~\ref{lem: tri_lem}}
The lemma follows quickly from two standard corollaries of the Rauch comparison theorem (Theorem~\ref{thm: Rauch I} and Theorem~\ref{thm: conjugate} below). While the lemma is standard, we did not locate a reference that states it in the form we need, so we provide a proof for completeness. 

\begin{theor}[{\cite[Corollary~1.35 (``Corollary of Rauch I'')]{CE08}}]
\label{thm: Rauch I}
Let $M_1,M_2$ be Riemannian manifolds with ${\rm dim}(M_2) \ge {\rm dim}(M_1)$. Let $p_1 \in M_1$ and $p_2 \in M_2$. Assume $K_{M_2} \ge K_{M_1}$; i.e., for every plane section $\sigma_1$ in $M_1$ and for every plane section $\sigma_2$ in $M_2$, we have $K(\sigma_2) \ge K(\sigma_1)$. Let $r > 0$ be such that $\exp_{p_1}|_{\B{r}{\vec{0}}}$ is an embedding and $\exp_{p_2}|_{\B{r}{\vec{0}}}$ is nonsingular. Let $I: T_{p_1}M_1 \to T_{p_2}M_2$ be a linear injection preserving inner products. Then for any curve $\gamma:[0,1] \to \exp_{p_1}(\B{r}{\vec{0}})$, we have   
$$
L[\gamma] \ge L[\exp_{p_2} \circ I \circ \exp_{p_1}^{-1}(\gamma)] = L[\gamma_2(t)],
$$
where $\gamma_2(t) = \exp_{p_2} \circ I \circ \exp_{p_1}^{-1}(\gamma(t))$ and $L[\cdot]$ denotes the length of a curve in the corresponding manifold.
\end{theor}

\begin{theor} [{\cite[Theorem~11.12]{Lee18}}]
    \label{thm: conjugate}
Suppose $(M,g)$ is a Riemannian $n$-manifold whose sectional curvatures are all bounded above by a constant $\kappa > 0$. Then there is no conjugate point along any geodesic segment shorter than $\varpi_\kappa$.
\end{theor}
\


\begin{proof}[Proof of Lemma~\ref{lem: tri_lem}]
Given $r \le \rinj(M)$, the two tangent vectors  
$$
    \tv{x} = \exp_p^{-1}(x), \qquad
    \tv{y} = \exp_p^{-1}(y) \in T_pM
$$
are unique and satisfy
\[
  \|\tv{x}\| = \D{p}{x}, \, \qquad
  \|\tv{y}\| = \D{p}{y}\,.
\]
\step{Positive-curvature comparison}
Recall that \(M_\kappa^d\) is the simply-connected \(d\)-dimensional space of constant curvature \(\kappa > 0\). Pick any $p_2 \in M_\kappa^d$ and  
choose a linear isometry
\[
  I:T_pM\longrightarrow T_{p_2}M_\kappa^d
\]
and set
\[
    x_2 = \exp_{p_2}(I\tv{x}), \qquad 
    y_2 = \exp_{p_2}(I\tv{y}).
\]
This is well-defined because \(I\) is a linear isometry, and the injectivity radius of \(M_\kappa^d\) is $\rinj(M_\kappa^d) \ge  \dm{\kappa}/2 = \pi/(2 \sqrt{\kappa})$. Following our notations, we have 
\begin{align*}
    \tv{x_2} = (\tv{x_2})_p = I\tv{x}  
\qquad \text{and} \qquad
    \tv{y_2} = (\tv{y_2})_p = I\tv{y}.
\end{align*}
We have
\[
  \ang{\kappa}{p_2}{x_2}{y_2}
      = \angle(\tv{x_2},\tv{y_2})
      = \angle(\tv{x},\tv{y})
      := \theta,
\]
and 
\[
    \D{p_2}{x_2} = \|\tv{x_2}\| = \|\tv{x}\| = \D{p}{x}:=a,\qquad
    \D{p_2}{y_2} = \|\tv{y_2}\| = \|\tv{y}\| = \D{p}{y}:=b.
\]
Let \(\gamma\) be the distance-minimizing geodesic from \(x\) to \(y\) in \(M\).  

\step{Claim: $\gamma \subseteq \B{r}{p}$}
Assume for the sake of contradiction that the minimizing geodesic
\(\gamma\) from \(x\) to \(y\) leaves \(\B{r}{p}\).
Let \(z\in\gamma\) be the first point with \(\D{p}{z}=r\). Since $z$ lies on the shortest geodesic $\gamma$ connecting $x$ and $y$, we have 
\[
  \D{x}{y}=\D{x}{z}+\D{z}{y}.
\]
By the triangle inequality,
\[
  \D{x}{z}\;\ge\;\D{p}{z}-\D{p}{x}=r-a,
  \quad
  \D{z}{y}\;\ge\;\D{p}{z}-\D{p}{y}=r-b,
\]
whence
\[
  \D{x}{y}\;\ge\;(r-a)+(r-b)
            \;>\;r-\tfrac{r}{2}+r-\tfrac{r}{2}=r.
\]
But another application of the triangle inequality gives
\[
  \D{x}{y}\;\le\;d(x,p)+\D{p}{y}=a+b
            \;<\tfrac{r}{2}+\tfrac{r}{2}=r,
\]
a contradiction.  Therefore \(\gamma\) must be contained in
\(\B{r}{p}\).

\step{Applying the Rauch comparison theorem}
Because \(r\le\rinj(M)\), the restriction
$\exp_p\bigl|_{\B{r}{\vec{0}}}$
is a diffeomorphism.
Further, with $r < \dm{\kappa}/2$
, we also have $\exp_{p_2}\bigl|_{\B{r}{\vec{0}}}$ is a diffeomorphism.

Let $\gamma_2 = \exp_{p_2} \circ I \circ \exp_{p}^{-1}(\gamma)$ be the image of $\gamma$ under the map $I$. 
Applying Theorem~\ref{thm: Rauch I} yields
\[
   \D{x}{y}=L[\gamma]\;\ge\;L[\gamma_{2}]\;\ge\;\D{x_2}{y_2}\,.
\]
Because \(a+b< \dm{\kappa}\), the geodesic triangle
\(\triangle p_{2}x_{2}y_{2}\) exists in \(M^d_{\kappa}\), so
\[
   \D{x_2}{y_2}=\os^{\kappa}(\theta;a,b).
\]
Consequently,
\[
   \D{x}{y}\;\ge\;\os^{\kappa}(\theta;a,b),
\]
which is exactly the first inequality claimed in~\eqref{eq: tri_lem_os}.

\step{Negative curvature comparison}
To obtain the upper bound on \(\D{x}{y}\) we reverse the roles of
\(M\) and the model space when we apply Theorem~\ref{thm: Rauch I}.

Let \(M_{1}:=M_{-\kappa}^d\) and set
\(M_{2}:=M\).
Choose any point \(p_{1}\in M_{1}\) and a linear isometry
\(I:T_{p_1}M_{1}\to T_{p}M\).
Because \(I\) is invertible, put
\[
 x_1 = \exp_{p_1}(I^{-1}\tv{x}), \qquad 
    y_1 = \exp_{p_1}(I^{-1}\tv{y}).
\]
and set 
\[
  \theta:=\ang{\kappa_{0}}{p_{1}}{x_{1}}{y_{1}}.
\]
Similarly to the previous case, we have 
\[
    \D{p_1}{x_1} = \|\tv{x_1}\| = \|\tv{x}\| = \D{p}{x}:=a,\qquad
    \D{p_1}{y_1} = \|\tv{y_2}\| = \|\tv{y}\| = \D{p}{y}:=b.
\]
Let \(\gamma_{1}\) be the minimizing geodesic from \(x_{1}\) to
\(y_{1}\) in \(M_{1}\).
By the same argument used earlier (triangle inequalities and the
choice of \(r\)), \(\gamma_{1} \subseteq \B{r}{p_1}\).

For the model space with negative curvature, 
\(\exp_{p_1}\) is a diffeomorphism without restrictions; 
and $\exp_{p}\bigl|_{\B{r}{\vec{0}}}$ for $M$ is a diffeomorphism as well,
hence we may apply
Theorem \ref{thm: Rauch I} to get 
\[
  \D{x_1}{y_1} = L[\gamma_{1}]
      \;\ge\;L[\gamma]
      \;\ge\;\D{x}{y}.
\]
Since the comparison triangle
\(\triangle p_{1}x_{1}y_{1}\) exists (indeed, hyperbolic space
admits geodesic triangles of any perimeter), we have
\[
  \D{x_1}{y_1} = \os^{\kappa_{0}}(\theta;a,b),
\]
and therefore
\[
  \D{x}{y} \;\le\; \os^{\kappa_{0}}(\theta;a,b),
\]
which is the second inequality in~\eqref{eq: tri_lem_os}. 

\step{From opposite‐side to angle comparison}
Because \(a+b<\varpi_{\kappa}\), the opposite‐side function  
\[
  \alpha\;\longmapsto\;\os^{\kappa}(\alpha;a,b)
\]
is smooth and strictly increasing on \((0,\pi)\), with boundary values  
\[
  \os^{\kappa}(0;a,b)=|b-a|, \qquad 
  \os^{\kappa}(\pi;a,b)=a+b.
\]

From triangle inequality also have   
\(|b-a|< \D{x}{y}<a+b\), so by strict monotonicity there is a unique
\(\alpha_{1}\in(0,\pi)\) satisfying  
\[
  \os^{\kappa}(\alpha_{1};a,b)=\D{x}{y}.
\]
By construction this \(\alpha_{1}\) is the model–space angle:
\[
  \alpha_{1}=\ang{\kappa}{p}{x}{y}.
\]

Because \(\os^{\kappa}(\,\cdot\,;a,b)\) is increasing and  
\(\os^{\kappa}(\alpha_{1};a,b)= \D{x}{y}\ge
  \os^{\kappa} \bigl(\ang{}{p}{x}{y};a,b\bigr)\),
we must have  
\[
  \ang{\kappa}{p}{x}{y}
     =\alpha_{1}
     \ge\ang{}{p}{x}{y},
\]
which yields the second inequality in \eqref{eq: tri_lem_ang}.

The first inequality is obtained in exactly the same way, replacing
the positive curvature model space \(M_{\kappa}\) by the lower‐curvature
model space \(M_{-\kappa}\) and using the inequality in the opposite
direction.
\end{proof}


\subsubsection{Quantitative estimates of opposite side length with 1 small adjacent side}

\begin{proof} [Proof of Lemma \ref{lem: M-opposite-side}]

    First, we apply the triangle comparison lemma (Lemma~\ref{lem: tri_lem}) to the triangle $\triangle pxq$ in $M$.  
     Comparing simultaneously with the positively curved model space $M_\kappa$ and the negatively curved model space $M_{-\kappa}$ yields  
    \begin{align*}
        \os^{\kappa}(\ang{}{p}{x}{q}; \D{p}{x}, \D{p}{q}) \le \D{q}{x} \le \os^{-\kappa}(\ang{}{p}{x}{q}; \D{p}{x}, \D{p}{q})\,.
    \end{align*}
    The rest of the proof follows by estimating the opposite side lengths in the two model spaces, which was stated in Lemma~\ref{lem:spherical-opposite-side} and Lemma~\ref{lem:hyperbolic-opposite-side}.  
\end{proof}

\begin{cor}
    \label{cor: M-opposite-side}
With the same assumptions as in Lemma~\ref{lem: M-opposite-side}, we have 
\begin{align*}
    |a - c - b \cos(\theta)| \le 4 \frac{b^2}{a}\,. 
\end{align*}
\end{cor}
\begin{proof}
    Given the assumptions on $a$ and $b$,  
\begin{align*}
    \frac{\kappa b^2}{3} 
=  
    \frac{\kappa a^2}{3} \frac{b^2}{a^2}
\le 
    \frac{1}{16\cdot 3} \frac{1}{16} \frac{b}{a}\,. 
\end{align*}
Together with $|\cos(\theta)| \le 1$, the estimate follows immediately.  
\end{proof}

\begin{lemma}\label{lem:spherical-opposite-side}
Fix $\kappa>0$ and consider a geodesic triangle in
$M_\kappa$ whose two adjacent sides have lengths $a$ and $b$. Assume
\[
   0<a,b<\tfrac12\,\dm{\kappa}
   \qquad\text{and}\qquad
   b\;\le\;\tfrac14\,a .
\]
Let $\theta \in [0, \pi]$ be the angle between the two sides of lengths \(a\) and \(b\) at the vertex opposite to the side of length \(c = \os^{-\kappa}(\theta;a,b)\). Then
\begin{align*}
     a -c -  b\cos \theta \le  \pi \frac{b}{a} \cdot b |\cos \theta| \,.
\end{align*}

\end{lemma}

\begin{proof}
Throughout, write
\[
    c:=\os^{\kappa}(\theta;a,b),
\]
where \( \os^\kappa(\theta;a,b) \) denotes—by definition—the length of the side opposite the angle \( \theta \) in a geodesic triangle of \( M_\kappa \) whose adjacent sides have lengths \( a \) and \( b \).

\step{Normalized variables and cosine law}
Set
\[
    \tilde a:=\sqrt{\kappa}\,a,\qquad
    \tilde b:=\sqrt{\kappa}\,b,\qquad
    \tilde c:=\sqrt{\kappa}\,c.
\]
From the spherical law of cosines in \(M_\kappa\) (see~\eqref{eq: spherical-cos-law}), we have $\tilde a, \tilde b, \tilde c$ and $\theta$ satisfies spherical law of cosines on the unit sphere \(M_{1}\):
\begin{equation}\label{eq:unit-cos}
    \cos\tilde c
    \;=\;
    \cos\tilde a\,\cos\tilde b
    +\sin\tilde a\,\sin\tilde b\,\cos\theta .
\end{equation}
Define
\[
    \Delta
    :=\cos\tilde a\,(\cos\tilde b-1)
      +\sin\tilde a\,\sin\tilde b\,\cos\theta ,
\]
so that~\eqref{eq:unit-cos} is \( \cos\tilde c = \cos\tilde a+\Delta \).
By the fundamental theorem of calculus,
\begin{equation}\label{eq:lem-spherical-opposite-side00}
    \tilde c
    =\arccos(\cos\tilde a+\Delta)
    =\tilde a-\int_{\cos\tilde a}^{\cos\tilde a+\Delta}
               \frac{1}{\sqrt{1-s^{2}}}\,ds .
\end{equation}

\step{Case \(\theta\in[0,\pi/2]\)}
Here \(\cos\theta\ge0\).

If \(\Delta\le0\), then \(\cos\tilde c\le\cos\tilde a\) and hence
\(\tilde c\ge\tilde a\), whence
\(
   \tilde a-\tilde c-\tilde b\cos\theta\le0,
\)
which already implies the desired bound.
Now we assume \(\Delta>0\).
For \(s\in[\cos\tilde a,\cos\tilde c]\subset(0,1)\) we have
\[
    \frac{1}{\sqrt{1-s^{2}}}
    \le\frac{1}{\sin\tilde c}
    \le\frac{1}{\sin(\tilde a-\tilde b)},
\]
because \(\tilde c\ge\tilde a-\tilde b\) and \(\sin\) is increasing on \([0,\pi/2]\).
Since \(\cos t\) decreases on \([0,\pi/2]\),
\[
   \sin\tilde a-\sin(\tilde a-\tilde b)
   =\int_{\tilde a-\tilde b}^{\tilde a}\cos t\,dt
   \le\int_{0}^{\tilde b}\cos t\,dt
   =\sin\tilde b,
\]
so \(\sin(\tilde a-\tilde b)\ge\sin\tilde a-\sin\tilde b\).
Combining these bounds yields
\[
    \tilde c
    \ge
    \tilde a-\frac{\Delta}{\sin\tilde a-\sin\tilde b}
    \ge
    \tilde a-\frac{\sin\tilde a}{\sin\tilde a-\sin\tilde b}
        \,\sin\tilde b\cos\theta .
\]
Because \(0<\tilde b\le\tilde a/4<\pi/8\),
\(
   \sin\tilde b/\sin\tilde a\le\frac12
\),
and for \(t\in[0,1/2)\) we have \((1-t)^{-1}\le1+2t\).  Hence
\[
    \frac{\sin\tilde a}{\sin\tilde a-\sin\tilde b}
    \le
    1+2\frac{\sin\tilde b}{\sin\tilde a},
\]
so that
\[
   \tilde c
   \ge
   \tilde a-\sin\tilde b\cos\theta
        -2\frac{\sin^{2}\tilde b}{\sin\tilde a}\cos\theta
   \ge
   \tilde a-\tilde b\cos\theta
        -\pi\frac{\tilde b^{2}}{\tilde a}\cos\theta ,
\]
where we used \(\sin\tilde b\le\tilde b\) and \(\sin\tilde a\ge\frac{\pi}{2}\tilde a\).
Rewriting, we obtain
\[
   \tilde a-\tilde c-\tilde b\cos\theta
   \le
   \pi\frac{\tilde b}{\tilde a}\,\tilde b|\cos\theta|
   \quad\Longleftrightarrow\quad
   a-c-b\cos\theta
   \le
   \pi\frac{b}{a}\,b|\cos\theta|.
\]

\step{Case \(\theta\in(\pi/2,\pi]\)}
Now \(\cos\theta\le0\). Note that
\[
    \Delta
    =\underbrace{\cos\tilde a(\cos\tilde b-1)}_{\le0}
     +\underbrace{\sin\tilde a\,\sin\tilde b\,\cos\theta}_{\le0}
    \le0,
    \qquad
    |\Delta|
    \ge
    -\sin\tilde a\,\sin\tilde b\,\cos\theta.
\]
Reversing the limits in~\eqref{eq:lem-spherical-opposite-side00},
\[
    \tilde c
    =\tilde a+\int_{\cos\tilde a+\Delta}^{\cos\tilde a}
              \frac{1}{\sqrt{1-s^{2}}}\,ds
    \ge
    \tilde a+\frac{|\Delta|}{\sin\tilde a}
    \ge
    \tilde a-\sin\tilde b\cos\theta .
\]
Because \(0<\tilde b\le\pi/8<1/2\), the alternating series for \(\sin\) yields
\(
   \sin\tilde b\ge\tilde b-\tilde b^{3}/6
\).
Hence
\[
    \tilde c
    \ge
    \tilde a-\tilde b\cos\theta+\frac{\tilde b^{3}}{6}\cos\theta ,
    \qquad\Longrightarrow\qquad
    a-c-b\cos\theta
    \le
    \frac{\kappa b^{2}}{6}\,b|\cos\theta|.
\]
Finally, since \(b\le a/4\) and \(a<\dm{\kappa}/2=\pi/(2\sqrt{\kappa})\),
\[
    \frac{\kappa b^{2}}{6}
    \le
    \pi\frac{b}{a},
\]
so the claimed inequality follows.
\end{proof}


\begin{lemma}
    \label{lem:hyperbolic-opposite-side}
    Fix $\kappa >0$ and consider a geodesic triangle in
$M_{-\kappa}$ whose two adjacent sides have lengths $a$ and $b$ satisfying  
\[
   0<a,b< \frac{1}{\sqrt{\kappa}}\,.
\]
 Let $\theta \in [0, \pi]$ be the angle between the two sides of lengths \(a\) and \(b\) at the vertex opposite to the side of length \(c = \os^{-\kappa}(\theta;a,b)\). Then
\[
   a - c - b \cos(\theta) \ge -\frac{7}{6}\frac{b^2}{a} - \frac{\kappa b^2}{3} \cdot b|\cos(\theta)|\,.
\]

\end{lemma}
\begin{proof}
    The proof in the hyperbolic case is indeed similar, thanks to the hyperbolic law of cosines~\eqref{eq: hyperbolic-cos-law}:
    \begin{align*}
\cosh(\sqrt{\kappa}c) = \cosh(\sqrt{\kappa}a)\cosh(\sqrt{\kappa}b) - \sinh(\sqrt{\kappa}a)\sinh(\sqrt{\kappa}b)\cos(\theta).
    \end{align*}
    Again, let us write  
    \[
        \tilde a = \sqrt{\kappa}a, \qquad
        \tilde b = \sqrt{\kappa}b, \qquad
        \tilde c = \sqrt{\kappa}c\,.
    \] 

    Recalling that  
    $$\cosh(x) = \frac{e^x + e^{-x}}{2} \qquad \mbox{and} \qquad \sinh(x) = \frac{e^x - e^{-x}}{2}\,.$$

    Let us write 
    $$
        \Delta = \tilde c - \tilde a\,. 
    $$
    Given $\cosh(t)$ is concave up with derivative $\sinh(t)$, we have   
    \begin{align*}
        \cosh(\tilde c) \ge \cosh(\tilde a) + \sinh(\tilde a)\Delta \,,
    \end{align*}
    which implies
    \begin{align}
        \label{eq: hyperbolic-opposite-side-00}
        \Delta \le& \frac{\cosh(\tilde a)}{\sinh(\tilde a)}(\cosh(\tilde b)-1) - \sinh(\tilde b)\cos(\theta)\,.
    \end{align}

\step{Estimate $\sinh$ and $\cosh$ by comparing their Taylor series with Geometric series}
Recall the Taylor series of $\sinh$ and $\cosh$ at $0$:
\begin{align*}
    \sinh(x) & = \sum_{n=0}^{\infty} \frac{x^{2n+1}}{(2n+1)!} = x + \frac{x^3}{6} + \frac{x^5}{120} + \cdots\,,\\
    \cosh(x) & = \sum_{n=0}^{\infty} \frac{x^{2n}}{(2n)!} = 1 + \frac{x^2}{2} + \frac{x^4}{24} + \cdots\,.
\end{align*}
When $x \in [0,1]$, both series converges faster than a geometric series with a rate $1/2$, so we can use the following estimates:
\begin{align}
    \label{eq: sinh-cosh-estimate}
    \forall x \in [0,1]: \quad
        x \le   \sinh(x) \le x + \frac{x^3}{3}  \le \frac{4}{3}x 
        \quad \mbox{and} \quad 
         \cosh(x) \le 1 + \frac{x^2}{2} + \frac{x^4}{12} \le 1 + \frac{7}{12}x^2\,.
\end{align}

\step{Estimate $\Delta$}
    The first term in the right-hand side of~\eqref{eq: hyperbolic-opposite-side-00} can be estimated as follows.
    $$
        \frac{\cosh(\tilde a)}{\sinh(\tilde a)}(\cosh(\tilde b)-1) \le 
             \frac{2}{\tilde a}\frac{7}{12}\tilde b^2  = \frac{7}{6}\frac{\tilde b^2}{\tilde a}\,.
    $$
    The second term can be estimated depending on the sign of $\cos(\theta)$. 
    When $\theta \in [0, \pi/2]$, we have $\cos(\theta) \ge 0$ and thus we can rely on the simple bound $\sinh(\tilde b) \ge \tilde b$ to get 
\begin{align*}
    - \sinh(\tilde b)\cos(\theta)
    \le -\tilde b \cos(\theta) \,.
\end{align*}
This gives us the estimate
\begin{align*}
    \tilde c - \tilde a \le \frac{7}{6}\frac{\tilde b^2}{\tilde a} - \tilde b\cos(\theta) 
    \quad \Leftrightarrow \quad
    a - c -  b \cos(\theta) \ge - \frac{7}{6}\frac{b^2}{a} \,.
\end{align*}

As for $\theta \in (\pi/2, \pi]$ where $\cos(\theta) <0$, we invoke~\eqref{eq: sinh-cosh-estimate} to get
\begin{align*}
    - \sinh(\tilde b)\cos(\theta)
 \le - \tilde b\cos(\theta) - \frac{\tilde b^3}{3}\cos(\theta)\,. 
\end{align*}
This gives us the estimate
\begin{align*}
    \tilde c - \tilde a \le \frac{7}{6}\frac{\tilde b^2}{\tilde a} - \tilde b\cos(\theta)  - \frac{\tilde b^3}{3}\cos(\theta)
    \quad \Leftrightarrow \quad
    a - c - b \cos(\theta) \ge -\frac{7}{6}\frac{b^2}{a} - \frac{\kappa b^2}{3} \cdot b|\cos(\theta)| \,. 
\end{align*}
\end{proof}


\subsubsection{Spherical angle estimate via side ratio}

\begin{proof}[Proof of Lemma~\ref{lem:spherical-angle}]

We can assume $\kappa = 1$ without loss of generality. 

\step{Spherical law of cosines and Taylor expansion}
Recall the spherical law of cosines in the model space \(M_1\):
\[
   \cos c
   \;=\;
   \cos a\,\cos b
   + \sin a \sin b \cos \theta
   \]
For the right hand side, we apply $\cos(a+b) = \cos a \cos b - \sin a \sin b$ to get 
\begin{align*}
   \cos a\,\cos b
   + \sin a \sin b \cos \theta
= 
    \cos(a+b) -  \sin a \sin b (1- \cos \theta)
     &\le 
    1-  \sin a \sin b (1- \cos \theta)\,.
\end{align*}
Since $c \le \frac{1}{2}a$, the triangle inequality yields $b \ge a - c \ge \frac{1}{2}a$. 
Together with the assumption that $a \le \frac{\pi}{4}$ and using
$$
    \sin(t) \ge  \frac{1}{\sqrt{2}}\frac{4}{\pi}t   \quad t \in [0, \pi/4]\,,
$$  we obtain  
\begin{align*}
    1 -   \sin a \sin b (1- \cos \theta)
\le 
    1 -  \frac{8}{\pi^2} \cdot \frac{1}{2} a^2   (1- \cos \theta)\,.
\end{align*}

On the other hand, the standard estimate for the cosine function gives the bound on the left hand side of the spherical law of cosines: 
$$
    \cos c \ge 1 - \frac{1}{2}c^2\,.
$$
Substituting the two inequalities into the spherical law of cosines yields
\begin{align}
\nonumber
    1 - \frac{1}{2}c^2  \le 1 - \frac{1}{2} \frac{8}{\pi^2}  a^2 (1-\cos \theta) \\
\label{eq:lem-spherical-angle00}
\qquad \Rightarrow \qquad &
    c \ge \frac{\sqrt{8}}{\pi} a \sqrt{1-\cos(\theta)}\,.
\end{align}

\step{Claim: $\theta \le \pi/3$}
Substituting the assumption $c \le \frac{1}{2}a$ into the above inequality~\eqref{eq:lem-spherical-angle00} gives
\begin{align*}
    \frac{1}{2} \ge \frac{\sqrt{8}}{\pi} \sqrt{1-\cos(\theta)}\,.
\end{align*}
Assume, for contradiction, that \(\theta\ge \pi/3\).  Then,
$$
\frac{\sqrt{8}}{\pi} \sqrt{1-\cos(\theta)}
\ge \frac{2}{\pi} > \frac{1}{2}\,.  
$$
This contradicts the above inequality, hence we conclude that \(\theta < \pi/3\).

\step{Bounding $\cos(\theta)$ from below on \( [0,\pi/3] \)}
We claim that
\begin{equation}
   \label{eq:cos-lower}
   \cos x \;\le\;
   1 - \tfrac12 \Bigl(\tfrac{3x}{\pi}\Bigr)^{ 2},
   \qquad
   \forall\,x\in[0,\pi/3].
\end{equation}
Equivalently, set
\[
   F(x)
   \;:=\;
   1 - \tfrac12\Bigl(\tfrac{3x}{\pi}\Bigr)^{ 2} - \cos x,
   \qquad x\in[0,\pi/3],
\]
and show \(F(x)\ge0\) on this interval.
First, a direct computation gives \(F(0)=F(\pi/3)=0\).
Now consider its derivative 
\[
   F'(x)
   \;=\;
   -\frac{9}{\pi^{2}}\,x \;+\; \sin x.
\]
We remark that $F'$ has the same sign as 
\[
   G(x)
   \;:=\;
   \frac{F'(x)}{x}
   \;=\;
   -\frac{9}{\pi^{2}}
   \;+\;
   \frac{\sin x}{x},
   \qquad x\in(0,\pi/3].
\]
on \((0,\pi/3]\).
Because \(x\mapsto \sin x/x\) is strictly decreasing on \((0,\pi/2)\), \(G\) is strictly decreasing on \((0,\pi/3]\).
Moreover,
\[
   \lim_{x\to0^{+}} G(x) = 1 > 0,
   \qquad
   G(\pi/3)
   = -\frac{9}{\pi^{2}}
     + \frac{3\sin(\pi/3)}{\pi}
   < 0.
\]
Hence there exists a unique \(x_{0}\in(0,\pi/3)\) such that \(G(x_{0})=0\), i.e.\ \(F'(x_{0})=0\).  Because \(G\) changes sign from positive to negative at \(x_{0}\), we have
\[
   F'(x)
   \begin{cases}
      >0, & x\in[0,x_{0}),\\
      <0, & x\in(x_{0},\pi/3],
   \end{cases}
\]
so \(F\) attains its unique local (and global) maximum at \(x_{0}\).
Since \(F\) vanishes at both endpoints and is non-negative at its single interior extremum, it follows that \(F(x)\ge0\) for all \(x\in[0,\pi/3]\), proving~\eqref{eq:cos-lower}.

\step{Conclusion}
With our claim that $\theta \le \pi/3$, we can substitute the above bound on $\cos(x)$ into \eqref{eq:lem-spherical-angle00} to obtain 
\begin{align*}
    c \ge \frac{\sqrt{8}}{\pi} a \sqrt{\frac{1}{2} (\frac{3}{\pi}\theta)^2} 
\qquad \Rightarrow \qquad
    \frac{c}{a} \ge \frac{6}{\pi^2} \theta \ge \frac{1}{2} \theta\,.
\end{align*}

\end{proof}

\subsubsection{ Opposite side estimates with error in terms of maximal side length}
\begin{lemma}
    \label{lem: M-opposite-side-fixed-angle}
    \MAssump
    Suppose $p,x,y$ be points in $M$ such that 
    $$
        0 < \D{p}{x}, \D{p}{y} < \frac{1}{8} \min\{ \dm{\kappa}, \rinj(M) \}\,.
    $$
    Let $\theta = \ang{}{p}{x}{y} \in [0, \pi]$.  Then, the following holds:
    \begin{align}
        \label{eq: M-opposite-side-fixed-angle}
     \Big| \D{x}{y}^2 - \big(\D{p}{x}^2 + \D{p}{y}^2 - 2\D{p}{x}\D{p}{y}\cos(\theta) \big) \Big| \le 
     2\kappa \max\{\D{p}{x}^4, \D{p}{y}^4\}\,.
    \end{align}
\end{lemma}
\begin{proof}
    Without lose of generality, we can assume that  
    \begin{align*}
        \D{p}{x} \ge \D{p}{y}\,.
    \end{align*}
    We first invoke Lemma~\ref{lem: tri_lem} to get  
    \begin{align*}
        \os^{\kappa}(\theta; \D{p}{x}, \D{p}{y})
        \le \D{x}{y} \le \os^{-\kappa}(\theta; \D{p}{x}, \D{p}{y})\,.
    \end{align*}
    Now, we will bound the opposite sides in the two cases using the spherical and hyperbolic law of cosines, respectively. 

    \step{Spherical Cosine Law}
    Let $a = \sqrt{\kappa}\D{p}{x}$, $b = \sqrt{\kappa}\D{p}{y}$, and $c = \sqrt{\kappa}\os^{\kappa}(\theta; \D{p}{x}, \D{p}{y})$. In particular, we have 
    $
        b \le a \le 1.
    $
    The spherical law of cosines in the model space \(M_\kappa\) states
    \begin{align*}
        \cos c
        = 
        \cos a \cos b + \sin a \sin b \cos \theta\,.
    \end{align*}
    For $0 \le x \le \pi/4 < 1$, the Taylor expansion of $\cos x$ and $\sin x$ together with the alternating-series remainder give    
    \begin{align*}
            1 - x^2/2 \le \cos x   \le  1 - x^2/2  + x^4/24\,, \qquad \mbox{and} \qquad
            x - \frac{x^3}{6}\le \sin x \le x \,.
    \end{align*}
    Hence, for $a,b \in [0, \pi/8]$, we have
    \begin{align*}
        1 - c^2/2 \le& \cos c \,.
    \end{align*}
    Using the upper bound for \(\cos\) and the fact that \(b\le a\le1\), 
    \begin{align*}
        \cos a \cos b 
        \le  &
        \Big(1 - \frac{a^2}{2} + \frac{a^4}{24}\Big)\Big(1 - \frac{b^2}{2} + \frac{b^4}{24}\Big) \\
        = &
        1 - \frac{a^2}{2} - \frac{b^2}{2} + \frac{a^2b^2}{4}
        +\Big(1 - \frac{a^2}{2}\Big)\frac{b^4}{24}
        +\Big(1 - \frac{b^2}{2}\Big)\frac{a^4}{24} + \frac{a^4}{24}\frac{b^4}{24}\\
        \le &
        1 - \frac{a^2}{2} - \frac{b^2}{2} + a^4 \Big(\frac{1}{4} + 2 \cdot \frac{1}{24} + \frac{1}{(24)^2}\Big)\\
        = & 
        1 - \frac{a^2}{2} - \frac{b^2}{2} + \frac{a^4}{2}\,.
    \end{align*}
    Because \(\cos(\theta)\) may be negative, we combine the upper and lower Taylor bounds for \(\sin\):
    \begin{align*}
        \sin(a)\sin(b) \cos(\theta) 
    \le 
        ab\cos \theta +  \frac{ab^3}{12} + \frac{a^3b}{12} + \frac{a^3b^3}{36}
    \le 
        ab\cos \theta + \frac{7}{36}a^4\,,
    \end{align*}
    again using \(b\le a\).
    Collecting the above bounds and substituting into the spherical law of cosines gives
    \begin{align*}
        1 - \frac{c^2}{2} 
        \le 
        1 - \frac{a^2}{2} - \frac{b^2}{2} + \frac{a^4}{2} + ab\cos(\theta) + \frac{7}{36}a^4\,.
    \end{align*} 
    which simplifies to 
    \begin{align*}
        c^2 \ge a^2 + b^2 - 2ab\cos(\theta) - 2a^4\,.
    \end{align*}
    Rewriting in the original variables yields
    \begin{align*}
        \os^{\kappa}(\theta; \D{p}{x}, \D{p}{y})^2 
    \ge 
        \D{p}{x}^2 + \D{p}{y}^2 - 2\D{p}{x}\D{p}{y}\cos(\theta) - 2\kappa \D{p}{x}^4\,.
    \end{align*}

    \step{Hyperbolic Cosine Law}
    Similar, let $a = \sqrt{\kappa}\D{p}{x}$, $b = \sqrt{\kappa}\D{p}{y}$, and $c = \sqrt{\kappa}\os^{-\kappa}(\theta; \D{p}{x}, \D{p}{y})$. 
    The hyperbolic law of cosines in the model space \(M_{-\kappa}\) states
    \begin{align*}
        \cosh c
        = 
        \cosh a \cosh b - \sinh a \sinh b \cos \theta\,.
    \end{align*}

For every \(x\ge0\),
\[
   \cosh x
   \;=\;
   1 \;+\; \tfrac12 x^{2}
   \;+\; \sum_{k=2}^{\infty}\tfrac{x^{2k}}{(2k)!}
   \;\ge\;
   1 \;+\; \tfrac12 x^{2},
   \qquad\Longrightarrow\qquad
   \cosh c \;\ge\; 1 \;+\; \tfrac12 c^{2}.
\]
    For $0 \le x \le \pi/4 < 1$,  
\begin{align*}
    \cosh x = 1 + \frac{x^2}{2} + \sum_{k=2}^{\infty} \frac{x^{2k}}{(2k)!}  
    \le 
    1 + \frac{x^2}{2} + \frac{x^4}{24} \Big( \sum_{s=0}^\infty \Big(\frac{1}{5\cdot 6}\Big)^s \Big)
    = 
    1 + \frac{x^2}{2} + \frac{x^4}{24} \cdot \frac{30}{29} 
   \le 
    1 + \frac{x^2}{2} + \frac{x^4}{23}\,. 
\end{align*}
so that 
    \begin{align*}
        \cosh a \cosh b \le &
        \Big(1 + \frac{a^2}{2} + \frac{a^4}{23} \Big)
        \Big(1 + \frac{b^2}{2} + \frac{b^4}{23} \Big)\\
    \le &
         1 + \frac{a^2}{2} + \frac{b^2}{2} + \frac{a^4}{2}\,,
    \end{align*}
    where the inequality follows from expanding the product and using \(b \le a \le 1\).

    For the second term on the right hand side, for the same reason that $\cos(\theta)$ can be negative.
    For $0 \le x \le \pi/4 < 1$, 
    \begin{align*}
        & |\sinh(x) - x|
    = 
        \sum_{k=1}^{\infty} \frac{x^{2k+1}}{(2k+1)!}
    \le 
        \frac{x^3}{3!} 
        \sum_{s=0}^{\infty} \Big(\frac{1}{4\cdot 5}\Big)^s
    =
        \frac{x^3}{6} \frac{20}{19}
    \le 
        \frac{x^3}{5} \\
   \qquad\Longrightarrow\qquad &
            |\sinh(a)\sinh(b) \cos(\theta) - ab\cos(\theta)| 
    \le 
         \frac{ab^3+ba^3}{5} + \frac{a^3b^3}{25}         
    \le 
        \frac{6}{25}a^4\,.
    \end{align*}

    Combining the above estimates, we have
    \begin{align*}
        c^2 \le  a^2 +b^2 -ab \cos(\theta) + 2a^4\,,
    \end{align*}
    or equivalently,
    \begin{align*}
        \os^{-\kappa}(\theta; \D{p}{x}, \D{p}{y})^2
    \le
        \D{p}{x}^2 + \D{p}{y}^2 - 2\D{p}{x}\D{p}{y}\cos(\theta) + 2\kappa \D{p}{x}^4\,.
    \end{align*}

\end{proof}


\newpage


\section{Lower Bounds}
This section discusses the lower bounds for errors of approximating distances from random geometric graphs. We are going to construct two manifolds which are ``very close'' in the sense that our random geometric graph model, with high probability, cannot distinguish them; on the other hand, the two manifolds are sufficiently different, and so this gives a lower bound for errors of distance estimation.

More precisely, we are going to construct $(M_1, g_1, \mu_1)$ and $(M_2, g_2, \mu_2)$ and a coupling $\pi$ of
\[
    G_1 \sim G(n,M_1,g_1,\mu_1,\rmp,\sp=1)
    \quad \text{ and } \quad
    G_2 \sim G(n,M_2,g_2,\mu_2,\rmp,\sp=1),
\]
so that asymptotically the random geometric graph model cannot distinguish $G_1$ and $G_2$. We are going to determine two special points $X,Y$ so that $|X Y|_{g_1}$ and $|XY|_{g_2}$ are sufficiently different.

What does our construction show? The upshot of our construction is we are going to {\em disprove} the following statement: there exists an algorithm which takes the random geometric graph as input and with high probability produces the intrinsic distances of all pairs of latent points simultaneously with the uniform error bound $\lesssim n^{-6/(d+2)}$. Since we can couple the two manifolds so that the resulting random geometric graphs from the two models are the same graph with high probability, we know that the algorithm would output the same distance estimates with high probability. However, since the two manifolds are different, our calculation below shows that one of the two applications of the algorithm must have too much error.


\subsection{Riemannian manifolds}
Let $d \ge 2$ be a positive integer. Fix a sectional curvature upper bound $\kappa > 0$ so that our two manifolds $M_1$ and $M_2$ to be constructed below respect this global curvature upper bound. We take $(M_1, g_1)$ to be the flat $d$-torus realized as a point set as
\[
M_1 := \mathbb{R}^d / \mathbb{Z}^d.
\]

As a point set, we take $M_2$ to be the same set
\[
M_2 := \mathbb{R}^d / \mathbb{Z}^d,
\]
while $g_2$ is to be modified. We take a small radius $\rb > 0$ (where $\rb$ is a small number to be determined such that $\rb \to 0$ as $n \to \infty$, where $n$ is the number of latent points to sample; for the moment, we can imagine $\rb < 0.01 \cdot \min\{1,\pi/\sqrt{\kappa}\}$), and consider the ball $B_{g_1}(\vec{0},\rb) \subseteq \mathbb{R}^d/\mathbb{Z}^d$. Outside this ball, let $g_2$ be the same as $g_1$ (flat metric), and inside the ball (``small bump''), $g_2$ is to be modified in a radially-symmetric manner about $\vec{0}$ (the center of the small bump).

Let us now define the two special points. Throughout this appendix, we let $X := (-0.02,0,\ldots,0)$ and $Y := (0.02,0,\ldots,0)$, which exist on both manifolds. Note that both $X$ and $Y$ are outside of the small bump.


\subsection{Conformal change to the flat metric}
We use the usual coordinates $x^1, x^2, \ldots, x^d$ in a coordinate chart near $\vec{0}$ that contains $B(\vec{0}, 0.03)$. With these coordinates, the original flat metric $g_1 = ({\rm d}x^1)^2 + ({\rm d}x^2)^2 + \cdots + ({\rm d}x^d)^2$ has the matrix representation
\[
g_1 = I_d \in \bbR^{d \times d}.
\]
While we use the rectangular coordinates, it is helpful to consider $r := \sqrt{(x^1)^2 + \cdots + (x^d)^2}$. For a smooth function $h: [0,0.03) \to \bbR_{>0}$ to be determined, we define the modified metric $g_2$ via the radially-symmetric conformal change
\[
g_2(p) = f(p) \cdot g_1(p) = h(\|p\|) \cdot g_1(p),
\]
for every point $p$ on this coordinate chart so that for every $r \ge \rb$, we have $h(r) = 1$. This last condition guarantees that $g_1$ and $g_2$ agree everywhere outside of the small bump. In matrix form, $g_2 = f(p) \cdot I_d = h(r) \cdot I_d$. Following the usual convention, let us write $f(p) := e^{2 \varphi(p)}$. Let us also write $h(r) = e^{2 \psi(r)}$, where $\psi: [0,0.03) \to \bbR$ so that $\psi(r) = 0$ for every $r \ge \rb$. This is saying that both $f$ and $\varphi$ factor through the norm function.


\subsubsection{Upper bound on sectional curvature in $M_2$}
Consider an arbitrary point $p \in B(\vec{0}, \rb) \subseteq M_2$. Consider two arbitrary nonparallel tangent vectors $u = u^i \partial_i, v = v^j \partial_j \in T_p M_2$. The sectional curvature for the plane spanned by $u$ and $v$ is given by
\[
K(u,v) := \frac{g_2(R(u,v)v,u)}{g_2(u,u) g_2(v,v) - g_2(u,v)^2} = \frac{1}{f(p)} \cdot \frac{g_1(R(u,v)v,u)}{g_1(u,u) g_1(v,v) - g_1(u,v)^2},
\]
where $R$ is the Riemann curvature tensor on $M_2$ (with respect to $g_2$). If we take $u$ and $v$ to be an orthonormal basis of the plane (with respect to $g_1$), then the sectional curvature simplifies to
\[
K(u,v) = \frac{1}{f(p)} \cdot g_1(R(u,v)v,u) = \frac{g_2(R(u,v)v,u)}{f(p)^2}.
\]

Let us recall the following formula for Christoffel symbols (with respect to the metric $g_2$):
\begin{align*}
\Gamma_{k\ell}^i &= \frac{1}{2} g^{im} \left( \frac{\partial g_{mk}}{\partial x^\ell} + \frac{\partial g_{m\ell}}{\partial x^k} - \frac{\partial g_{k\ell}}{\partial x^m} \right) \\
&= \frac{1}{2 \cdot f(p)} \cdot \left( \frac{\partial f(p)}{\partial x^\ell} \cdot {\bf 1}_{i = k} + \frac{\partial f(p)}{\partial x^k} \cdot {\bf 1}_{i = \ell} - \frac{\partial f(p)}{\partial x^i} \cdot {\bf 1}_{k = \ell} \right).
\end{align*}
Since $f(p) = e^{2\varphi(p)}$, the chain rule gives
\[
\Gamma_{k\ell}^i = \varphi_\ell \delta_{ik} + \varphi_k \delta_{i\ell} - \varphi_i \delta_{k\ell},
\]
where $\varphi_i = \partial_i \varphi = \partial \varphi/\partial x^i$, and $\delta_{ij}$ denotes the Kronecker delta ${\bf 1}_{i = j}$.

Following Petersen~\cite[\S\S~3.1.6]{Pet16}, we have the following formula for the $(1,3)$-Riemann curvature tensor (with respect to $g_2$):
\[
g_2(R(\partial_i, \partial_j) \partial_k, \partial_\ell) = R_{ijk}^\ell = \partial_i \Gamma_{jk}^\ell - \partial_j \Gamma_{ik}^\ell + \Gamma_{jk}^s \Gamma_{is}^\ell - \Gamma_{ik}^s \Gamma_{js}^\ell.
\]
So we have that
\begin{align*}
    R_{ijk}^\ell
    &= (\varphi_{ik} - \varphi_i \varphi_k) \delta_{j\ell} + (\varphi_{j\ell} - \varphi_j \varphi_\ell)\delta_{ik} - (\varphi_{i\ell} - \varphi_i \varphi_\ell) \delta_{jk} - (\varphi_{jk} - \varphi_j \varphi_k) \delta_{i \ell} \\
    &\hphantom{=} + (\delta_{ik} \delta_{j \ell} - \delta_{jk} \delta_{i\ell}) \| \nabla \varphi \|^2.
\end{align*}
Here, the function $\varphi_{ij}$ denotes $\partial_i \varphi_j = \partial_i \partial_j \varphi$, and
\[
\| \nabla \varphi \|^2 = \sum_s \varphi_s^2.
\]
Since
\[
K(u,v) = \frac{1}{f(p)^2} \sum_{i,j,k,\ell} u^i v^j v^k u^\ell R_{ijk}^\ell,
\]
we obtain the following formula for orthonormal $u = \sum_i u^i \partial_i$ and $v = \sum_j v^j \partial_j$:
\[
-f(p)^2 \cdot K(u,v) = \sum_{i,j} (u^i u^j + v^i v^j)(\varphi_{ij} - \varphi_i \varphi_j) + \| \nabla \varphi \|^2.
\]


Now we use radial symmetry. Recall that
\[
f(p) = h(\|p\|) = h(r)
\qquad \text{ and } \qquad
\varphi(p) = \psi(\|p\|) = \psi(r).
\]
By the chain rule, we find
\[
\|(\nabla \varphi)(p)\|^2 = (\psi'(r))^2.
\]
Note that for any $i \in [d]$, we have
\[
\varphi_i(p) = \psi'(r) \cdot \frac{x^i}{r},
\]
which implies, for any $i,j \in [d]$,
\[
\varphi_i \varphi_j = \frac{\psi'(r)^2}{r^2} \cdot x^i x^j,
\]
and
\[
\varphi_{ij} = \frac{\psi''(r)}{r^2} \cdot x^i x^j - \frac{\psi'(r)}{r^3} \cdot x^i x^j + \frac{\psi'(r)}{r} \cdot \delta_{ij}.
\]

Hence, for any $i,j \in [d]$, we have
\[
\varphi_{ij} - \varphi_i \varphi_j = \zeta(r) \cdot x^i x^j + \frac{\psi'(r)}{r} \delta_{ij},
\]
where
\[
\zeta(r) := \frac{\psi''(r)}{r^2} - \frac{\psi'(r)^2}{r^2} - \frac{\psi'(r)}{r^3}.
\]
This implies the following simple formula for sectional curvatures:
\[
-f(p)^2 \cdot K(u,v) = \zeta(r) \left( \sum_i u^i x^i \right)^2 + \zeta(r) \left( \sum_i v^i x^i \right)^2 + 2 \frac{\psi'(r)}{r} + \psi'(r)^2.
\]


\begin{lemma}
Let $u^1, \ldots, u^d, v^1, \ldots, v^d$ be real numbers such that
\[
\sum_i (u^i)^2 = \sum_i (v^i)^2 = 1
\qquad \text{ and } \qquad
\sum_i u^i v^i = 0.
\]
For any real numbers $x^1, \ldots, x^d$, we have
\[
\left( \sum_i u^i x^i \right)^2 + \left( \sum_i v^i x^i \right)^2 \le \sum_i (x^i)^2.
\]
\end{lemma}
\begin{proof}
Consider the usual $\bbR^d$ with the usual Euclidean norm $\| \, \|$ and Euclidean metric $\langle \, \, , \, \rangle$. Write ${\bf x} := (x^1, \ldots, x^d)$, ${\bf u} := (u^1, \ldots, u^d)$, and ${\bf v} := (v^1, \ldots, v^d)$. Note that ${\bf u}$ and ${\bf v}$ are orthonormal vectors. Therefore, we can write
\[
{\bf x} = \langle {\bf x}, {\bf u} \rangle \cdot {\bf u} + \langle {\bf x}, {\bf v} \rangle \cdot {\bf v} + {\bf y},
\]
where the three summands are pairwise orthogonal. Thus,
\[
\|{\bf x}\|^2 
= \langle {\bf x}, {\bf u} \rangle^2 + \langle {\bf x}, {\bf v} \rangle^2 + \|{\bf y}\|^2 
\ge \langle {\bf x}, {\bf u} \rangle^2 + \langle {\bf x}, {\bf v} \rangle^2,
\]
which is the desired inequality.
\end{proof}

Combining the lemma and the formula above, we obtain the following bound:
\[
f(p)^2 \cdot |K(u,v)| \le \max\!\left\{ \left| 2 \cdot \frac{\psi'(r)}{r} + \psi'(r)^2 \right|, \left| \psi''(r) + \frac{\psi'(r)}{r} \right| \right\},
\]
which gives
\[
|K(u,v)| \le e^{-4 \psi(r)} \max\!\left\{ \left| 2 \cdot \frac{\psi'(r)}{r} + \psi'(r)^2 \right|, \left| \psi''(r) + \frac{\psi'(r)}{r} \right| \right\},
\]

By the triangle inequality, the following proposition is immediate.
\begin{prop}\label{prop: upper-bound-K-exp-m4-psi-r}
In the above notations, the magnitude of the sectional curvature $K$ is bounded above as follows.
\[
|K| \le 3 e^{-4\psi(r)} \cdot \max\!\left\{ \frac{|\psi'(r)|}{r}, \psi'(r)^2, |\psi''(r)|\right\}.
\]
\end{prop}


\subsection{Using a smooth transition function}
In this subsection, we construct an explicit $\psi$ from a smooth transition function. To begin with, let us consider the {\em standard smooth transition function} $\tau: \bbR \to \bbR$ given by
\[
\tau(x) :=
\begin{cases}
    0, & \text{ if } x < 0, \\
    \left( 1 + \exp\!\left( \frac{1}{x} - \frac{1}{1-x} \right)\right)^{-1}, & \text{ if } 0 \le x < 1, \\
    1, & \text{ if } x \ge 1.
\end{cases}
\]
The following facts follow from single-variable calculus.
\begin{lemma}\label{lem: sstf-properties}
The standard smooth transition function $\tau$ satisfies the following properties:
\begin{itemize}
    \item[(i)] $\tau: \bbR \to \bbR$ is smooth everywhere on $\bbR$.
    \item[(ii)] $\tau$ is strictly increasing on $[0,1]$.
    \item[(iii)] $\tau(1/2) = 1/2$.
    \item[(iv)] $|\tau'(x)| \le 2$, for every $x \in \bbR$.
    \item[(v)] $|\tau''(x)| \le 10$, for every $x \in \bbR$.
\end{itemize}
\end{lemma}

By linear scaling (dilation), we can define scaled versions of transition function. For any real numbers $a, b, h$ with $a < b$ and $h > 0$, we define $\tau_{a,b;h}: \bbR \to \bbR$ given by
\[
\tau_{a,b;h}(x) := 
\begin{cases}
    0, & \text{ if } x < a, \\
    h \cdot \tau\!\left( \frac{x-a}{b-a} \right), & \text{ if } a \le x < b, \\
    h, & \text{ if } x \ge b.
\end{cases}
\]
We immediately have a scaled version of Lemma~\ref{lem: sstf-properties} above.
\begin{lemma}\label{lem: tau-a-b-h-properties}
The standard smooth transition function $\tau$ satisfies the following properties:
\begin{itemize}
    \item[(i)] $\tau_{a,b;h}$ is smooth everywhere on $\bbR$.
    \item[(ii)] $\tau_{a,b;h}$ is strictly increasing on $[a,b]$.
    \item[(iii)] $\tau((a+b)/2) = h/2$.
    \item[(iv)] $|\tau'_{a,b;h}(x)| \le 2h(b-a)^{-1}$, for every $x \in \bbR$.
    \item[(v)] $|\tau''_{a,b;h}(x)| \le 10h(b-a)^{-2}$, for every $x \in \bbR$.
\end{itemize}
\end{lemma}

Now we construct the function $\psi:[0,0.03) \to \bbR$ explicitly:
\[
\psi(x) :=
\begin{cases}
    0, & \text{ if } 0 \le x < \frac{\rb}{2} \text{ or } x \ge \rb, \\
    \tau_{\frac{1}{2} \rb, \frac{2}{3} \rb; \alpha \rb^2} (x), & \text{ if } \frac{\rb}{2} \le x < \frac{2}{3} \rb, \\
    \alpha \rb^2, & \text{ if } \frac{2}{3} \rb \le x < \frac{5}{6} \rb, \\
    1 - \tau_{\frac{2}{3}\rb, \frac{5}{6} \rb; \alpha \rb^2}(x), & \text{ if } \frac{5}{6} \rb \le x < \rb,
\end{cases}
\]
where we may take $\alpha \in \bbR$ to be any real number which satisfies
\[
0 < \alpha < 2^{-12} \kappa.
\]


It is immediate from our use of smooth transition functions that $\psi:[0,0.03) \to \bbR$ is smooth. Furthermore, since $\psi$ is smooth and is identically zero on an open neighborhood of $0$, it follows from the chain rule that $\varphi(p) := \psi(\|p\|)$ is smooth everywhere in $B(\vec{0},0.03)$.

\begin{prop}
The manifold $(M_2, g_2)$, where the metric $g_2$ is obtained by the conformal change as described above, has the sectional curvature bound
\[
- \kappa < K < \kappa,
\]
everywhere on the manifold.
\end{prop}
\begin{proof}
This is a consequence of Proposition~\ref{prop: upper-bound-K-exp-m4-psi-r} and Lemma~\ref{lem: tau-a-b-h-properties}. Since $\psi$ is nonnegative, Proposition~\ref{prop: upper-bound-K-exp-m4-psi-r} implies that
\[
|K| \le 3 \cdot \max\!\left\{ \frac{|\psi'(r)|}{r}, \psi'(r)^2, |\psi''(r)|\right\},
\]
whence it suffices to show
\[
\max\!\left\{ \frac{|\psi'(r)|}{r}, \psi'(r)^2, |\psi''(r)|\right\} < \frac{\kappa}{3}.
\]
Since $\psi'$ vanishes on
\[
\Big[ 0, \frac{1}{2} \rb \Big) \cup \Big( \frac{2}{3} \rb, \frac{5}{6} \rb \Big) \cup \Big( \rb, 0.03 \Big),
\]
we only need to consider the case
\[
r \in \left[ \frac{1}{2} \rb, \frac{2}{3} \rb \right] \cup \left[ \frac{5}{6} \rb, \rb \right].
\]
Now we use Lemma~\ref{lem: tau-a-b-h-properties}. First,
\[
\frac{|\psi'(r)|}{r} \le 2 \alpha \rb^2 \left( \frac{1}{6} \rb \right)^{-1} \left( \frac{1}{2} \rb \right)^{-1} = 24 \alpha < \frac{\kappa}{3}.
\]
Second,
\[
\psi'(r)^2 \le \left( 2 \alpha \rb^2 \right)^2 \left( \frac{1}{6} \rb \right)^{-2} = 144 \alpha^2 \rb^2 < \frac{\kappa}{3},
\]
where in the last step we have used our assumptions that
\[
\alpha < 2^{-12} \kappa
\quad \text{ and } \quad
\rb < \frac{\pi}{100 \sqrt{\kappa}}.
\]
Third,
\[
|\psi''(r)| \le 10 (\alpha \rb^2) \left( \frac{1}{6} \rb \right)^{-2} = 360 \alpha < \frac{\kappa}{3}.
\]
We have finished the proof.
\end{proof}
We have successfully shown that our construction provides a conformal change to the metric in a way that respects the desired sectional curvature upper bound.


\subsection{Lower bound on the injectivity radius \texorpdfstring{$\rinj(M_2)$}{rinj(M_2)}}
In the original flat torus $(M_1,g_1)$, it is clear that the injectivity radius is $\rinj(M_1) = 1/2$, which is exactly half the length of the shortest geodesic loop on $M_1$. Above we have constructed $(M_2, g_2)$ with a proven sectional curvature two-sided bounds $-\kappa < K_{M_2} < \kappa$. Potentially the injectivity radius $\rinj(M_2)$ might be different from that of $M_1$ after we added the bump. In this subsection, we provide a lower bound on $\rinj(M_2)$.

A crucial feature of our lower bound is that it is independent of $n$. Hence, the injectivity radius, even as the size of the small bump vanishes asymptotically as $n \to \infty$, remains bounded away from zero.

Our main tool for this part is the following theorem of Cheeger.


\begin{theor}[{\cite{Che70}}; also see {\cite[Thm.~5.8]{CE08}}]\label{thm: Cheeger-dDVH}
For any $d \in \mathbb{Z}_{\ge 1}$, $D > 0$, $V > 0$, and $H \in \bbR$, there exists a constant $c_{\tref{thm: Cheeger-dDVH}} = c(d,D,V,H) > 0$, depending only on $d, D, V, H$, such that for every $d$-dimensional Riemannian manifold $M$ with $\diam(M) \le D$, $\Vol(M) \ge V$, and $K_M \ge H$, every smooth closed geodesic on $M$ has length greater than $c_{\tref{thm: Cheeger-dDVH}}$.    
\end{theor}


We claim that $\diam(M_2) \le \frac{\sqrt{d}}{2} + 0.03$. To see this, note that $\diam(M_1) = \frac{\sqrt{d}}{2}$. For any two points $p, q \in M_2$, consider the path $\gamma$ from $p$ to $q$ such that $L(\gamma;M_1) = \| p \, q \|_{M_1}$. Note that the difference between $L(\gamma;M_1)$ and $L(\gamma;M_2)$ can be only when $\gamma$ passes through the small bump, and this adds at most
\[
2 \rb \cdot e^{\alpha \rb^2} \le 0.02 \cdot e^{2^{-12} \cdot \pi^2 \cdot 10^{-4}} < 0.03.
\]
Thus, $L(\gamma;M_2) \le L(\gamma;M_1) + 0.03 \le \frac{\sqrt{d}}{2} + 0.03$.

We claim that $\Vol(M_2) \ge 1$. To see this, note that outside the small bump, the volume form is just the Euclidean volume form. Inside the small bump, the volume form satisfies
\[
\sqrt{|\deg(g_2)|}\Big|_p = e^{d \cdot \varphi(p)} \ge 1.
\]
Hence, $\Vol(M_2) \ge \Vol(M_1) = 1$.

From the previous subsection, we have shown $K_{M_2} \ge -\kappa$. Now we can show the following result.


\begin{prop}\label{prop: rinj-M2}
Let $c_{\tref{thm: Cheeger-dDVH}} = c(d, \sqrt{d}/2 + 0.03, 1, -\kappa) > 0$ be the constant from Theorem~\ref{thm: Cheeger-dDVH}. Define
\[
c_{\tref{prop: rinj-M2}} = c_{\tref{prop: rinj-M2}}(d,\kappa) := \min\!\left\{ \frac{\pi}{\sqrt{\kappa}}, \frac{1}{2} c_{\tref{thm: Cheeger-dDVH}} \right\} > 0.
\]
Then $\rinj(M_2) \ge c_{\tref{prop: rinj-M2}}$. Note that $c_{\tref{prop: rinj-M2}}$ depends only on $d$ and $\kappa$.
\end{prop}
\begin{proof}
This follows immediately from combining our discussions above with Klingenberg's lemma (see e.g.~\cite[Thm.~6.4.7]{Pet16}).
\end{proof}


\subsection{Difference in distances}
Recall that we consider two special points $X = (-0.02, 0, \ldots, 0)$ and $Y = (+0.02, 0, \ldots, 0)$. In $M_1$, the distance $|XY|_{M_1}$ is $0.04$.

To ensure that $|XY|_{M_2}$ is sufficiently larger than $|XY|_{M_1}$, let us from now on take 
\[
\alpha := 3 \cdot 2^{-14} \kappa.
\]
This subsection proves the following lower bound:

\begin{prop}\label{prop: XY-M2-rb-3}
We have
\[
|XY|_{M_2} \ge 0.04 + 2^{-14} \cdot \min\{\kappa, 1\} \cdot \rb^3.
\]
\end{prop}
\begin{proof}
Take an arbitrary curve $\gamma$ from $X$ to $Y$. We consider two cases.

\medskip

\noindent \underline{Case 1.} Suppose that $\gamma$ never visits $B(\vec{0}, \rb/2)$, the ``flat'' part inside of the bump. In this case, note that
\[
L(\gamma;M_2) \ge L(\gamma;M_1),
\]
and elementary geometry gives
\[
L(\gamma;M_1) \ge 2 \cdot \sqrt{(0.02)^2 - (\rb/2)^2} + 25 \rb^2 \ge 0.04 + 12 \rb^2,
\]
which yields the desired lower bound.

\medskip

\noindent \underline{Case 2.} Suppose that $\gamma$ visits $B(\vec{0}, \rb/2)$. Then from our construction, $\gamma$ must spend at least $2 \cdot (\rb/6) = \rb/3$ $g_1$-distance over where $\psi$ is $\alpha \rb^2$. Hence,
\[
L(\gamma;M_2) \ge L(\gamma;M_1) + \frac{\rb}{3} \cdot \left( e^{\alpha \rb^2} - 1 \right) \ge 0.04 + \frac{1}{3} \alpha \rb^3 = 2^{-14} \kappa \rb^3.
\]
This finishes the proof of the proposition.
\end{proof}

Using a similar idea, we can obtain an {\em upper bound} of the difference between $|pq|_{M_1}$ and $|pq|_{M_2}$ for arbitrary points $p, q$ on the torus.

\begin{prop}\label{prop: pq-M2-pq-M1-diff-upper}
For any two points $p, q \in \mathbb{R}^d/\mathbb{Z}^d$, we have
\[
0 \le |pq|_{M_2} - |pq|_{M_1} < 2^{-11} \kappa \rb^3.
\]
\end{prop}
\begin{proof}
Take a minimizing $g_1$-geodesic $\gamma$ from $p$ to $q$, so that $|pq|_{M_1} = L(\gamma;M_1)$. Note that the length of $\gamma \cap B(\vec{0}, \rb)$ is at most $2 \rb$. Hence,
\[
L(\gamma;M_2) - L(\gamma;M_1) \le 2 \rb \left( e^{\alpha \rb^2} - 1 \right) \le \frac{5}{2} \alpha \rb^3.
\]
Now because $\alpha = 3 \cdot 2^{-14} \kappa$, we have
\[
L(\gamma;M_2) - L(\gamma;M_1) < 2^{-11} \kappa \rb^3.
\]
Since $L(\gamma;M_2) \ge |pq|_{M_2}$, we have finished the proof.
\end{proof}


\subsection{Coupling}
Let us now construct the measures $\mu_1$ and $\mu_2$ on $(M_1, g_1)$ and $(M_2, g_2)$, respectively. Take a small mass $m > 0$ (where $m = \frac{{\rm polylog}(n)}{n}$ is much smaller than $1$ to be determined; for the moment, imagine $m < 0.01$). Define $\mu$ to be the sum of $(1-2m)$ times the uniform measure on $\mathbb{R}^d/\mathbb{Z}^d$ (induced from the standard Euclidean measure) and two point masses at $X = (-0.02,0,\ldots,0)$ and $Y = (0.02,0,\ldots,0)$, each of mass $m$. We let both $\mu_1$ and $\mu_2$ to be the same probability measure $\mu$. We can do this because $(M_1, g_1)$ and $(M_2, g_2)$, as point sets, are the same set, even though the Riemannian metrics are different.

For the distance-probability function, let us simply take $\rp(x) := e^{-x}$, for all $x \ge 0$.

Now we construct the coupling between
\[
    G_1 \sim G(n,M_1,g_1,\mu_1,\rmp,\sp=1)
    \quad \text{ and } \quad
    G_2 \sim G(n,M_2,g_2,\mu_2,\rmp,\sp=1).
\]
First sample $n$ i.i.d.\ points $X_1, X_2, \ldots, X_n$ from $\mu$ on $\mathbb{R}^d/\mathbb{Z}^d$. Let $\mathcal{U}_{i,j}$ be i.i.d.\ ${\rm uniform}(0,1)$ random variables. Using the {\em same} $\{X_i\}$ and $\{\mathcal{U}_{i,j}\}$, we define the edge sets:
\[
E(G_1) := \left\{ \{i,j\} \in \binom{[n]}{2} \, : \, \mathcal{U}_{i,j} < \rp(|X_i X_j|_{M_1}) \right\}
\]
and
\[
E(G_2) := \left\{ \{i,j\} \in \binom{[n]}{2} \, : \, \mathcal{U}_{i,j} < \rp(|X_i X_j|_{M_2}) \right\}.
\]
Note that $E(G_2) \subseteq E(G_1)$. It is easy to see that the marginal distributions of $G_1$ and $G_2$ indeed follow the correct random geometric graph models.

Our analysis so far in the previous subsections has not explicitly imposed how $\rb$ depends on $n$. This is because so far we only consider the {\em geometry}. Now that we would like the coupling $\pi$ to be such that $G_1 = G_2$, let us now impose
\[
\rb < C(d,\kappa) \cdot n^{-\frac{2}{d+2}},
\]
where $C(d,\kappa) > 0$ is some constant depending on $d$ and $\kappa$, to be described below.

To help with notations in our analysis below, let us describe sampling from $\mu$ as follows. First, let us define i.i.d.\ $W_1, W_2, \ldots, W_n \sim {\rm uniform}(0,1)$. For each $i$, if $W_i \le 1 - 2m$, then take $X_i \sim {\rm uniform}(\mathbb{R}^d/\mathbb{Z}^d)$ according to the induced Lebesgue measure. If $1 - 2m < W_i \le 1 - m$, take $X_i := X$. If $1 - m < W_i$, take $X_i := Y$.

We would like to analyze the probability that there is a minimizing geodesic segment from $X_i$ to $X_j$ which passes through the bump $B(\vec{0}, \rb)$. To do so, we need some geometry. For each point $p \in M_1$, let us define the following set (which we call ``geodesic cone'') of $p$
\[
\mathcal{GC}(p) := \bigcup \left\{ q \in \gamma \, : \, \gamma \text{ is a minimizing } g_1 \text{-geodesic from } p \text{ to } q \text{ such that } \gamma \cap B(\vec{0},\rb) \neq \varnothing \right\}.
\]
The main property of this set that we use is that for any point $q \not\in \mathcal{GC}(p)$, any minimizing $g_1$-geodesic from $p$ to $q$ does not pass through the bump $B(\vec{0},\rb)$. Note that since our conformal transformation satisfies $g_2 \ge g_1$ everywhere on the set $\mathbb{R}^d / \mathbb{Z}^d$, we find that any minimizing $g_2$-geodesic from $p$ and $q$ does not pass through the bump either. In particular, if $q \not\in \mathcal{GC}(p)$, we obtain
\[
|pq|_{M_1} = |pq|_{M_2}.
\]

The following lemma can be proved by standard geometry.
\begin{lemma}\label{lem: vol-GC-p}
For every positive real number $R \in \Big(\rb,\frac{\sqrt{d}}{2}\Big]$, and for every point $p \in \mathbb{R}^d / \mathbb{Z}^d$, if the $g_1$-distance from $p$ to $\vec{0}$ is $R$, then
\[
\Vol(\mathcal{GC}(p)) = C_{\tref{lem: vol-GC-p}} \cdot \left( \frac{\rb}{R} \right)^{d-1} \sqrt{ 1- \left( \frac{\rb}{R} \right)^2},
\]
where $C_{\tref{lem: vol-GC-p}} = C_{\tref{lem: vol-GC-p}}(d) > 0$ is a constant depending only on $d$.
\end{lemma}
In particular, the lemma shows that
\[
\Vol(\mathcal{GC}(p)) \le C_{\tref{lem: vol-GC-p}} \cdot \left( \frac{\rb}{R} \right)^{d-1}.
\]

Now consider a point $Y \sim {\rm uniform}(\mathbb{R}^d/\mathbb{Z}^d)$ according to the standard Lebesgue measure induced on the flat torus. Let $\mathfrak{R}$ denote the $g_1$-geodesic distance from $Y$ to $\vec{0}$. It is not hard to see that $\mathfrak{R}$ is absolutely continuous, so let us write $f_{\mathfrak{R}}$ for the unique continuous density of $\mathfrak{R}$. Then the following estimate on $f_{\mathfrak{R}}$ of $\mathfrak{R}$ follows from the volume formula for spheres:
\begin{lemma}\label{lem: sphere-pdf}
For $0 \le R \le \frac{\sqrt{d}}{2}$, we have
\[
f_{\mathfrak{R}}(R) \le C_{\tref{lem: sphere-pdf}} \cdot R^{d-1},
\]
where $C_{\tref{lem: sphere-pdf}} = C_{\tref{lem: sphere-pdf}}(d) > 0$ is a constant depending only on $d$.
\end{lemma}

Now we are ready to bound the probability that $G_1 \neq G_2$. Using the union bound, we find
\begin{equation}\label{eq: union-G1-neq-G2}
\bbP\{G_1 \neq G_2\} \le \sum_{\{i,j\} \in \binom{[n]}{2}} \bbP\!\left\{ \rp(|X_iX_j|_{M_2}) < \mathcal{U}_{i,j} \le \rp(|X_i X_j|_{M_1}) \right\}.
\end{equation}
Let us write $\mathcal{T}(X_i,X_j)$ to denote the event where there is a minimizing $g_1$-geodesic from $X_i$ to $X_j$. From our above discussion, if $\mathcal{T}(X_i,X_j)$ does {\em not} happen, then $|X_iX_j|_{M_1} = |X_iX_j|_{M_2}$, and hence
\[
\bbP\!\left\{ \rp(|X_iX_j|_{M_2}) < \mathcal{U}_{i,j} \le \rp(|X_i X_j|_{M_1}) \, | \, \mathcal{T}(X_i,X_j)^c \right\} = 0.
\]

Let us break into cases according to $W_i, W_j$.

First, if both $W_i$ and $W_j$ are in $[0,1-2m]$, then $X_i$ and $X_j$ are simply i.i.d.\ uniform on the torus. Thus,
\begin{align*}
    &\bbP\!\left\{ \rp(|X_iX_j|_{M_2}) < \mathcal{U}_{i,j} \le \rp(|X_i X_j|_{M_1}) \, | \, W_i, W_j \in [0,1-2m] \right\} \\
    &= \bbP\!\left\{ \rp(|X_iX_j|_{M_2}) < \mathcal{U}_{i,j} \le \rp(|X_i X_j|_{M_1}) \, | \, W_i, W_j \in [0,1-2m], \mathcal{T}(X_i,X_j) \right\} \\
    &\hphantom{=} \times \bbP\!\left\{ \mathcal{T}(X_i,X_j)  \, | \, W_i, W_j \in [0,1-2m] \right\}.
\end{align*}
Now we bound the two factors on the right-hand side. For the first factor, Proposition~\ref{prop: pq-M2-pq-M1-diff-upper} shows that
\[
\bbP\!\left\{ \rp(|X_iX_j|_{M_2}) < \mathcal{U}_{i,j} \le \rp(|X_i X_j|_{M_1}) \, | \, W_i, W_j \in [0,1-2m], \mathcal{T}(X_i,X_j) \right\}
\le 2^{-11} \kappa \rb^3,
\]
where we have used the upper Lipschitz constant of $\rp$. For the second factor, let us define $p$ and $q$ to be i.i.d.\ uniform on the torus to avoid confusion. We have
\[
\bbP\!\left\{ \mathcal{T}(X_i,X_j)  \, | \, W_i, W_j \in [0,1-2m] \right\} = \bbP\!\left\{ \mathcal{T}(p,q)\right\}.
\]

Lemmas~\ref{lem: vol-GC-p} and~\ref{lem: sphere-pdf} yield
\begin{align*}
    \bbP\{\mathcal{T}(p,q)\}
    &= \int_0^{\sqrt{d}/2} \bbP\!\left\{\mathcal{T}(p,q) \, \Big| \, d_{g_1}(p,\vec{0}) = R\right\} \cdot f_{\mathfrak{R}}(R) \, {\rm d} R \\
    &\le \int_{\rb}^{\sqrt{d}/2} C_{\tref{lem: vol-GC-p}} \cdot \left( \frac{\rb}{R} \right)^{d-1} \cdot C_{\tref{lem: sphere-pdf}} \cdot R^{d-1} \, {\rm d} R + \int_0^{\rb} C_{\tref{lem: sphere-pdf}} \cdot R^{d-1} \, {\rm d} R \\
    &\le C(d) \cdot \rb^{d-1},
\end{align*}
where $C(d) > 0$ is a constant depending only on $d$.

This shows that
\[
\bbP\!\left\{ \rp(|X_iX_j|_{M_2}) < \mathcal{U}_{i,j} \le \rp(|X_i X_j|_{M_1}) \, | \, W_i, W_j \in [0,1-2m] \right\}
\le C(d, \kappa) \cdot \rb^{d+2},
\]
where $C(d,\kappa) > 0$ is a constant depending only on $d$ and $\kappa$.

Second, if one of $W_i$ and $W_j$ is in $[0,1-2m]$ and one is in $(1-2m,1]$, then this is saying one of $X_i$ and $X_j$ is $X$ or $Y$. For instance, we consider
\begin{align*}
&\bbP\!\left\{ \rp(|X_iX_j|_{M_2}) < \mathcal{U}_{i,j} \le \rp(|X_i X_j|_{M_1}) \, | \, W_i \in [0,1-2m], W_j \in (1-2m,1] \right\} \\
&= \bbP\!\left\{ \rp(|X_iX_j|_{M_2}) < \mathcal{U}_{i,j} \le \rp(|X_i X_j|_{M_1}) \, | \, W_i \in [0,1-2m], W_j \in (1-2m,1], \mathcal{T}(X_i,X_j) \right\} \\
&\hphantom{=} \times \bbP\!\left\{ \mathcal{T}(X_i,X_j) \, | \, W_i \in [0,1-2m], W_j \in (1-2m,1] \right\}.
\end{align*}
Once again, the first factor can be bounded by Proposition~\ref{prop: pq-M2-pq-M1-diff-upper}. For the second factor, note that $X_j$ is either $X$ or $Y$, which is exactly $0.02$ away from $\vec{0}$. Lemma~\ref{lem: vol-GC-p} yields
\[
\bbP\!\left\{ \mathcal{T}(X_i,X_j) \, | \, W_i \in [0,1-2m], W_j \in (1-2m,1] \right\}
\le C(d,\kappa) \cdot \rb^{d+2},
\]
where $C(d,\kappa) > 0$ is a constant depending only on $d$ and $\kappa$.

Third, if both $W_i$ and $W_j$ are in $(1-2m,1]$, then we simply use Proposition~\ref{prop: pq-M2-pq-M1-diff-upper} to obtain
\[
\bbP\!\left\{ \rp(|X_iX_j|_{M_2}) < \mathcal{U}_{i,j} \le \rp(|X_i X_j|_{M_1}) \, | \, W_i, W_j \in (1-2m,1] \right\} \le 2^{-11} \kappa \rb^3.
\]

Combining all cases, we have shown:
\begin{align*}
&\bbP\!\left\{ \rp(|X_iX_j|_{M_2}) < \mathcal{U}_{i,j} \le \rp(|X_i X_j|_{M_1}) \right\} \\
&\le C(d,\kappa) \left( (1-2m)^2 \rb^{d+2} + 4m \cdot \rb^{d+2} + 4m^2 \cdot \rb^3 \right),
\end{align*}
where $C(d,\kappa) > 0$ depends only on $d$ and $\kappa$. We record this as a proposition below.
\begin{prop}\label{prop: C-d-kappa-rb-d2-rb-d2-rb-3}
For any $i,j \in [n]$, we have
\[
\bbP\!\left\{ \rp(|X_iX_j|_{M_2}) < \mathcal{U}_{i,j} \le \rp(|X_i X_j|_{M_1}) \right\} \le C_{\tref{prop: C-d-kappa-rb-d2-rb-d2-rb-3}}(d, \kappa) \cdot \left( \rb^{d+2} + m^2 \cdot \rb^3 \right),
\]
where $C_{\tref{prop: C-d-kappa-rb-d2-rb-d2-rb-3}}(d, \kappa) > 0$ depends only on $d$ and $\kappa$.
\end{prop}

Combining this proposition with the union bound \eqref{eq: union-G1-neq-G2}, we conclude that if $\rb = o(n^{-2/(d+2)})$ and $m = \frac{{\rm polylog}(n)}{n}$, then $\bbP\{G_1 \neq G_2\} = o(1)$ as $n \to \infty$.


\subsection{Final step}
We have shown above that $G_1 = G_2$ with high probability. It is not hard to see that our construction satisfies the regularity assumptions on $\rp$, $\mu$, and $M_i$. It remains to show that both $X$ and $Y$ appear as latent points with high probability.

This last step is simple. We have
\begin{align*}
    \bbP\!\left\{ \forall i \in [n], \, W_i \le 1 - 2m \right\} = (1-2m)^n \le e^{-2mn} \to 0,
\end{align*}
since $m = \frac{{\rm polylog}(n)}{n}$.

Recall from Proposition~\ref{prop: XY-M2-rb-3} that the two points $X$ and $Y$ satisfy
\[
|XY|_{M_2} - |XY|_{M_1} \ge C(d,\kappa) \cdot \rb^3,
\]
where $C(d,\kappa) > 0$ depends on only $d$ and $\kappa$.

Since we have shown above that we can pick any $\rb$ for which
\[
\rb = o\!\left( n^{-\frac{2}{d+2}} \right),
\]
we have finished proving Theorem~\ref{theor: lower-bound-informal}.

\newpage


\end{document}